\documentclass[oneside,english]{amsart}
\usepackage{cmap}  % should be before 'fontenc'
\usepackage[T2A]{fontenc}
\usepackage[utf8]{inputenc}
\usepackage[english]{babel}
\usepackage[usenames]{color}
\usepackage{amsmath,amssymb,amsthm}
\usepackage{arcs}
\usepackage{epsfig}
\usepackage{filecontents}
\usepackage{graphicx}
\usepackage[pdftex,colorlinks=true,linkcolor=blue,urlcolor=red,unicode=true,hyperfootnotes=false,bookmarksnumbered]{hyperref}
\usepackage{ifthen}
\usepackage{import}
\usepackage{indentfirst}
\usepackage{mathtools}
\usepackage{cancel}
\usepackage{xcolor}
\usepackage{thmtools}
\usepackage{thm-restate}
\usepackage{enumitem}

\makeatletter
\numberwithin{equation}{section}
\numberwithin{figure}{section}
\theoremstyle{plain}
\newtheorem{thm}{\protect\theoremname}
\theoremstyle{plain}
\newtheorem{lem}[thm]{\protect\lemmaname}

\newtheorem{conj}[thm]{Conjecture}
\newtheorem{claim}[thm]{Claim}

\newtheorem{cor}[thm]{Corollary}
\newtheorem{obs}[thm]{Observation}
\newtheorem{definition}[thm]{Definition}
\newtheorem*{theorem*}{Theorem}

\makeatother

\newcommand{\ff}{\mathcal{F}}
\newcommand{\s}{\mathcal{S}}

\newcommand{\aaa}{\mathcal{A}}
\newcommand{\bb}{\mathcal{B}}

\newcommand{\m}{\mathcal}

\newcommand{\mb }{\mathbb}

\usepackage{babel}
\providecommand{\lemmaname}{Lemma}
\providecommand{\theoremname}{Theorem}
\providecommand{\theoremname}{Claim}
\providecommand{\theoremname}{Proposition}

\title{A Complete Intersection Theorem for Large Permutation Groups}

\author{Nathan Keller}
\thanks{Department of Mathematics, Bar-Ilan University. \texttt{Nathan.Keller@biu.ac.il}. Supported by the Israel Science Foundation (grants no.~2669/21 and 2456/25) and by the Binational US-Israel Science Foundation (grant no.~2024120).}
\author{Andrey Kupavskii}
\thanks{Moscow Institute of Physics and Technology, Saint-Petersburg State University, Innopolis University. \texttt{kupavskii@ya.ru}.} 
\author{Noam Lifshitz}
\thanks{Einstein Institute of Mathematics, Hebrew University. \texttt{noamlifshitz@gmail.com}. Supported by the European Research Council (StG no.~101163794), by the Israel Science Foundation (grant no.~1980/22), and by the Binational US-Israel Science Foundation (grant no.~2024120).} \author{Ohad Sheinfeld}
\thanks{Einstein Institute of Mathematics, Hebrew University. \texttt{oshenfeld@gmail.com}}
\begin{document}

\maketitle

\begin{abstract}
    A family of permutations is called $t$-intersecting if any two permutations in the family agree on at least $t$ elements. We prove that there exists $n_0 \in \mathbb{N}$ such that for any $n>n_0$ and any $1 \leq t \leq n$, the maximum size of a $t$-intersecting family in $S_n$ is obtained by one of the families $\mathcal{F}_{n,t,r}=\{\sigma \in S_n: |\mathrm{Fixed}(\sigma) \cap \{1,2,\ldots,t+2r\}|\geq t+r\}$, where $\mathrm{Fixed}(\sigma)$ is the set of fixed points of $\sigma$. This proves an analogue of the classical Complete Intersection Theorem for large permutation groups, thus providing an essentially complete solution of the Deza-Frankl intersection problem for permutations (1977).
%    Our proof relies on a refinement of the spread approximation method, introduced by Kupavskii and Zakharov (2022). 
\end{abstract}

\section{Introduction}

\subsection{Background} 
A family $F$ of subsets of $[n]=\{1,2,\ldots,n\}$ is \emph{$t$-intersecting} if for any $A,B \in F$, we have $|A \cap B| \geq t$. For $t=1$, such families are simply called `intersecting'. In 1961, Erd\H{o}s, Ko, and Rado~\cite{EKR61} 
proved that for $n \geq n_0(k,t)$, the maximum size of a $t$-intersecting family of $k$-element subsets of $[n]$ is ${n-t}\choose{k-t}$, and asked, what is the minimal number $n_0(k,t)$ for which this upper bound holds. For $t=1$, they provided a complete solution, proving that 
the maximum size is ${n-1}\choose{k-1}$ for all $k<\frac{n}{2}$.  
This result was highly influential, and by now grew into a subfield of extremal combinatorics, studying collections of objects with forbidden intersections (see the survey~\cite{FT16}). 

Naturally, one of the central problems in this field is determining the maximum size of a $t$-intersecting family $F \subset \mathcal{U}$, for various `universes' $\mathcal{U}$. This problem was studied, e.g., for vector spaces~\cite{FranklG85},  graphs~\cite{EFF12}, set partitions~\cite{Kupavskii26EJC,MM05}, simplicial complexes~\cite{Borg09,Kupavskii23hereditary},
linear maps~\cite{Linear-maps}, etc. Arguably, the two most thoroughly studied `universes' are the original setting where $\mathcal{U}$ consists of all $k$-subsets of $[n]$, and the setting of  \emph{$t$-intersecting families of permutations}, i.e., families $F \subset S_n$ such that for any $\sigma,\tau \in F$, there exist $i_1,\ldots,i_t$ with $\sigma(i_j)=\tau(i_j)$ for $j=1,\ldots,t$.
 
For $k$-subsets of $[n]$, Frankl~\cite{F78} determined the minimal $n_0(k,t)$ for which the maximum size is ${n-t}\choose{k-t}$ for all $t \geq 15$, and then Wilson~\cite{W84} determined it for all $t$: they showed that  $n_0(k,t)=(k-t+1)(t+1)$. Furthermore, for all $n>n_0(k,t)$, the maximum size is obtained only for \emph{$t$-umvirates}, i.e., families of the form $\{S \subset \binom{[n]}{k}: T \subset S\}$, for some $|T|=t$. 
For the general question of determining the maximum possible size of a $t$-intersecting family for any triple $(n,k,t)$, Frankl~\cite{F78} introduced the families $\mathcal{F}_{n,k,t,r}=\{S \subset \binom{[n]}{k}:|S\cap [t+2r]|\geq t+r\}$ and conjectured that the maximum is always obtained by one of them. This was shown for a wide range of parameters by Frankl and F\"{u}redi~\cite{FranklF91} and then for all $(n,k,t)$ by Ahlswede and Khachatrian~\cite{AK97} in the so-called \emph{Complete Intersection Theorem}. This theorem has become one of the best-known results in extremal combinatorics and played an important role in major applications to computer science, in particular in the seminal result of Dinur and Safra~\cite{DS05} on hardness of approximation.

For permutations, Deza and Frankl~\cite{DF77} proved in 1977 an analogue of the Erd\H{o}s-Ko-Rado theorem (for $t=1$) for permutations: they showed that the maximum size of an intersecting family $F \subset S_n$ is $(n-1)!$, which is obtained for the \emph{dictatorship} families $(S_n)_{i \to j} = \{\sigma \in S_n: \sigma(i)=j\}$. They also showed that for $t=2,3$ and $n$ of a certain form, the maximum size is $(n-2)!$ and $(n-3)!$, respectively.\footnote{The reason behind the latter results is the existence of certain multiply transitive permutation subgroups for such values of $n$, over which it is possible to do an averaging argument. This is one of the reasons why the algebraic combinatorics community got interested in the problem.}
At the other end of the spectrum, they showed that for $t \geq 3$ and $n>n_0(t)$, the maximum size of an $(n-t)$-intersecting $F \subset S_n$ is obtained by the family 
$\{\sigma \in S_n: |\mathrm{Fixed}(\sigma)| \geq n-\frac{t}{2}\}$ for even $t$ and by the family $\{\sigma \in S_n: |\mathrm{Fixed}(\sigma) \cap [n-1]|\geq (n-1)-\frac{t-1}{2}\}$ for odd $t$, where $\mathrm{Fixed}(\sigma)$ is the set of fixed points of $\sigma$. Deza and Frankl conjectured that an analogue of the Erd\H{o}s-Ko-Rado theorem holds for permutations -- namely, that for all $n>n_0(t)$, the maximum size is $(n-t)!$. Cameron~\cite{Cam88} conjectured in 1986 that the $t$-umvirate families 
\[
(S_n)_{i_1 \to j_1,\ldots,i_t \to j_t} = \{\sigma \in S_n: \forall 1 \leq \ell \leq t, \sigma(i_\ell)=j_\ell\}
\]
are the only maximum-sized families.

Despite a large volume of research, the progress on the $t$-intersection problem for permutations was significantly slower than on the corresponding problem for $k$-subsets of $[n]$. 
The first advance was obtained in 2003 by Cameron and Ku~\cite{CK03} and (independently) by Larose and Malvenuto~\cite{LM04}, who showed that for $t=1$, the dictatorships $(S_n)_{i \to j} = \{\sigma \in S_n: \sigma(i)=j\}$ are the unique maximizers. A major breakthrough was obtained by Ellis, Friedgut, and Pilpel~\cite{EFP11} in a paper published at the Journal of the AMS in 2011. They used representation-theoretic techniques to prove that the Deza-Frankl conjecture indeed holds for all $n \geq n_0(t)$.  
Roughly at the same time, Ellis~\cite{Ell11} proved that Cameron's conjecture holds for all $n>n_0(t)$, and that furthermore, any $t$-intersecting family of size $|F|>(1-1/e+o(1))(n-t)!$ must be contained in a $t$-umvirate. Such \emph{stability} results for families of $k$-subsets of $[n]$ were obtained in~\cite{AhlswedeK96,EKL19,EKL16,Fra87,Fri08,HM67,KZ22}. 
Based upon their breakthrough, Ellis et al.~\cite{EFP11}~raised a daring conjecture:
\begin{conj}[Ellis, Friedgut and Pilpel]
\label{conj:EFP}
For any $n \in \mathbb{N}$ and any $1 \leq t \leq n$, the maximum size of a $t$-intersecting family in $S_n$ is obtained by one of the families 
\[
\mathcal{F}_{n,t,r}=\{\sigma \in S_n: |\mathrm{Fixed}(\sigma) \cap \{1,2,\ldots,t+2r\}|\geq t+r\},
\]
where $\mathrm{Fixed}(\sigma)$ is the set of fixed points of $\sigma$. In particular, for all $t<n/2$, the maximum is obtained by the $t$-umvirate $\m F_{n,t,0}=\{\sigma \in S_n: \mathrm{Fixed}(\sigma) \supset [t]\}$.

Furthermore, all maximum-sized $t$-intersecting families are the double translates of the families $\mathcal{F}_{n,t,r}$, i.e., have the form $\tau_1 \m F \tau_2$, where $\m F= \m F_{n,t,r}$ for some $r$ and $\tau_1,\tau_2 \in S_n$. 
\end{conj}
In recent years, numerous papers used various techniques to prove Conjecture~\ref{conj:EFP} in special cases. In particular, 
Ellis and Lifshitz~\cite{DN22} showed that the conjecture holds for all $t=O(\frac{\log n}{\log\log n})$ using the discrete Fourier-analytic \emph{junta method}~\cite{NN21}, along with a representation-theoretic argument. Kupavskii and Zakharov~\cite{KZ22} proved that the conjecture holds for all $t=O(\frac{n}{(\log(n))^2})$ using their \emph{spread approximations method}. 
Meagher and Razafimahatratra~\cite{MeagherR21} proved the conjecture for $t=2$ and all $n$ using spectral methods, and Chase, Dafni, Filmus and Lindzey~\cite{CDFL26} proved the uniqueness of the extremal families in the same setting using the theory of \emph{complexity measures of Boolean functions}~\cite{BuhrmanW02,DafniFLLV21}.
Keller, Lifshitz, Minzer and Sheinfeld~\cite{keller2024t} 
proved that the conjecture holds for all $t\leq cn$ for a universal constant $c$, using the analytic method of \emph{hypercontractivity for global functions}~\cite{KLLM21}. All the above results apply only in the range $t<\frac{n}{2}$ where the maximum size is obtained by the $t$-umvirates. 
In an earlier version of this paper, posted on arXiv in 2024~\cite{kupavskii2024almost}, the second author proved that the conjecture holds for all $t \leq n-O(\frac{n \log\log n}{\log n})$ using some of the methods of this paper, thus obtaining the first result for the general problem, where all families $\mathcal{F}_{n,t,r}$ are candidates for being extremal, since the 1977 work of Deza and Frankl  that covered the case of $t$ being very close to $n$. Very recently, Saengrungkongka~\cite{saengrungkongka2026extremal} used enhancements of some of these methods to push the bound further to $t \leq n-n^{5/7+\epsilon}$.

%For the general problem, where all families $\mathcal{F}_{n,t,r}$ are candidates for being extremal, no results were obtained since the 1977 work of Deza and Frankl that covered the case of $t$ being very close to $n$. %\nathan{We have to mention the new result which solved the problem for $n<n-n^{5/7}$ and say that it's based on Andrey's draft from 2024.}
%For comparison, in the $t$-intersection problem for sets, the range where the $t$-umvirate $\{F \subset \binom{[n]}{k}:T \subset S\}$ was resolved already by Wilson~\cite{Wilson84}  in 1984, while the general problem was solved by Ahlswede and Khachatrian~\cite{AK97} only 13 years later, in a much more complex way.
  
%Unlike all previous works in this direction (e.g.,~\cite{CK03,EFP11,DN22,DF77,LM04}), the work of Kupavskii and Zakharov does not use any specific properties of $S_n$ (and in particular, its representation theory). Instead, they prove that the  conjecture holds for the general setting of $t$-intersecting families in $\mathcal{U}$, where the `universe' $\mathcal{U}$ is a `pseudorandom' sub-family of ${N}\choose{M}$, and show that $S_n$ can be viewed as a `pseudorandom' sub-family of ${n^2}\choose{n}$. The technique of Kupavskii and Zakharov is called the \emph{spread approximations method}. Similarly to \cite{DN22}, Kupavskii and Zakharov reduce their analysis to the study of global sets. Their method relies on the recent breakthrough results of Alweiss, Lovett, Wu, and Zhang~\cite{ALWZ21} on the sunflower problem in place of the Fourier analytic approach of \cite{DN22}. 

\subsection{Our results} In this paper we obtain an analogue of the Complete Intersection Theorem for all permutation groups $S_n$, $n > n_0$, thus proving Conjecture~\ref{conj:EFP} except for finitely many cases (i.e., permutation groups on at most $n_0$ elements) and essentially solving the Deza-Frankl problem. Specifically, we prove the following theorem, which also generalizes the `stability' result of Ellis~\cite{Ell11}. 
\begin{thm}\label{Thm:Main}
    There exists $n_0 \in \mathbb{N}$ such that for all $n >n_0$ and all $1 \leq t \leq n$, the maximum size of a $t$-intersecting family $\m F  \subset S_n$ is obtained by one of the families 
    \[
    \mathcal{F}_{n,t,r}=\{\sigma \in S_n: |\mathrm{Fixed}(\sigma) \cap \{1,2,\ldots,t+2r\}|\geq t+r\},
    \]
    where $\mathrm{Fixed}(\sigma)$ is the set of fixed points of $\sigma$.
%Furthermore, all maximum-sized $t$-intersecting families are the double translates of the families $\mathcal{F}_{n,t,r}$, i.e., have the form $\tau_1 F \tau_2$, where $F \in \mathcal{F}_{n,t,r}$ and $\tau_1,\tau_2 \in S_n$. 

Furthermore, for any $\eta>0$, there exists $n_1 \in \mathbb{N}$ such that for all $n>n_1$ and all $1 \leq t \leq n$, if $\m G \subset S_n$ is a $t$-intersecting family and $|\m G| \geq (1-\frac{1}{e}+\eta)\max_r |\m F_{n,t,r}|$, then $\m G \subset \tau_1 \mathcal{G'} \tau_2$, where $\m G'$ is a family of the form $\mathcal{F}_{n,t,r}$ for some $0 \leq r \leq \lfloor \frac{n-t}{2} \rfloor$ and $\tau_1,\tau_2 \in S_n$.  
\end{thm}
As was shown by Ellis~\cite{Ell11} for $t=1$, the constant $1-1/e$ in the stability statement is optimal. Indeed, let $t<n/2$ (the range in which $\max_r |\m F_{n,t,r}|=|\m F_{n,t,0}|=(n-t)!$), and let $\sigma \in S_n$ be the transposition that interchanges $1$ and $n$ and leaves all other points fixed. The family $\m G = \{\tau \in (S_n)_{1 \to 1,\ldots,t \to t}:|\tau \cap \sigma|\geq t\} \cup \{\sigma\}$ is $t$-intersecting, satisfies $|\m G| = (1-1/e-o(1))(n-t)!=(1-1/e-o(1))\max_r |\m F_{n,t,r}|$, and is not contained in a double translate of a family of the form $\m F_{n,t,r}$. 

\subsection{Our techniques} 
Unlike some of the papers that obtained advances on the Deza-Frankl problem using representation-theoretic and Fourier-theoretic techniques (e.g.,~\cite{Ell11,EFP11,DN22,keller2024t,MeagherR21}), our proof is purely combinatorial. Our main technical tool is the \emph{spread approximation method} developed by Kupavskii and Zakharov~\cite{KZ22} around the breakthrough work of Alweiss, Lovett, Wu and Zhang~\cite{ALWZ21} on the Erd\H{o}s-Rado sunflower conjecture. Importantly, unlike the representation-theoretic methods used before, this method does not use the special structure of permutations, and instead it works for a large class of settings that can be viewed as `quasirandom' subsets of a simple-structured `universe', such as $S_n$ viewed as a subset of $[n]^n$ or $\binom{[n]\times [n]}{n}$. In this work
we develop the spread approximation method into a general framework for solving $t$-intersection problems, which will hopefully enable obtaining complete $t$-intersection theorems over various universes.  

%Our main technical tool is a new \emph{iterative spread approximation lemma} that will hopefully be useful for other problems in extremal combinatorics. 
%\nathan{Assuming Theorem 4 is indeed a technical novelty, I propose to write here the definition of spreadness, an informal statement of Theorem 4, and a short mention of the relation to previous works on spread approximations and to the work on the sunflower conjecture. If Theorem 4 cannot be considered a technical novelty, probably it will be better to skip this subsection altogether.}

%\nathan{We will have to think whether we want to elaborate more on previous techniques, and whether we want to mention the fact that we probably know a proof of the theorem via hypercontractivity, but we prefer the current proof since it is simpler and shorter. Currently, I think the answer to both questions is `no'.}\andrey{I don't know. we may or may not elaborate on the `history', if we want to explain the replacement of my original submission on arXiv. A complication is that it's now used in a couple of places, including the MIT person paper.}

\subsection{Organization of the paper}
In Section~\ref{sec:overview} we introduce notations that will be used throughout the paper and present an overview of the proof of Theorem~\ref{Thm:Main}. For the sake of convenience, the structure of the following sections is described at the end of Section~\ref{sec:overview}, at the point where the high-level structure of the proof will be clear.

\section{Overview of the Proof}
\label{sec:overview}

\subsection{Notations and basic notions}

%\nathan{Here we want to include all notations that appear throughout the paper. Section-specific notations will be introduced in the relevant sections.}

Throughout the paper, families of sets or permutations are denoted by calligraphic letters (e.g., $\m F, \m G$), sets are denoted by standard letters, and permutations are denoted by Greek letters. $\m G$ denotes a family of permutations (unless explicitly stated otherwise), and $\m F$ always denotes a family of sets. We identify each permutation $\sigma \in S_n$ with the set of pairs $\{(i,\sigma(i))\}_{i \in [n]}$ and say that each pair $(i,\sigma(i))$ is contained in $\sigma$. A double translate of a family $\m G \subset S_n$ is a family of the form $\tau_1 \m G \tau_2=\{\tau_1 \sigma \tau_2: \sigma \in \m G\}$, for some  $\tau_1,\tau_2 \in S_n$. 

For a permutation $\sigma \in S_n$, the set of fixed points is denoted by $\mathrm{Fixed}(\sigma)$ and the set of moving points\footnote{The set of moving points is usually called the \emph{support} of the permutation. We prefer to use a different term, in order to avoid confusion with the support of \emph{partial permutations}, see below.} is denoted by $\mathrm{Moving}(\sigma) = [n] \setminus \mathrm{Fixed}(\sigma)$. The power set of a set $S$ (i.e., the family of all subsets of $S$) is denoted by $\m P(S)$. The family of $k$-element subsets (resp., $(\le k)$-element subsets) of a set $S$ is denoted by $\binom{S}{k}$ (resp., $\binom{S}{\le k}$). Disjoint union is denoted by $\sqcup$. The notation $x:=y$ means that we define $x$ to be equal to $y$.  

\medskip \noindent \emph{Restrictions.} A central notion used multiple times in the paper is restriction of a family, which considers all sets in the family that contain a certain set or one of several possible sets. For set families $\ff, \m S$ and a set $X$, we use the notations
\begin{equation*}
\ff(X):=\{F\setminus X:X\subset F, F\in \ff\}, \qquad
\ff[X]:= \{F: X\subset F, F\in \ff\}, 
\end{equation*}
%\ff(X, Y)&:= \{F\setminus X: F\cap Y = X\},\\
%\ff[X, Y]&:= \{F: F\cap Y = X\},\\
\begin{equation*}
    \ff(\s):=  \bigcup_{A\in \s}\ff(A), \qquad
\ff[\s]:=  \bigcup_{A\in \s}\ff[A].
\end{equation*}

\medskip \noindent \emph{Spreadness.} 
The main technique used in the paper is \emph{spread approximations}. For a real number $r \geq 1$, we say that a family $\ff$ of sets is {\it $r$-spread} if for each non-empty set $X$ we have $|\ff(X)|< r^{-|X|}|\ff|$. For $s \in \mathbb{N}$, we say that $\m F$ is {\it $(r,s)$-spread} if for any disjoint sets $S,T$ with $|S|=s$ and $T \neq \emptyset$, we have $|\ff(S \cup T)|< r^{-|T|}|\ff(S)|$. %\ohad{Generally I think it is more convinient to have spreadness with $\le$ instead of $<$, this way we dont have to say nonempty set every time. so here for example we have to say nonempty $T$}. \andrey{Maybe, indeed. Some inequalities work out nicer with this, though.} \nathan{I left the strict inequality so that formulas will look nicer, and added the assumption $T \neq \emptyset}$.
We say that $\m F$ is {\it weakly $(r,s)$-spread} if the above holds for $S_0$ such that $|\ff(S_0)|=\max_{S:|S|=s}|\ff(S)|$ and any $T$. In other words, denoting $a_m=\max_{S:|S|=m}|\ff(S)|$ for every $m \in \mathbb{N}$, $\m F$ is {\it weakly $(r,s)$-spread} if $a_{s+t}< r^{-t}a_s$ for all $t>0$. 
Intuitively, spreadness is a pseudo-randomness property saying that the family is not concentrated on sets that contain certain elements.

Essentially, the theorem underlying the spread approximation technique asserts that if a family of sets in some ambient space is `locally quasirandom' (i.e., sufficiently spread), then it is `globally quasirandom' -- which means that we expect to see sets from the family inside a typical subset of the ambient space (see Theorem~\ref{thmtao} below). An important yet simple observation is 
that any sufficiently large family contains a spread subfamily (see Observations~\ref{obs13} and~\ref{obs34} below). In the spread approximation technique, the set family in study is `approximated' by a well-structured spread family.

\subsection{A different point of view on the problem}

A central ingredient of our proof is approaching the $t$-intersection problem for permutations from a different point of view. All previous works on the problem, except for the original work of Deza and Frankl~\cite{DF77},
%, worked in the range $t<\frac{n}{2}$ in which the extremal example is a double translate of the family $\m F_{n,t,0}=\{\sigma \in S_n:\mathrm{Fixed}(\sigma) \supset [t]\}$, and 
focused on fixed points of the permutations in the family. On the contrary, we focus on the set of the \emph{moving} points, %$D_{\sigma}=[n] \setminus \mathrm{Fixed}(\sigma)$, 
$\mathrm{Moving}(\sigma)=[n] \setminus \mathrm{Fixed}(\sigma)$, like Deza and Frankl did in the analysis of $(n-t)$-intersecting families for $n \geq n_0(t)$. 

As a result, we replace $t$ with $n-t$ and assume throughout the paper that $\m G \subset S_n$ is an $(n-t)$-intersecting family. The smaller the value of $t$, the easier it will be for us to prove the assertion of the theorem. 

\subsection{Three simplification procedures}

In the course of the proof, we repeatedly apply to the families we study the following three simplification procedures.

\medskip \noindent \emph{Non-standard representation of permutations by sets.} We introduce a new representation of permutations by sets. For $\sigma \in S_n$, we define $D_{\sigma}=\mathrm{Moving}(\sigma)$, 
%to be the set of points moved by $\sigma$ as above, 
and define $E_{\sigma}$ to be the set of pairs $\{(i,\sigma(i)):i \in D_\sigma\}$. We use the transformation 
\begin{equation}\label{Eq:Perms_to_Sets}
    \sigma \mapsto F_\sigma := D_\sigma \sqcup E_\sigma 
\end{equation}
to replace a family of permutations by a family of subsets of $\mb X = [n]\sqcup\{(i,j):i,j \in [n], i\neq j\}$. That is, we represent each permutation by the set of its moving points and the information on where they move. It turns out that this transformation conveys in a good way the intersection properties of families of permutations. Specifically, if $\m G \subset S_n$, $\m F$ is a family of subsets of $\mb X$ obtained from $\m G$ by the transformation~\eqref{Eq:Perms_to_Sets} and $\m F_i =\{A \in \m F: |A|=2i\}$, then $\m G$ is $(n-t)$-intersecting if and only if for any $i,j$, the families $\m F_i, \m F_j$ are cross $(i+j-t)$ intersecting (meaning that for any $A \in \m F_i, B \in \m F_j$, we have $|A \cap B| \geq i+j-t$). 
%Conversely, if $\m F \subset \m X$ satisfies the cross-intersection properties then it is the image under~\eqref{Eq:Perms_to_Sets} of an $(n-t)$-intersecting family $\m G \subset S_n$.   

One advantage of this transformation is that it allows us to operate almost entirely with sets rather than with permutations, which turns out to be much more convenient. Another advantage is that since we may assume w.l.o.g.~that the original $(n-t)$-intersecting family $\m G$ contains the identity permutation and hence any $\sigma \in \m G$ has at most $t$ moving points, all sets in the set family $\m F$ we obtain are `small' -- i.e., have at most $2t$ elements. On the other hand, the transformation makes spreadness calculations more complicated, as restrictions $\m F[X]$ usually lose the natural correspondence between the $D$ part (representing the moving points) and the $E$ part (representing the information on where they move). The presentation of the transformation, along with comparisons between sizes of various families of the type $\m F[X]$, which are extensively used throughout the proof, span Section~\ref{sec:PermutationstoSets} of the paper. 
%\nathan{If this representation is novel (is it?), it will be good to mention this.}\andrey{It is novel, as far as I know.}

\medskip \noindent \emph{Iterative spread approximation.} This simplification procedure allows (under certain conditions, of course) approximating a $t$-intersecting family 
%\ohad{of subsets of $X$?}\nathan{Yes, thanks!} 
$\m F \subset \m P(\mb X)$ by a $t$-intersecting family $\m S$ of sets of size only slightly larger than $t$, such that almost every set in $\m F$ contains some set in $\m S$ and for each $A \in \m S$, there exists $\m F_A \subset \m F$ for which $\m F_A(A)$ is `spread'. That is, we approximate $\m F$ by a union of spread parts based on `small-sized' restrictions, while maintaining the $t$-intersection property. The relatively small size of the elements of $\m S$ and the spreadness of the parts of $\m F$ that contain them turn out to be very helpful in our proof. On the other hand, in the elements of $\m S$, the natural correspondence between the $D$ and the $E$ parts is lost, and its recovery requires some technical effort. 
%are very helpful in finding a sub-structure that is common to most elements of $\m F$, which will be a main step in our proof.   

The approximation is obtained by a complex iterative procedure which spans Section~\ref{sec:IterativeSpreadApproximation} and is probably the `heaviest' part of the paper. A similar procedure in the special case of families of $k$-subsets of $[n]$ was introduced by Frankl and Kupavskii~\cite{FranklK25} who used it to study the Hajnal-Rothschild problem; we develop the method in a general setting.  

\medskip \noindent \emph{Peeling simplification.} In this iterative process, we start with a family $\m S_0=\m S$, and at the $\ell$'th step, we find a set $Y_\ell$ such that some part of $\m S_\ell(Y_\ell)$ is spread and contains more than one element, and replace $\m S_{\ell}$ with $\m S_{\ell+1}=\m S_\ell \setminus \m S_{\ell}[Y_\ell] \cup \{Y_{\ell}\}$. That is, we replace all sets in $\m S_{\ell}$ that contain $Y_\ell$ with the single set $Y_{\ell}$. The process ends when no more $Y_{\ell}$ can be found -- i.e., when the resulting family does not contain spread parts. We show that the process preserves the intersection properties of the original family and that each element in the original family contains an element of the resulting family, and obtain a bound on the numbers of sets of each size the resulting family contains (the bound stems from the fact that it does not contain spread parts). Thus, the process allows us to replace a set family with a `kernel' whose size we can bound efficiently. On the other hand, as in the case of the iterative spread approximation, this simplification procedure loses the natural correspondence between the $D$ and the $E$ parts of the sets in the family, and its recovery requires some technical effort. This process and its properties are described in Section~\ref{sec:peeling_off} of the paper. A similar procedure was introduced by Kupavskii and Zakharov~\cite{KZ22} and refined by Kupavskii~\cite{Kupavskii26EJC}.
%\nathan{If a similar process was already used in previous papers (I think it was), we should write where.}\andrey{I left a comment in Section 5.}

\subsection{Proof overview}

Most of the proof is performed within the realm of sets. Given $n,t$, we consider a maximum-size $(n-t)$-intersecting family $\m G \subset S_n$, assume w.l.o.g.~that it contains the identity permutation, and apply to it the transformation~\eqref{Eq:Perms_to_Sets} to obtain a family $\m F$ of subsets of size $\leq 2t$ of $\mb X = [n]\cup\{(i,j):i,j \in [n], i\neq j\}$. Then, the proof proceeds in two steps:
\begin{enumerate}
    \item We find a common sub-structure in most of the elements of $\m F$. Informally, we  show that there exist an integer $m\le t$ and a set $F = D \cup M$, where $|D \cup M|=2m-t$, $D \subset [n]$ and $M=\{(i_1,\sigma_0(i_1)),\ldots,(i_s,\sigma_0(i_s))\}$ for $i_1,\ldots,i_s \in D$ and some $\sigma_0 \in S_n$, such that for any $m' \leq 2t$, a large portion of the sets in $\ff_{m'}$ intersect $F$ in at least $m+m'-t$ elements.

    \item We show that there exists a `correction' $\sigma_1 \m G$ of the original family $\m G$ such that for $\m F'$ constructed from $\sigma_1 \m G$ by~\eqref{Eq:Perms_to_Sets}, a similar intersection property holds with respect to a set $F' = D' \cup M'$, where $D'=D \setminus \{i_1,\ldots,i_s\}$ and $M'=\emptyset$. Specifically, we show that for a large portion of $\tau \in \sigma_1 \m G$, we have $|\mathrm{Moving}(\tau) \setminus D'|\leq (t-|D'|)/2$. As the family $\{\sigma \in S_n: |\mathrm{Moving}(\sigma) \setminus D'|\leq (t-|D'|)/2\}$ is a double translate of the family $\m F_{n,n-t,(t-|D'|)/2}$, this implies that a large portion of $\m G$ is contained in a double translate of $\m F_{n,n-t,(t-|D'|)/2}$. Then, an argument using the $(n-t)$-intersection property of $\m G$ and its maximality among the $(n-t)$-intersecting families allows deducing that the entire $\m G$ is included in a double translate of $\m F_{n,n-t,(t-|D'|)/2}$. Moreover, a similar argument holds whenever $\m G$ satisfies $|\m G|>(1-\frac{1}{e}-o(1))\max_r |\m F_{n,n-t,r}|$. 
    %is larger than the size of the second-largest among the families $\m F_{n,n-t,r}$. \nathan{Is the stability statement true?} \andrey{At the very least, in the stability statement we should expect Hilton--Milner-type examples as well. (Unless the parameters are `very favourable' for us) I think that this is possible to prove, but may require some extra effort. Probably, proving that most of the family should be contained in one of these examples is easier.}
\end{enumerate}
The hard step of the proof is Step~(1), which is performed differently for three ranges of $t$ -- Small $t$: $t < n^{\epsilon}$, Medium $t$: $n^{\epsilon}\leq t<n^{\frac{1}{2}+\epsilon}$, and Large $t$: $n^{\frac{1}{2}+\epsilon} \leq t<n-n^{1-\frac{\epsilon}{8}}$. There is no need to cover the range $t>n-n^{1-\frac{\epsilon}{8}}$ (i.e., the range where the intersection size is less than $n^{1-\frac{\epsilon}{8}}$), as the assertion of Theorem~\ref{Thm:Main} in this range was already proved in~\cite{keller2024t}. We do not specify the value of $\epsilon$ and do not try to optimize it; one may take $\epsilon=0.01$ throughout the proof.
\begin{enumerate}[label=(\alph*)]
\item \emph{Small $t$: $t<n^{\epsilon}$.} We consider the decomposition of $\m F$ into the families $\{\m F_{i}\}_{i \leq t}$, which are $(2i-t)$-intersecting as was written above. By the pigeonhole principle, there exists $m \leq t$ such that $|\m F_m|\geq n^{-\epsilon}|\m F|$. We apply the \emph{peeling simplification} to $\m F_m$ to obtain a $(2m-t)$-intersecting family $\m S_m$ such that each $F \in \m F_m$ contains some $S \in \m S_m$ and we can bound the number of sets of each size in $\m S_m$. Using these bounds, we deduce that most of the sets in $\m F_m$ contain a set in $\m S'=\{S \in \m S_m:|S|=2m-t \mbox{ or } 2m-t+1\}$. 

Then, we use the $(2m-t)$-intersection property of $\m S_m$ to derive that there is a single element in $\m S'$ such that most of the sets in $\m F_m$ contain it. Moreover, this element is of the form $F=D \cup M$ where $D \subset [n]$, $M=\{(i_1,\sigma_0(i_1)),(i_2,\sigma_0(i_2)),\ldots\}$ for distinct elements $i_j \in D$, and $|M|=|D|$ if $t$ is even or $|M|=|D|-1$ if $t$ is odd. While this statement is not exactly of the form described in Step~(1) above, it allows completing the proof by a variant of Step~(2) above. Namely, we deduce that there exists a correction $\sigma_1 \m G$ such that an $n^{-\epsilon}$-fraction of it is included in a double translate of either $\m F_{n,n-t,\frac{t}{2}}$ (if $t$ is even) or $\m F_{n,n-t,\frac{t-1}{2}}$ (if $t$ is odd), and then we derive that the entire family $\m G$ is included in such a double translate.      

\item \emph{Medium $t$: $n^{\epsilon}\leq t<n^{\frac{1}{2}+\epsilon}$.} The large size of the sets in $\m F$ does not allow applying the strategy of the `small $t$' case, and hence, we first apply to each subfamily $\m F_i$ ($0 \leq i \leq t$) the \emph{iterative spread approximation} simplification procedure, in order to reduce the sizes of the sets. We obtain families $\m S_i$ of sets of size at most $(2i-t)+k_i$, where $k_i \leq c\log n$ if $t\leq n^{1/2}$ and $k_i \leq n^{\epsilon}$ if $t>n^{1/2}$, such that $|\m F_i \setminus \m F_i[\m S_i]|=o(|\m F_i|)$, each $\m S_i, \m S_j$ are cross $(i+j-t)$ intersecting, and for each $S \in \m S_i$, there exists a subfamily $\m F_S \subset \m F_i$ such that $\m F_S[S]$ is spread. Using the cross-intersection property, we deduce that almost all elements of $\m F$ are included in one of the families $\m F_{p+j}[\m S_{p+j}]$, for some $\lfloor \frac{t}{2} \rfloor \leq p \leq t$ and $0 \leq j \leq k_p$.   

At this stage, we apply the \emph{peeling simplification} to each $\m S_i$ ($p \leq i \leq p+k_p)$ to obtain a $(2i-t)$-intersecting family $\m W_i$ such that most sets $F \in \m F$ contain some $W \in \m W_i$ (for some $i$) and we can bound the number of sets of each size in $\m W_i$. However, these bounds are not sufficient in this case, and hence a more complex argument is needed for obtaining improved bounds. 

By the pigeonhole principle, there exist $p \leq m \leq p+k_p$ and $0 \leq k \leq k_p$ such that
$|\m F_{m}[\m W_m^{(2m-t+k)}]| \geq (2k_p)^{-2}|\m F|$, where $W_m^{(\ell)}$ denotes the subfamily $\{W \in \m W_m:|W|=\ell\}$. If this happens for $k=0$, then $\m W_m^{(2m-t+k)}$ consists of a single set $F$, and hence, $|\m F_{m}[F]| \geq (2k_p)^{-2}|\m F|$, which makes this case easy to handle. Otherwise, we take $A,B \in \m W_m^{(2m-t+k)}$ such that $|A \cap B|=2m-t$ and set $\m Y=\m W_m^{(2m-t+k)} \cap \binom{A \cup B}{2m-t+k}$. We show that most of $\m F_m[\m W_m^{(2m-t+k)}]$ is contained in $\m F_m[\m Y]$, by a complex counting argument that exploits the fact that $\m W_m$ does not contain spread parts and intersection properties. This means that an $\Omega(k_p^{-2})$-portion of the sets in $\m F$ have intersection of at least $2m-t+k$ with $A \cup B$. Then, we use the maximality of $\m G$ among the $(n-t)$-intersecting families to deduce that $|\m Y| \geq ck_p^{-2}\binom{2m-t+2k}{2m-t+k}$, meaning that $\m Y$ is `dense' within $\binom{A \cup B}{2m-t+k}$. Finally, we define $F=A \cup B$ and use intersection arguments and the large density of $\m Y$ to show that $\m F$ satisfies the properties described in Step~(1) above with respect to the sub-structure $F$.  

\item \emph{Large $t$: $n^{\frac{1}{2}+\epsilon} \leq t<n-n^{1-\frac{\epsilon}{8}}$.} In this case, the sets in $\m F$ are so large that even the more complex strategy of the `medium $t$' case cannot be applied. Instead, we begin with a preparation size-reduction step. We consider the original family $\m G \subset S_n$, view its elements as $n$-element subsets of $[n] \times [n]$ -- i.e., sets of $n$ pairs of the form $(i,\sigma(i))$, and apply to it the \emph{iterative spread approximation} simplification procedure. We obtain an $(n-t)$-intersecting family $\m Q$ of subsets of $[n] \times [n]$ of size $\leq n-t+k$, where $k \leq n^{\frac{1}{2}-\frac{\epsilon}{8}}$. Then, we go back to the permutation setting and view the elements of $\m Q$ as \emph{partial permutations} -- i.e., permutations defined on part of the domain. 

 In order to handle such families, we develop a generalization of the \emph{transformation from permutations to sets}~\eqref{Eq:Perms_to_Sets} to the setting of partial permutations. To this end, we assume w.l.o.g.~that $\m Q$ contains a partial identity permutation on $I \subset [n]$, where $|I|$ is the maximum domain size among the partial permutations in $\m Q$. For each $\sigma \in \m Q$, we define $D_{\sigma}=\mathrm{Moving}(\sigma) \cap I$ and $M_{\sigma}=\{(j,\sigma(j)):j \in D_{\sigma} \cup ([n] \setminus I)\}$ (this is needed to account for the fact that our knowledge is limited to what happens inside $I$), and transform $\sigma$ to $D_\sigma \cup M_\sigma \subset \mb X':=[n] \cup ([n] \times [n])$. Applying this transformation to all elements of $\m Q$, we obtain a family $\m F$ of subsets of $\mb X'$ partitioned into the subfamilies $\m F_i=\{D_\sigma \cup M_{\sigma} \in \m F: |D_{\sigma}|=i\}$, $0 \leq i \leq t$, such that for any $i,j$, the families $\m F_i, \m F_j$ are cross $(i+j-t')$-intersecting, for $t'=|I|-(n-t)$.

%\andrey{again, a small comment on this peeling-simplification --- what I called peeling-simplification in other papers is what you call the `iterative peeling off spread parts'. The reason for the name was that we 1. Simplify by replacing spread subfamilies by their core 2. Peel off top layers. 3. Repeat.} \nathan{Thanks for the explanation! So do you prefer that I switch to the name `peeling-simplificaiton', which you used in other works? I am fine with this, if so you prefer.} 
Then, we apply to each of the families $\m F_i$ an iterative weighted variant of the \emph{peeling simplification} procedure, where the weight of each $D_\sigma \cup M_{\sigma} \in \m F$ is the number of extensions of $\sigma$ to a full permutation. We obtain approximating families $\m S_i$ consisting of sets of size at most $2i-t'+\ell$, such that $\ell\leq \frac{200}{\epsilon}$ and each $\m S_i, \m S_j$ are cross $(i+j-t')$ intersecting. 

The rest of the proof is similar to the `medium $t$' case. We show that most elements of $\m F$ are included in $\cup_{p \leq i \leq p+\ell} \m F_i[\m S_i]$, for some $p \leq t'$. Then we apply again the
\emph{peeling simplification} procedure to each $\m S_i$ ($p \leq i \leq p+\ell)$ to obtain $(2i-t')$-intersecting families $\m W_i$ such that most sets $F \in \m F$ contain some $W \in \m W_i$ (for some $i$) and we can bound the number of sets of each size in $\m W_i$. We deduce by the pigeonhole principle that there exist $p \leq m \leq p+\ell$ and $0 \leq k \leq \ell$ such that
$|\m F_{m}[\m W_m^{(2m-t'+k)}]| \geq (2\ell)^{-2}|\m F|$. Then, we use the bounds on the sizes of the $\m W_i$'s to deduce that this must occur for $k=0$, i.e., that $|\m F_{m}[\m W_m^{(2m-t')}]| \geq (2\ell)^{-2}|\m F|$. As $\m W_m^{(2m-t')}$, being an $(2m-t')$-intersecting family of $(2m-t')$-element sets, must consist of a single element $F$, this allows us to show that the properties described in Step~(1) above hold with respect to $F$.
\end{enumerate}
The combination of Step~(1), in each of the three ranges of $t$, with Step~(2), completes the proof.

\subsection{Organization of the rest of the paper.} The following sections are organized as follows.  In Section~\ref{sec:PermutationstoSets} we present the non-standard representation of permutations as sets and prove some relations between the sizes of restrictions of set families obtained as a result of this permutations-to-sets transformation. In Section~\ref{sec:IterativeSpreadApproximation} we present our most complex technical tool, the iterative spread approximation lemma and its proof. In Section~\ref{sec:peeling_off} we present the `peeling simplification' procedure and prove its properties. In Sections~\ref{sec:Small_t},~\ref{sec:Medium_t} and~\ref{sec:Large_t} we accomplish Step~(1) of the proof -- namely, we find a common simple sub-structure within the sets of $\m F$, for small $t$ (i.e., $t< n^{\epsilon}$), medium $t$ (i.e., $n^{\epsilon} \leq t< n^{\frac{1}{2}+\epsilon}$), and large $t$ (i.e., $n^{\frac{1}{2}+\epsilon} \leq t < n- n^{1-\frac{\epsilon}{8}}$), respectively. Finally, in Section~\ref{sec:Completing_the_Proof} we accomplish Step~(2) in all three ranges of $t$, thus completing the proof, and prove the stability version of the theorem. 
%\nathan{Do we prove the stability version?}

\section{Non-standard Representation of Permutations by Sets}
\label{sec:PermutationstoSets}

In this section, we present our first simplification, a new non-standard representation of permutations by sets that will allow us to look at the $t$-intersection problem for permutations `from the other side'. First, we present the representation of permutations by sets and study the sizes of restrictions of set families obtained from it. Then, we generalize the representation to \emph{partial permutations} (i.e., permutations $\sigma \in S_n$ in which some of the pairs $(i,\sigma(i))$ were removed).

\subsection{Our representation of permutations by sets}

There are various natural ways to represent permutations by sets. An example, which we will also use frequently in the sequel, is representing $\sigma \in S_n$ by the set of pairs $\{(i,\sigma(i)):i \in [n]\}$, which is an $n$-element subset of $[n] \times [n]$. Here, we use another representation that focuses on the \emph{moving points} of the permutation. 
For $\sigma \in S_n$, we define 
\[
D_{\sigma}:=\mathrm{Moving}(\sigma)=[n] \setminus \mathrm{Fixed}(\sigma), \qquad  
E_{\sigma}:= \{(i,\sigma(i)):i \in D_\sigma\}, \quad \mbox{and}
\]
\[
F_\sigma := D_\sigma \sqcup E_\sigma. 
\]
That is, we represent each permutation by the set of its moving points and the information on where they move. It is clear that $\sigma$ can be recovered from $F_\sigma$. We view $F_\sigma$ as a subset of size $2|\mathrm{Moving}(\sigma)|$ of the $n^2$-element set
\[ 
\mb X = [n]\sqcup\{(i,j):i,j \in [n], i\neq j\}.
\]
Each $X \subset \mb X$ can be written in the form $X =D \sqcup E$, where $D=X \cap [n]$ and $E= X \cap \{(i,j):i,j \in [n], i\neq j\}$. Slightly abusing notation, we denote $(D,E):= D \sqcup E$. In the same vein, for any $D'\subset [n]$ and $E'=\{(i,\sigma(i)): i\in D'\}$, we denote $\sigma(D'):=\{\sigma(i):i \in D'\}$ and sometimes write $(D',\sigma(D'))$ instead of $E'$. In particular, for $\sigma \in S_n$, we write $F_\sigma = \{D_\sigma,E_\sigma\}=\big\{D_\sigma,(D_\sigma,\sigma(D_\sigma))\big\}.$

In the sequel, we will apply the representation of $\sigma$ by $F_{\sigma}$ to transform (intersecting) families of permutations $\m G \subset S_n$ to families $\m F\subset \m P(\mb X)$. Let $\m G \subset S_n$ be $(n-t)$-intersecting. We may assume w.l.o.g.~that $\m G$ contains the identity permutation, as otherwise, we can multiply all elements of $\m G$ by $\tau^{-1}$ for some $\tau \in \m G$, without affecting the $(n-t)$-intersection assumption on $\m G$. This implies that for any $\sigma \in \m G$, we have $|\mathrm{Moving}(\sigma)|\leq t$. Decompose  $\m G = \cup_{i=0}^t\m G_i$, where
$$\m G_i:= \{\sigma\in \m G: |\mathrm{Moving}(\sigma)|=i\}.$$
We transform $\m G$ into the family $\m F:= \sqcup_{i=0}^t\m F_i \subset \m P(\mb X)$, where
\begin{equation}
\label{eqfi1}
\m F_i = \{F_\sigma: \sigma\in \m G_i\}.
\end{equation}
Throughout the paper, we shall use the notations
\[
\aaa_i := \{F \subset \mb X: F=F_{\sigma} \mbox{ for some } \sigma \in S_n, |\mathrm{Moving}(\sigma)|=i\}, \qquad \aaa: =\sqcup_i \aaa_i.
\]
In these notations, we have $\aaa_i=\{(D,E) \in \aaa: |D|=i\}$ and $\m F_i \subset \aaa_i$, for all $i$.

\medskip It turns out that, in some sense, our transformation preserves intersection properties.
\begin{claim}
\label{clatranslation}
    A family $\m G \subset S_n$ is $(n-t)$-intersecting if and only if for any $i,j\in \{0,\ldots, t\}$, the corresponding families $\m F_i\subset {\mb X \choose 2i}$ and $\m F_j\subset {\mb X \choose 2j}$ are cross $(i+j-t)$-intersecting.
\end{claim}
We prove the claim later in a greater generality, for partial permutations, below (see proof of Claim~\ref{clatranslation2}).

\medskip To exemplify our representation of permutations by sets, let us apply it to the $(n-t)$-intersecting families 
\[
\m G = \{\sigma \in S_n:|\mathrm{Fixed}(\sigma) \cap [n-t+2r]| \geq n-t+r\},
\]
denoted by $\m F_{n,n-t,r}$ in the statement of Theorem~\ref{Thm:Main}.
Any permutation $\sigma \in \m G$ has at most $t-r$ moving points, and hence the family $\m F =\cup_i \m F_i \subset \m P(\mb X)$ that corresponds to $\m G \subset S_n$ is contained in $\cup_{j=0}^{t-r} \m A_j$. Moreover, $\m F_{t-r}=\m F\cap \m A_{t-r}$ corresponds to the family of all  permutations whose set of moving points contains the entire set $\{n-t+2r+1,\ldots,n\}$ and exactly $r$ points in $[n-t+2r]$. 
%Thus, $\m T_i\cap \m A_{t-i}$ is a star with center of size $t-2i$: \nathan{The terms `star' and `center' weren't defined and are not used elsewhere, so maybe it's better to not use them.}
Thus, we have
$$\m F_{t-r} = \{(D,E) \in \m A_{t-r}: \{n-t+2r+1,\ldots,n\}\subset D\}.$$ 
More generally, for each $0 \leq j \leq t-r$, we have
$$\m F_{t-r-j} = \{(D,E) \in \aaa_{t-r-j}: |\{n-t+2r+1,\ldots,n\}\cap D|\ge t-2r-j\}.$$
It is easy to see that for any  $F_1 \in \m F_{t-r-j_1}$ and $F_2\in \m F_{t-r-j_2}$ we have $|F_1\cap F_2| \geq |F_1 \cap F_2 \cap \{n-t+2r+1,\ldots,n\}| \ge t-2r-j_1-j_2,$ as is guaranteed by Claim~\ref{clatranslation}.

\subsection{Restrictions of the families $\aaa_i$ and their sizes}

A restriction $\m F[X]$ of a family $\m F \subset \m P(\mb X)$ is obtained by fixing some elements and taking all sets $F \in \m F$  
%\ohad{$F \in \m F$?}\nathan{Yes, thanks!}
that contain them. Such elements can be singletons (i.e., $j \in [n]$) and/or pair-elements (i.e., $(i,j):i,j \in [n], i \neq j$).   

Recall that $\aaa_i := \{F \subset \mb X: F=F_{\sigma} \mbox{ for some } \sigma \in S_n, |\mathrm{Moving}(\sigma)|=i\}$ and $\aaa:= \sqcup_i \aaa_i$. In the sequel, we will make significant use of restrictions of $\m A_i$'s, and in particular, of restrictions whose size is maximal among all restrictions that fix the same number of elements. Specifically, for all $\ell,i \leq n$, we define 
\[
a^{(i)}_\ell:=\max_{X: |X|=\ell} |\aaa_i[X]|.
\]
Note that for $X=(D,E)$ where $E=(M,\sigma(M))$, $\m A_i[X]$ is the subfamily of $\m P(\mb X)$ that corresponds to the family of all permutations with exactly $i$ moving points, whose set of moving points contains $D$, and which contain the pairs $(i,\sigma(i))$, for all $i \in M$. 
While the maximum $a^{(i)}_\ell$ is not necessarily obtained for a set of the form $X=(D,(D,\sigma(D)))$, a weaker structural property does hold.
\begin{obs}\label{obs:max-structure}
    Let $\aaa_i, a^{(i)}_\ell$ be as defined above. For any $i,\ell \leq n$, the value $a^{(i)}_\ell$ is attained (also) by $\aaa_i[X]$ for some $X=(D,E)$ such that $E=(M,\sigma(M))$, where $M \subset D$ and $\sigma(M) \subset D$.
\end{obs}

\begin{proof}
Consider the restriction $\m A_i[X]$, where $X=(D,E)=(D,(M,\sigma(M)))$. If $M \not \subset D$, then for any $u \in M \setminus D$, we have $\m A_i[X \cup \{u\} \setminus \{(u,\sigma(u))\}] \supset \m A_i[X]$, as any permutation that contains the pair $(u,\sigma(u))$ must have $u$ among its moving points. Hence, one can replace each such pair $(u,\sigma(u)) \in X$ with the singleton $u$ without decreasing the size of the restricted family. This can be repeated as long as $M \not \subset D$, and thus, for any $i,\ell$, the maximum restriction size $a^{(i)}_\ell$ is obtained (also) for a restriction $\m A_i[X']$ such that $X'=(D,(M,\sigma(M)))$ where $M \subset D$.

Similarly, if $\sigma(M) \not \subset D$, then for any $\sigma(u) \in \sigma(M) \setminus D$, we have $\m A_i[X \cup \{\sigma(u)\} \setminus \{(u,\sigma(u))\}] \supset \m A_i[X]$, as any permutation that contains the pair $(u,\sigma(u))$ must have $\sigma(u)$ among its moving points. The assertion follows.
\end{proof}

The following technical lemma proves some properties of the values $a^{(i)}_\ell$ that we will use in our proofs. 
%For convenience, when $\ell\le 0$ then $a^{(i)}_\ell$ is just $a^{(i)}_0$. 
%We prove the lemma later on. \andrey{All of it is relatively straightforward calculations. I have an earlier variant in the last section of the paper, but this needs to be cleaned (and new parts need to be added)}

\begin{lem}
\label{lemadecay}
Let $n,i,\ell \in \mathbb{N}$ and let $\aaa_i, a^{(i)}_\ell$ be as defined above.
%, $1 \leq i \leq n$. 
%For $\pi\in S_n$ let $D_\pi,E_\pi,F_\pi$ be as above, and let
%\[
%\mathcal A_i:=\{F_\pi:\ |D_\pi|=i\}, \qquad
%a_\ell^{(i)}:=\max\{|\mathcal A_i[Y]|:\ Y\subset %\mb X,\ |Y|=\ell\}.
%\]
Then the following properties hold.
\begin{enumerate}
\item[(i)] For every $0\le \ell\le i-1$, we have
%\[
%\frac{a_\ell^{(i)}}%{a_{\ell+1}^{(i)}}
%\;\ge\;
%\min\Bigl\{\frac{n}{i},\,\frac{i}{2e}\Bigr\}.
%\]
$a_\ell^{(i)}/{a_{\ell+1}^{(i)}}
\;\ge\;
\min\Bigl\{\frac{n}{i},\,\frac{i}{2e}\Bigr\}$.
%\]

\item[(ii)] For every even value of $\ell$, $4\le \ell\le i$, we have
%\[
%\frac{a_\ell^{(i)}}{a_{\ell+1}^{(i)}}
%\;\ge\;
%\frac{1}{3}\cdot \frac{n}{i}.
%\]
$
a_\ell^{(i)}/{a_{\ell+1}^{(i)}}
\;\ge\;
\frac{n}{3i}$.

\item[(iii)] For every $3 \leq i \leq \sqrt{n/2}$, every $3 \leq \ell_0 \leq i$, and every $0 \leq \ell \leq 2i-\ell_0$, we have
$
    a^{(i)}_{\ell_0+\ell}
    \le
    3\left(\frac{2i}{n}\right)^{\lfloor \ell/2\rfloor}
    a^{(i)}_{\ell_0}.
$

%\item[(iii)] For every $3\le \ell\le i-2$, we have
%\nathan{The chat wrote $3\le \ell\le 2i-3$, I think this is wrong.}
%\[
%\frac{a_\ell^{(i)}}{a_{\ell+2}^{(i)}}
%\;\ge\;
%\min\Bigl\{\frac13\Bigl(\frac{n}{i}\Bigr)^2,\ \frac{n}{ei}\Bigr\}.
%\]
%$
%a_\ell^{(i)}/a_{\ell+2}^{(i)}
%\;\ge\;
%\min\Bigl\{\frac13\Bigl(\frac{n}{i}\Bigr)^2,\ \frac{n}{6e}\Bigr\}$.
%\nathan{The chat wrote in the right term $\frac{n}{ei}$, I think this is wrong.}

\item[(iv)] For every $0\le \ell\le 2i-1$, we have
%\[
%\frac{a_\ell^{(i)}}{a_{\ell+1}^{(i)}}\le n.
%\]
$
a_\ell^{(i)}/a_{\ell+1}^{(i)}\le n$. In addition,
%\[
%a_{2i}^{(i)}=1
%\qquad\text{and}\qquad
%a_\ell^{(i)}=0 \ \text{ for %all }\ \ell>2i.
%\]
$a_{2i}^{(i)}=1$ and $a_\ell^{(i)}=0$ for all  $\ell>2i$.

\item[(v)] For every $1 \leq t \leq n$ and every $i \geq \lceil t/2\rceil$, there exists an
$(n-t)$-intersecting family in $S_n$ of size $a_{2i-t}^{(i)}$. %\ohad{I think we dont use $i \le t$ in the proof, but we need $i>t$ in lemma~\ref{lemdensity2}}\nathan{Good, I changed this.}

\item[(vi)] For every $4 \leq t \leq n$, we have
%if $m:=\lfloor t/2\rfloor\ge2$ 
%\nathan{The chat also had the condition: and $t\ll n^{1-\varepsilon}$. I think it is not needed}, then
\[
\sum_{j=0}^{\lceil t/2 \rceil-1} a_0^{(j)} \;\le\; \tfrac{6}{n} \cdot\,a_0^{(\lceil t/2 \rceil)}. 
%\qquad \mbox{and} \qquad \sum_{j=0}^{\lceil t/2 \rceil-1} a_1^{(j)} \;\le\; \tfrac{6}{n} \cdot\,a_1^{(\lceil t/2 \rceil)}.
\] 
%for all sufficiently large $n$.

\item[(vii)] %Let $t\le n^\varepsilon$, let $i\le t$, and let $F=(D,E)\subset X$ be
%consistent, with
%\[
%|F|=\ell,\qquad 3\le %\ell,\qquad |D|\ge |E|+2.
%\]
For every $1 \leq i \leq n$, every $3 \leq \ell \leq 2i-3$ and every $F=(D,E)\subset \mb X$ such that $|F|=\ell$ and $|D|\geq |E|+2$, we have    
%\ohad{looks like $\ell \le 2i-3$ is needed}\nathan{Right, fixed.}
%$
%|F|=\ell, 3\le \ell,$ and $|D|\ge |E|+2$.
%Then
%\[
%|\mathcal A_i[F]|
%\;\le\;
%2e\,n^{-1+2\varepsilon}\,a^{(i)}_\ell.
%\]
%$
%|\mathcal A_i[F]|
%\;\le\;
%2e \cdot n^{-1+2\varepsilon}\,a^{(i)}_\ell$.
$|\mathcal A_i[F]|\le \frac{3i^2}{n} \cdot a^{(i)}_\ell$.

\item[(viii)] For every $1 \leq j \leq n-1$, we have $a_1^{(j+1)} \geq (1-\frac{j}{n}) \cdot j \cdot a_0^{(j)}$.
\end{enumerate}
\end{lem}

\begin{proof}
We begin with a few preparation steps. Throughout the proof (except for the proof of part $(vi)$ below), we fix $i$ and write $a_\ell:=a_\ell^{(i)}$. We consider sets $X=(D,E)\subset \mb X$, where 
$E=(M,\sigma(M)):=\{(u,\sigma(u)):u\in M\}$, for some $\sigma \in S_n$. 

%Note that for each $\ell$, the maximum $a_\ell=\max_{|Y|=\ell}\m |A_i[Y]|$ is obtained for $Y=(D,(M,\sigma(M)))$ such that $M \subset D$. Indeed, if for some $Y$ we have $M \not \subset D$, then for any $u \in M \setminus D$, we have $|\m A_i[Y \cup \{u\} \setminus \{(u,\sigma(u))\}]| \geq |\m A_i[Y]|$. Hence, we assume throughout the proof that the sets $Y$ we fix satisfy $M \subset D$. 

%where $\sigma:M\to[n]$ is injective. As in the proof of part~(i) in the draft, by replacing
%a pair $(u,\sigma(u))$ with the point-element $u$ whenever $u\notin D$, we may assume
%throughout that
%\[
%M\subset D.
%\]

Let $d_r$ denote the number of derangements of $[r]$ (i.e., the number of permutations over $[r]$ that do not have any fixed point). We use the standard bounds
\begin{equation}\label{eq:der-short}
\frac{d_r}{r!}\ge \frac13,
\qquad
\frac{d_r}{(r-1)!}\ge \frac{r-1}{e}
\qquad (r\ge2).
\end{equation}
We shall also use the following bound. 
%If $W\subset[n]$ and
%\[
%Q=(N,\tau(N))=\{(u,\tau(u)):u\in N\}
%\]
%is a consistent family of pair-elements with $N\subset W$ and $\tau(N)\subset W$, then
%\begin{equation}\label{eq:canon-lower}
%|\mathcal A_i[(W,Q)]|
%\ge
%\binom{n-|W|}{i-|W|}\,d_{i-|Q|}.
%\end{equation}
%\ohad{It is equal if $N=\tau(N)$ }
For any $X=(D,E)=(D,(M,\sigma(M)))$ such that $M \subset D$ and $\sigma(M) \subset D$, we have
\begin{equation}\label{eq:canon-lower}
|\mathcal A_i[(D,E)]|
\ge
\binom{n-|D|}{i-|D|}\,d_{i-|E|},
\end{equation}
with equality if $\sigma(M)=M$.
%\ohad{It is equal if $N=\tau(N)$ }
Indeed, once a moving-points-set of size $i$ containing $D$ is fixed, the remaining part of the permutation is a bijection between two sets of size $i-|E|$, with fixed points forbidden only in their intersection. %and the bipartite-derangement injection gives at least
%$d_{i-|Q|}$ completions.
The number of such bijections is clearly at least 
$d_{i-|E|}$.

\smallskip \noindent\emph{Proof of (i).}
Let $0\le \ell\le i-1$. Let $X=(D,E)$ be a restriction for which  $a_{\ell+1}$ is attained, such that $E=(M,\sigma(M))$ where $M\subset D$ and $\sigma(M) \subset D$. (Such an $X$ exists, by Observation~\ref{obs:max-structure}). Put
\[
x:=|D|,
\qquad
y:=|E|=|M|.
\]
Then $x+y=\ell+1$ and $y\le x$, so
$
y\le \tfrac{\ell+1}{2}\le \tfrac{i}{2}$.

If $y=0$, then $X$ consists of $x=\ell+1$ singleton-elements. Removing any singleton from $D$ (i.e., setting $X'=X \setminus \{u\}$ for some $u \in D$) gives, using the equality case of~\eqref{eq:canon-lower},
\begin{equation}\label{Eq:Aux-Restrictions1}
\frac{a_\ell}{a_{\ell+1}} \ge \frac{|\m A_i[X']|}{|\m A_i[X]|}
\ge 
\frac{\binom{n-x+1}{i-x+1}d_i}{\binom{n-x}{i-x}d_i}
=
\frac{n-x+1}{i-x+1}
\ge
\frac{n}{i}.
\end{equation}
Assume now that $y\ge1$. Choose $u\in M$, and let
\[
X':=\left(D,\ E\setminus\{(u,\sigma(u))\}\right).
%\qquad
%U:=D\cup \sigma(M).
\]
For every moving-points-set $S\subset[n]$ of size $i$ containing $D$, the number of permutations
$\pi$ with $D_\pi=S$ and $X\subset F_\pi$ is at most $(i-y)!$, whereas the number with
$D_\pi=S$ and $X'\subset F_\pi$ is at least $d_{i-y+1}$. For every moving-points-set $S\subset[n]$ of size $i$ that does not contain $D$, the number of permutations
$\pi$ with $D_\pi=S$ and $X\subset F_\pi$ is $0$. 
Hence, we have
\[
\frac{a_\ell}{a_{\ell+1}}
\ge
\frac{|\mathcal A_i[X']|}{|\mathcal A_i[X]|}
\ge
\frac{d_{i-y+1}}{(i-y)!}
\ge
\frac{i-y}{e}
\ge
\frac{i}{2e},
\]
where the penultimate inequality uses~\eqref{eq:der-short} and the last inequality holds since $y \leq \frac{i}{2}$. Combining the cases $y=0$ and $y \geq 1$, we get
$
\frac{a_\ell}{a_{\ell+1}}
\ge
\min\left\{\frac{n}{i},\,\frac{i}{2e}\right\}.
$

\smallskip
\noindent\emph{Proof of (ii).}
Let $\ell$, $4 \leq \ell \leq i$, be even. 
Let $X=(D,E)$ be a restriction for which  $a_{\ell+1}$ is attained, such that $E=(M,\sigma(M))$ where $M\subset D$ and $\sigma(M) \subset D$. 
%Assume that $a_{\ell+1}$ is obtained for $|\m A_i[Y]|$ for $Y=(D,E)$, where $E=(M,\sigma(M))$ with $M\subset D$. 
As above, put
$
x:=|D|,
y:=|E|$.
Then $x+y=\ell+1$ is odd and $y\le x$, so in fact $x\ge y+1$.

%Choose $W\subset[n]$ with $|W|=x-1$, and choose a consistent family \ohad{consistent just means $(i,j)$ pairs are as we expect from permutation, $i\neq j$, all $i$'s different.} $Q$ of $y$ pair-elements
%supported on $W$ \ohad{Meaning $M,\sigma(M) \subseteq W$. You choose a totally unrelated $Q$ to the one in $a_{\ell+1}$}; this is possible since $x-1\ge y$, and because $\ell\ge4$ we are not in the
%exceptional case $(|W|,|Q|)=(1,1)$. By \eqref{eq:canon-lower},
Choose $X'=(D',(M',\sigma'(M'))$, such that $|D'|=x-1$, $|M'|=y$, $M' \subset D'$, and $\sigma'(M') \subset D'$. Such a choice is possible since $x-1 \geq y$ and $\ell \geq 4$. By \eqref{eq:canon-lower}, we have 
\[
a_\ell \ge |\m A_i[X']| \geq \binom{n-x+1}{i-x+1}d_{i-y}.
\]
On the other hand, every permutation counted by $\mathcal A_i[X]$ must move all points of $D$,
and once the $y$ prescribed images from $E$ are fixed there are at most $(i-y)!$ completions.
Thus,
\[
a_{\ell+1}\le \binom{n-x}{i-x}(i-y)!.
\]
Therefore,
\[
\frac{a_\ell}{a_{\ell+1}}
\ge
\frac{n-x+1}{i-x+1}\cdot \frac{d_{i-y}}{(i-y)!} \geq \frac{n}{i}\cdot \frac{d_{i-y}}{(i-y)!}.
\]
Since $2y\le \ell\le i$, we have $i-y\ge2$, and hence,~\eqref{eq:der-short} gives
$
\frac{a_\ell}{a_{\ell+1}}
\ge
\frac{n}{i}\cdot \frac13.
$

\smallskip
\noindent\emph{Proof of (iii).}
% assertion follows from multiplying the assertions of parts $(i),(ii)$.
Let $X=(D,E)$ be a restriction for which $a_{\ell_0+\ell}$ is attained, such that $E=(M,\sigma(M))$ where $M\subset D$ and $\sigma(M) \subset D$. As above, put
$
x:=|D|,
y:=|E|$.
Let 
$x_0:=\left\lceil \ell_0/2\right\rceil, y_0:=\left\lfloor \ell_0/2\right\rfloor$.
Choose $X'=(D',(M',\sigma'(M'))$, such that $|D'|=x_0$, $|M'|=y_0$, $M' \subset D'$, and $\sigma'(M') \subset D'$. (This is possible, as by assumption, $\ell_0 \geq 3$). Since $a^{(i)}_{\ell_0} \geq |\m A_i[X']|$, by~\eqref{eq:der-short} and~\eqref{eq:canon-lower}, we have 
\[
    a^{(i)}_{\ell_0} \geq |\m A_i[X']|
    \ge
    \binom{n-x_0}{i-x_0}d_{i-y_0}
    \ge
    \frac13\binom{n-x_0}{i-x_0}(i-y_0)!,
\]
while 
\[
    a^{(i)}_{\ell_0+\ell}
    \le
    \binom{n-x}{i-x}(i-y)!.
\]
Therefore,
\begin{equation}\label{Eq:Aux-Restrictions2}
\frac{a^{(i)}_{\ell_0+\ell}}{a^{(i)}_{\ell_0}}
\le
3\frac{\binom{n-x}{i-x}}{\binom{n-x_0}{i-x_0}}
  \frac{(i-y)!}{(i-y_0)!}.
\end{equation}
Put \(z:=x-x_0\) and \(w:=y-y_0\). Since \(x+y=\ell_0+\ell\) and
\(x_0+y_0=\ell_0\), we have \(z+w=\ell\). Moreover, \(y\le x\), and hence,
\(z\ge\lfloor \ell/2\rfloor\). 

The binomial ratio in~\eqref{Eq:Aux-Restrictions2} is at most \((2i/n)^z\). If \(w\ge0\), the factorial
ratio is at most \(1\), and hence,~\eqref{Eq:Aux-Restrictions2} yields the desired bound. If \(w<0\),
then \(z>\ell\), and the factorial ratio is at most \(i^{z-\ell}\), and thus,~\eqref{Eq:Aux-Restrictions2} yields
\[
\frac{a^{(i)}_{\ell_0+\ell}}{a^{(i)}_{\ell_0}}
\le
3\frac{\binom{n-x}{i-x}}{\binom{n-x_0}{i-x_0}}
  \frac{(i-y)!}{(i-y_0)!} \leq 3\left(\frac{2i}{n}\right)^z i^{z-\ell} \leq 3\left(\frac{2i}{n}\right)^{\lfloor \ell/2 \rfloor},
\]
where the last inequality holds since
$2i^2 \leq n$. Hence, the assertion holds in both cases. 

\smallskip
\noindent\emph{Proof of (iv).}
Every set in $\mathcal A_i$ has size $2i$, so clearly,
$
a_{2i}=1$ and $a_\ell=0$ for all $\ell>2i$.
Fix $0\le \ell\le 2i-1$. 
Let $X=(D,E)$ be a restriction for which  $a_{\ell}$ is attained, such that $E=(M,\sigma(M))$ where $M\subset D$ and $\sigma(M) \subset D$. 
Put
$
\mathcal S:=\mathcal A_i[X]$.

If $\ell=0$, then by the equality case of~\eqref{eq:canon-lower},
\[
a_0=|\mathcal A_i|=\binom{n}{i}d_i,
\qquad
a_1 = \binom{n-1}{i-1}d_i,
\]
so $a_0/a_1 = n/i\le n$. Assume from now on that $\ell\ge1$.
%, so $\mathcal S\neq\emptyset$.

%Every $\pi\in \mathcal S$ moves all points of $D\cup \sigma(M)$.
%If $\sigma(M)\setminus D\neq\emptyset$, pick $w\in \sigma(M)\setminus D$. Then $w\notin Y$,
%but $w\in F_\pi$ for every $\pi\in \mathcal S$, and hence,
%\[
%a_{\ell+1} \geq |\mathcal A_i[Y\cup\{w\}]|=|\mathcal S|=a_\ell.
%\]
%Therefore, $a_\ell/a_{\ell+1} = 1\le n$. Assume now that $\sigma(M)\subset D$.

If there exists $u\in D\setminus M$, then as $\pi$ ranges over $\mathcal S$, the value
$\pi(u)$ takes at most $n$ possibilities. Hence, for some $v\in[n]$, at least $|\mathcal S|/n$
permutations $\pi \in \m S$ satisfy $\pi(u)=v$. Since $u\notin M$, the pair $(u,v)$ does not already belong
to $E$, and therefore,
\[
a_{\ell+1}\ge |\mathcal A_i[X\cup\{(u,v)\}]|\ge \tfrac1n\,|\mathcal S|=\tfrac1n\,a_\ell.
\]

It remains to consider the case $M=D$. Since $\sigma(M)\subset D$ and
$|M|=|\sigma(M)|=|D|$, we have $\sigma(M)=D$, so the restriction to $D$ is completely
determined by $E$.

%So we get $i>|D|$, then e
Each $\pi\in \mathcal S$ moves exactly $i-|D|\ge1$ points outside $D$.
So there exists $w \not \in D$ that is moved by at least $1/n$ of the elements $\pi\in \mathcal S$.
Therefore,
\[
a_{\ell+1}\ge |\mathcal A_i[X\cup\{w\}]|\ge \tfrac1n\,|\mathcal S|=\tfrac1n\,a_\ell.
\]
Thus, $a_\ell/a_{\ell+1}\le n$ in all cases.

\smallskip
\noindent\emph{Proof of (v).} Choose $X=(D,E)\subset \mb X$ with $|X|=2i-t$, such that $|\mathcal A_i[X]|=a_{2i-t}$. Define
\[
\mathcal G:=\{\pi\in S_n:\ |D_\pi|=i,\ X\subset F_\pi\}.
\]
Then $|\mathcal G|=a_{2i-t}$. Denoting by $\m F$ the family of subsets of $\mb X$ corresponding to $\m G$, we have $\m F_i=\m A_i[X]$ and $\m F_j=\emptyset$ for all $j \neq i$. As $\m F_i=\m A_i[X]$ is $(2i-t)$-intersecting, $\m G$ is $(n-t)$-intersecting by Claim~\ref{clatranslation}. 
%If $\pi,\rho\in\mathcal G$, then
%\[
%|F_\pi\cap F_\rho|\ge |Y|=2i-t.
%\]
%Since $|D_\pi|=|D_\rho|=i$,
%\[
%|\pi\cap \rho|
%=
%n-|D_\pi\cup D_\rho|+|E_\pi\cap E_\rho|
%=
%n-2i+|F_\pi\cap F_\rho|,
%\]
%and so $|\pi\cap \rho|\ge n-t$. Thus $\mathcal G$ is $(n-t)$-intersecting. \ohad{This is immediate from lemma~\ref{clatranslation}}

\smallskip
\noindent\emph{Proof of (vi).}
%First, we prove the claim about the values $a_0^{(j)}$. 
Put $m=\lceil t/2 \rceil$. For each $j$, we clearly have $a_0^{(j)}=|\mathcal A_j|=\binom{n}{j}d_j$. Note that $\sum_{j=0}^{m-1} a_0^{(j)}$ is the number of permutations with at most $m-1$ moving points, and so, 
\[
\sum_{j=0}^{m-1} a_0^{(j)} \le {n \choose m-1} \cdot (m-1)!.
\]
Indeed, there are \(\binom{n}{m-1}(m-1)!\) ways to choose a set \(S\subset[n]\) of size \(m-1\)
and a permutation on $S$; extending it by the identity outside \(S\)
produces a permutation of \([n]\) with a moving-points-set contained in \(S\). Every permutation with at most \(m-1\) moving points arises in this way.
%(take any \(S\) containing its support).

Hence, we have 
\[
\frac{\sum_{j=0}^{m-1} a_0^{(j)}}{a_0^{(m)}} \leq \frac{{n \choose m-1} \cdot (m-1)!}{{n \choose m} \cdot d_m} \le \frac{3}{n-m+1} \le \frac{6}{n},
\]
where the penultimate inequality holds by~\eqref{eq:der-short} and the last inequality holds since $m-1=\lceil t/2 \rceil -1 < \frac{n}{2}$.

%\smallskip To prove the claim about the values $a_1^{(j)}$, note that $a_1^{(j)}$ is the number of permutations whose set of moving points is of size $j$ and includes the singleton $\{1\}$, and that $\sum_{j=0}^{m-1} a_1^{(j)}$ is the number of permutations whose set of moving points is of size at most $m-1$ and includes the singleton $\{1\}$. Hence, we have $a_1^{(j)}=\binom{n}{j-1}d_j$ and $\sum_{j=0}^{m-1} a_1^{(j)} \leq \binom{n}{m-2}(m-1)!$. The rest of the proof is a calculation similar to the calculation in the proof of the assertion for the values $a_0^j$.

\smallskip
\noindent\emph{Proof of (vii).}
Let $\ell \geq 3$ and let $F=(D,E)\subset \mb X$ be such that $|F|=\ell$ and $|D|\geq |E|+2$.
Write $x:=|D|$, $y:=|E|$. If $\mathcal A_i[F]=\emptyset$, the assertion holds trivially, and hence by the assertion of part $(iv)$, we may assume $\ell \le 2i$. We clearly have
\[
|\mathcal A_i[F]|
\le
\binom{n-x}{i-x}(i-y)!,
\]
since one first chooses the remaining moving points, and then completes the pair-elements prescribed by $E$.

Choose $F'=(D',(M',\sigma'(M'))$, such that $|D'|=x-1$, $|M'|=y+1$, $M' \subset D'$, and $\sigma'(M') \subset D'$. Such a choice is possible since $x \geq y+2$ and $\ell \geq 3$. By \eqref{eq:canon-lower}, we have %\ohad{What if $i-y-1=1$}\nathan{Fixed now.}
\[
a^{(i)}_\ell \ge |\m A_i[F']| \geq \binom{n-x+1}{i-x+1}d_{i-y-1},
\]
%Now choose a set $D'\subset[n]$ with $|D'|=x-1$, and choose
%\[
%E'=(M',\tau(M'))
%\]
%with $|M'|=y+1$, supported on $D'$, where $\tau$ is a derangement on $M'$; this is possible
%because $x\ge y+2$ (and $\ell \ge 3$). The resulting set $F'=(D',E')$ has size $\ell$ so
%\[
%a^{(i)}_\ell \ge \binom{n-x+1}{i-x+1}D_{i-y-1}.
%\]
%Also $\ell=x+y\le 2i$ and $x\ge y+2$ imply $y\le i-3$, so $m:=i-y-1\ge0$.
where $i-y-1 >1 $ since $\ell \leq 2i-3$. Hence,
\[
\frac{|\mathcal A_i[F]|}{a^{(i)}_\ell}
\le
\frac{i-x+1}{n-x+1}\cdot \frac{(i-y)!}{d_{i-y-1}}
\leq \frac{i}{n} \cdot 3(i-y) \leq \frac{3i^2}{n},
\]
where the penultimate inequality holds by~\eqref{eq:der-short}.
%\[
%\frac{|\mathcal A_i[F]|}{a^{(i)}_\ell}
%\le
%\frac{i-x+1}{n-x+1}\cdot \frac{(i-y)!}{D_{i-y-1}}
%=
%\frac{i-x+1}{n-x+1}\cdot \frac{(m+1)!}{D_m}.
%\]
%Using \eqref{eq:der-short},
%\[
%\frac{(m+1)!}{D_m}\le 3(m+1)\le 3i,
%\]
%Therefore
%\[
%\frac{|\mathcal A_i[F]|}{a^{(i)}_\ell}\le \frac{i}{n}\cdot 3\,i \le 3\,\frac{i^2}{n} \le3\,n^{-1+2\varepsilon} .
%\]

\smallskip \noindent\emph{Proof of (viii).} Let $\m A_j'=\{F_{\sigma}=(D_{\sigma},E_{\sigma}) \in \m A_j:1 \in D_{\sigma}\}$ and $\m A_j''=\m A_j \setminus \m A_j'$. Clearly, $|\m A_j'|=\frac{j}{n}|\m A_j|=\frac{j}{n}a_0^{(j)}$, and thus, $|\m A_j''|=(1-\frac{j}{n})a_0^{(j)}$. 

Let $\m U=\{F_{\sigma'}=(D_{\sigma'},E_{\sigma'}) \in \m A_{j+1}: 1 \in D_{\sigma'}\}$. We have $|\m U|=a_{1}^{(j+1)}$. We will show that $|\m U|\geq j |\m A_j''|$, which clearly implies the assertion. To show this, we define $j$ injections $f_1,\ldots,f_j:\m A_{j}'' \to \m U$ and show that the sets $\{f_i(\m A_j'')\}_{i=1,\ldots,j}$ are pairwise disjoint. 

The injection $f_i$ is defined as follows. Let $F_{\sigma}=(D_{\sigma},E_{\sigma}) \in \m A_{j}''$ and denote $D_{\sigma}=\{\ell_1,\ell_2,\ldots,\ell_j\}$, where $\ell_1 <\ldots<\ell_j$. Set $\sigma'=(1,\ell_i) \circ \sigma$ (where $(1,\ell_i)$ is the transposition that exchanges $1$ and $\ell_i$) and set $f_i(F_{\sigma}):=F_{\sigma'}$. We have $D_{\sigma'}=\{1,\ell_1,\ldots,\ell_j\}$, and thus, $f_i(F_\sigma) \in \m U$. It is clear that $f_i$ is indeed an injection. Finally, the sets $\{f_{i}(\m A_{j}'')\}$ 
%\ohad{$ U_{j}''= \m A_{j}''$?}
%\nathan{Yes!}
are pairwise disjoint, since if $F_{\sigma'}=(D_{\sigma'},E_{\sigma'}) \in f_{i}(\m A_{j}'')$
%\ohad{$ U_{j}''= \m A_{j}''$?}
, then $\sigma'(1)$ is the $(i+1)$'th smallest element in $D_{\sigma'}$. This completes the proof.
\end{proof}

\iffalse
Take a family $\ff = \cup_{i=0}^t\ff_i$ defined from $\m G$ as above. Since $\m G$ was the largest, by Lemma~\ref{lemadecay}~(v), we must have $|\m F|\ge a^{(\lfloor t/2\rfloor)}_0$, and thus by Lemma~\ref{lemadecay}~(vi)
\begin{equation}\label{eqlowlayers}|\cup_{i=0}^{\lfloor t/2\rfloor-1}\ff_i| = o(|\ff|).\end{equation} It means that we can safely ignore the contributions of small $i$ for a while. Thus, in the regimes (i) and (ii) we concentrate on the families $\ff_i$ for $\lfloor t/2\rfloor\le i\le t$. 

We discuss regime~(iii) in Section~\ref{sec8} and first deal with regimes~(i) and~(ii).
\fi

\subsection{Our representation of partial permutations by sets}
\label{subsec:PartialPerms_to_Sets}

A partial permutation is an injective %\ohad{injective?}\nathan{Yes!} 
function from $S \subset [n]$ to $[n]$. We denote the family of all partial permutations on $[n]$ by $\Sigma_n$, and the family of all partial permutations on $[n]$ with domain 
%\andrey{domain or support?} \nathan{The term `support' here can be misleading as it can refer both to the set of moving points and to the set of points on which the partial permutation is defined. To avoid this ambivalence, I didn't use the word `support' at all, and used `moving' for the former and `domain' for the latter. Here I indeed meant `domain'. Is this fine?} 
size $i$ by $\Sigma_n^{(i)}$. So, $\Sigma_n =\sqcup_{i=0}^n \Sigma_n^{(i)}$. For $\sigma \in \Sigma_n$, we denote the domain of $\sigma$ by $I_\sigma$ and identify $\sigma$ with the set of pairs $\{(j,\sigma(j)):j \in I_{\sigma}\}$. As in the case of `full' permutations, we use the notation $(I_{\sigma},\sigma(I_{\sigma})):=\{(j,\sigma(j)):j \in I_{\sigma}\}$.

In this subsection, we generalize the representation of permutations by sets presented above, in a way that will allow us to transform families of partial permutations to families of sets. In the proof of our main theorem, we shall apply this transformation to $(n-t)$-intersecting families $\m Q \subset \Sigma_{n}$ in which the size of the domain of all elements is at most $n-t+k$, for some $k \ll n-t$. (Specifically, in all applications we will have $n-t \geq n^{1-\epsilon}$ and $k \leq n^{1/2-\epsilon/8}$). Unlike the case of full permutations considered above, for partial permutations the transformation depends on the choice of some parameters. We pick one such choice (presented below) and stick to it consistently. %\nathan{We have to verify that this choice does not cause problems in Section 8.2.}

\medskip \noindent \emph{Preparation steps.} Fix $n,t,k$, and let $\m Q \subset \Sigma_{n}$ be an $(n-t)$-intersecting family such that for each $\sigma \in \m Q$, $|I_{\sigma}| \leq n-t+k$. 
%By the intersection property of $\m Q$, we have $|I_{\sigma}| \geq n-t$ for each $\sigma \in \m Q$. 
Without loss of generality, we may assume that $\m Q$ contains a \emph{partial identity permutation} on some set $I$ of size $n-t \leq |I| \leq n-t+k$, i.e., the permutation $id_I$ defined by $I_{id_I}=I$ and $id_I(j)=j$ for all $j \in I$. (Otherwise, we can take some $\pi \in \m Q$, extend it arbitrarily to a full permutation $\pi' \in S_n$, and replace $\m Q$ by the family $\pi'^{-1}\m Q = \{\pi'^{-1}\tau:\tau \in \m Q\}$ which preserves the size and the intersection property of $\m Q$ and contains the partial identity permutation $id_{I_{\pi}}$). As $\m Q$ is $(n-t)$-intersecting, for any $\sigma \in \m Q$ we have $|I_\sigma\cap I|\ge n-t$.

We slightly modify $\m Q$, in order to make all partial permutations in it defined on all elements of $I$. 
%w.r.t. $I$. 
%(Formally, we fix the modification described below for every valid pair $\m Q$, $I$ as described below, but otherwise arbitrarily. \nathan{This sentence is not clear to me. It seems that we take a single $\pi$ arbitrarily, this defines $I=I_\pi$ and this is $I$ we work with.} \andrey{I meant that the modification that follows is defined for every valid choice of $\m Q, I$. If it's confusing, then we may just remove it}.) 
For each $\sigma=(I_\sigma,\sigma(I_\sigma))\in \m Q$ and for each element $i\in I\setminus I_\sigma$, we replace $\sigma\in \m Q$ with all partial permutations on $I_{\sigma} \cup \{i\}$ that extend it. We repeat this process until for each resultant partial permutation $\sigma'$, we have $I_{\sigma'} \supset I$. We obtain an $(n-t)$-intersecting family $\m Q'$ of partial permutations such that for each $\sigma' \in \m Q'$, $n-t \leq |I_{\sigma'}| \leq n-t+2k$. Importantly, we have 
$S_n[\m Q] = S_n[\m Q']$. (That is, each full permutation that contains some element of $\m Q$ must contain some element of $\m Q'$). 

\medskip \noindent \emph{The transformation from partial permutations to sets.} For each $\sigma \in \m Q'$, we let 
\[
D_\sigma := \mathrm{Moving}(\sigma)\cap I
, \qquad M_\sigma:=(I_\sigma\setminus I)\cup D_\sigma,
\]
\[
E_{\sigma} := \{(i,\sigma(i)): i\in M_\sigma\}, \qquad \mbox{and} \qquad F_{\sigma}:= D_\sigma \sqcup E_{\sigma}.
\]
%\andrey{$F_\sigma$ above?} \nathan{To Andrey: I don't understand this comment, can you please explain?} \andrey{It's just that in the second displayed line it's written $\m F_{\sigma}:= D_\sigma \sqcup E_{\sigma}$, but I guess you meant $F_{\sigma}:= D_\sigma \sqcup E_{\sigma}$}\nathan{Right, of course. Fixed.}
We represent $\sigma$ by $F_{\sigma}$, i.e., by the set of its moving points \emph{inside $I$} and the information on where these points move and where \emph{all points outside $I$} move. Note that for full permutations, this definition reduces to the definition of $F_{\sigma}$ presented above. 
%The set $M_\sigma$ is the set of points on which $\sigma$ does not agree with $id_I$ (i.e., the partial identity on $I$ from which we started our construction). 
As above, we use the notation $E_\sigma := (M_\sigma, \sigma(M_\sigma))$.

It is clear that $\sigma$ can be recovered from $F_{\sigma}$. We view $F_{\sigma}$ as a subset of size $2|\mathrm{Moving}(\sigma) \cap I|+|I_{\sigma} \setminus I|$ of the $(n^2+n)$-element set 
\[
\mb X' = [n] \sqcup ([n] \times [n]).
\]
Note that unlike the case of full permutations, here pairs of the form $(j,j)$ can appear in $F_{\sigma}$ if $j \not \in I$. 

In the sequel we will use the following simple observation on the sizes of $D_\sigma, M_{\sigma}$ and $F_{\sigma}$ for $\sigma \in \m Q'$. 
\begin{obs}\label{obssize}
    For any $\sigma\in \m Q'$, we have 
    \[
    |D_\sigma|\le |I|-(n-t) \leq k, \qquad |M_\sigma|\le |D_{\sigma}|+k \leq 2k, \qquad  |F_\sigma|\le 2|D_{\sigma}|+k \leq 3k.
    \]
\end{obs}
\begin{proof}
By the $(n-t)$-intersection property of $\m Q'$, each $\sigma \in \m Q'$ agrees with $id_I$ on at least $n-t$ elements of $I$. Hence, $|D_\sigma|=|\mathrm{Moving}(\sigma) \cap I|\le |I|-(n-t)\le k$. 
As for each $\sigma \in \m Q$ we have $|I_{\sigma} \setminus I| \leq k$ and only elements from $I$ are added to $I_{\sigma}$ in the transition from $\m Q$ to $\m Q'$, for each $\sigma' \in \m Q'$ we have $|M_{\sigma'}| = |D_{\sigma'}|+|I_{\sigma'} \setminus I| \leq 2k$. Thus, $|F_\sigma|= |D_\sigma|+|M_\sigma|\le 3k,$ as asserted. 
\end{proof}
\noindent Decompose  $\m Q' = \cup_{i=0}^k \m Q'_i$, where
$$\m Q'_i:= \{\sigma\in \m Q': |D_{\sigma}|=i\}.$$
We transform $\m Q'$ into the family $\m F:= \sqcup_{i=0}^k \m F_i \subset \m P(\mb X')$, where
\begin{equation}
\label{eqfi2}
\m F_i = \{F_\sigma: \sigma\in \m Q'_i\}.
\end{equation}
For each $0 \leq i \leq k$, we denote 
\[
\m B_i := \{F \subset \mb X': F=F_{\sigma},\sigma \in \Sigma_n^{(j)}, j \in \{n-t,\ldots,n-t+2k\}, |D_{\sigma}|=i, I_{\sigma} \supset I\}.
\]
Clearly, for each $i$ we have $\m F_i \subset \m B_i$.

\medskip \noindent \emph{Preserving intersection.} The following claim shows that, in some sense, our transformation preserves intersection properties.
This claim generalizes Claim~\ref{clatranslation} above, which follows from it by substituting $I=[n]$.
\begin{claim}\label{clatranslation2}
    The family $\m Q'$ is $(n-t)$-intersecting if and only if for any $i,j$, the corresponding families $\m F_i$  and $\m F_j$ are cross $(i+j-|I|+n-t)$-intersecting.
\end{claim}

\begin{proof}
    For a partial permutation $\sigma$, let $X_\sigma:=I \setminus D_\sigma$ be the set of fixed points of $\sigma$ on $I$. Take any two partial permutations $\sigma\in \m Q'_i,\pi\in \m Q'_j$.  Their intersection consists of two parts. One part is the set $X_\sigma\cap X_\pi$ of common fixed points on $I$. The second part is the set of points in $M_\sigma\cap M_\pi$ on which the two partial permutations agree, including points outside of $I$ that are fixed by both $\sigma$ and $\pi$. The size of this set is exactly $|E_{\sigma} \cap E_{\pi}|$.
    We have $X_\sigma\cap X_\pi = I\setminus (D_\sigma\cup D_\pi)$, and thus,
    \begin{align*}
        |X_\sigma\cap X_\pi|&= |I|- |D_\sigma\cup D_\pi| = |I| -|D_\sigma|-|D_\pi|+ |D_\sigma\cap D_\pi| \\
        &=|I|-i-j+|D_\sigma\cap D_\pi|.
    \end{align*}
    Hence,
    \begin{align*}|\sigma\cap \pi| &= |X_\sigma\cap X_\pi|+|E_\sigma\cap E_\pi|\\
    &= |I|-i-j+|D_\sigma\cap D_\pi|+|E_\sigma\cap E_\pi|\\
    &=|I|-i-j+|F_\sigma\cap F_\pi|.\end{align*}
    This implies that $|F_\sigma\cap F_\pi|\ge i+j-|I|+n-t$ holds if and only if $|\sigma\cap \pi|\ge n-t.$ The assertion follows by applying this equivalence statement to all pairs of partial permutations in $\m Q'$.
\end{proof}

\noindent \emph{A weighted setting.}
A main difference between our treatment of partial permutations and the treatment of full permutations presented above is that for partial permutations, we introduce a \emph{weighted} setting. For each $\sigma \in \Sigma_n^{(i)}$, we set the weight $\mu(\sigma):= (n-i)!$, which is the number of extensions of $\sigma$ to a full permutation on $[n]$. For a family $\m Z \subset \Sigma_n$ we set $\mu(\m Z)$ to be the number of full permutations that extend at least one permutation from $\m Z.$ 
%\nathan{This means that $\mu(\m Z) \neq \sum_{\sigma \in \m Z} \mu(\sigma)$, which looks quite odd. In particular, this does not correspond to the iterative spread approximation lemma where we wrote that in the weighted version, the weight of a set is the sum of weights of its elements. On the other hand, this is indeed the definition we use in the proof below. What do I miss?}
In particular, this applies to the families $\m Q', \m Q_i'$ defined above. Note that since for each $\sigma \in \m Q'$ we have $I_{\sigma} \supset I$, all full permutations that extend $\sigma \in \m Q_i'$ have exactly $i$ moving points in $I$.

For a set $F_\sigma$, where $\sigma \in \Sigma_n$, we define $\mu(F_{\sigma}):=\mu(\sigma)$. For a family $\m B_i \subset \m P(\mb X')$, we define $\mu(\m B_i)$ to be equal to the weight of the corresponding family of partial permutations. 

The definition of spreadness naturally extends to the weighted setting. %\andrey{In this paragraph, shouldn't we use square brackets, so that we count the correct number of extensions?}\nathan{Yes, of course! Thanks!} 
For $r \geq 1$, we say that a family $\ff \subset \m P(\mb X')$ is {\it $(r,\mu)$-spread} if for each non-empty set $T$ we have $\mu(\ff[T])< r^{-|T|} \mu(\ff)$. For $s \in \mathbb{N}$, we say that $\m F$ is {\it $(r,\mu,s)$-spread} if for any disjoint sets $S,T$ with $|S|=s$ and $T \neq \emptyset$, we have $\mu(\ff[S \cup T])< r^{-|T|}\mu(\ff[S])$. We say that $\m F$ is {\it weakly $(r,\mu,s)$-spread} if the above holds for $S_0$ such that $\mu(\ff[S_0])=\max_{S:|S|=s}\mu(\ff[S])$ and any $T \neq \emptyset$. In other words, denoting $b_m=\max_{S:|S|=m} \mu(\ff[S])$ for every $m \in \mathbb{N}$, $\m F$ is {\it weakly $(r,\mu,s)$-spread} if $b_{s+t}< r^{-t}b_s$ for all $t>0$. 

\medskip \noindent \emph{Restrictions of the families $\m B_i$ and their weighted sizes.}
Given $F=(D,E)$ with $D\subset I$, $\mu(\bb_i[F])$ is defined as the number of full permutations that extend some partial permutation that corresponds to an element of $\m B_i[F]$. (Note that the restriction $F$ does not necessarily correspond to a partial permutation). 
Put $$b^{(i)}_\ell:=\max_{F:|F|=\ell}\mu(\bb_i[F]).$$ 
We obtain a comparison between weighted sizes of families of the form $\m B_i[F]$, in the spirit of Lemma~\ref{lemadecay}(i).
\begin{claim}
\label{claimdecay}
   Let $n,t,k,i$ be such that $i \leq k \leq t/5$. For any $F=(D,(M, \sigma(M)))$, where $|D|\le i$, $|M|\le 2k$, $(M \cap I) \subset D$ and $(\sigma(M) \cap I) \subset D$, and any $F'$ obtained from $F$ by adding either a singleton element or a pair element, we have 
     \begin{equation}\label{eqdecayb} \mu(\bb_i[F])/\mu(\bb_i[F'])\ge \tfrac12 \cdot \min \left\{\frac{n-t-k}{k}, t- 3k \right\}.
     \end{equation}
   Consequently, for all $0 \leq i \leq k$ and all $\ell$, we have 
%\begin{align}\label{Eq:Compb}
$b^{(i)}_\ell/ b^{(i)}_{\ell+1} \geq \tfrac12 \min \left\{\frac{n-t-k}{k}, t- 3k \right\}$, 
%\end{align}
which means that $\m B_i$ is weakly $(r,\mu,t')$-spread for  $r=\tfrac{1}2 \cdot \min \left\{\frac{n-t-k}{k}, t- 3k \right\}$ and any $t' \in \mathbb{N}$.
\end{claim}
\begin{proof}
Consider $F = (D,E) = (D,(M, \sigma(M)))$ that satisfies the assumptions of the claim. Note that since each partial permutation $\pi$ that corresponds to an element of $\m B_i[F]$ satisfies $|D_\pi|=i$ and $I_\pi \supset I$, each full permutation $\sigma$ that extends it satisfies $|D_{\sigma}|= i$ as well. Since $(M \cap I) \subset D$ and $(\sigma(M) \cap I) \subset D$, the number of such permutations is ${|I|-|D|\choose i-|D|}\cdot d(E,D_\sigma)$, where $d(E, D_\sigma)$ is the number of bijections from $(D_\sigma\cup ([n]\setminus I))\setminus M$ to $(D_\sigma\cup ([n]\setminus I))\setminus \sigma(M)$ with no fixed points in $D_\sigma$. Indeed, there are ${|I|-|D|\choose i-|D|}$ possible ways to choose $D_{\sigma}$ (where $\sigma$ is the full permutation), and once $D_{\sigma}$ is chosen, there are $d(E, D_\sigma)$ possible ways to complete $\sigma$. 
As $$|(D_\sigma\cup ([n]\setminus I))\setminus M| \geq 
|[n]\setminus I|-|M|\ge t-3k,$$ and by assumption, $i \leq k \leq t/5,$ a union bound implies that among the bijections from $(D_\sigma\cup ([n]\setminus I))\setminus M$ to $(D_\sigma\cup ([n]\setminus I))\setminus \sigma(M)$,
%\ohad{$f$ is undefined. $\sigma$ seems right} \nathan{Right!} 
at most a $\frac{|D_{\sigma}|}{|(D_\sigma\cup ([n]\setminus I))\setminus M|} \leq \frac{k}{t-3k} \leq \frac{1}{2}$ fraction have a fixed point on $D_\sigma$. Hence, we have
$$d(E,D_{\sigma}) \geq \tfrac12(n-|I|+i-|M|)!,$$ %\ohad{$d(E,D_\sigma)$?}
since $(n-|I|+i-|M|)!$ is the number of bijections between these two sets with no fixed points restrictions.

If $F'$ is formed from $F$ by adding one element to $D$, then the number of extensions to a full permutation decreases by a factor of at least $$\frac12 \cdot\frac{{|I|-|D|\choose i-|D|}}{{|I|-|D|-1\choose i-|D|-1}}=\frac12 \cdot\frac{|I|-|D|}{i-|D|}\ge \frac12\cdot \frac {n-t-k}k. $$ 
%\sim \frac{n}k \gtrsim n^{\frac 12+\frac \epsilon 8}.$$
If $F'$ is formed from $F$ by adding one element to $E$, then the number of extensions to a full permutation decreases by a factor of at least
$$\frac12 \cdot \frac{(n-|I|+i-|M|)!}{(n-|I|+i-|M|-1)!}\ge \tfrac12 \cdot (t-3k).$$
%\ohad{If $|I| \le n-t+k$, $M\le 2k$ dont we get $1/2(t-3k)$? (we had the $t-3k$ calculation in the frist equation od this proof) does it matter?}\nathan{Right, thanks!}
%\gtrsim n^{\frac 12+\frac \epsilon 8}.
This proves the assertion~\eqref{eqdecayb}. 

To show the `Consequently' part, note that like in Observation~\ref{obs:max-structure}, for any $i,\ell$, the value $b_{\ell}^{(i)}$ is attained (also) by $\m B_i[X]$ for some $F'=(D',(M',\sigma'(M'))$ such that $(M' \cap I) \subset D'$ and $(\sigma'(M') \cap I) \subset D'$. This can be shown by repeating the proof of Observation~\ref{obs:max-structure} almost verbatim.

The `Consequently' part follows by applying~\eqref{eqdecayb} to a restriction $F'$ of the form $F'=(D',(M',\sigma'(M'))$ with $(M' \cap I) \subset D'$ and $(\sigma'(M') \cap I) \subset D'$ for which $b_{\ell+1}^{(i)}$ is obtained, and a restriction $F=(D,(M,\sigma'(M))$ obtained from $F'$ by removing a singleton element if $M'=\emptyset$ or a pair element if $M' \neq \emptyset$.
\end{proof}

\noindent \emph{Relation of the quantities $b_{\ell}^{(i)}$ to the sizes of $(n-t)$-intersecting families of permutations.} The following claim, which is a variant of Lemma~\ref{lemadecay}(v) above, shows that for certain values of $i$ and $\ell$, $b_{\ell}^{(i)}$ is the size of an $(n-t)$-intersecting family of permutations. For all $t \in [n]$,  denote 
\[
t':=|I|-(n-t).
\]
\begin{claim}\label{b_intersecting}
    For every $1 \leq t \leq n$ and every $\lceil t'/2\rceil\le i\le t'$, 
    %\ohad{$\lceil t'/2\rceil\le i\le t'$?}\nathan{Right, thanks!} 
    there exists an
    $(n-t)$-intersecting family in $S_n$ of size $b_{2i-t'}^{(i)}$.
\end{claim}

\begin{proof}
    Let $t \in [n]$
    %, and denote $t':=|I|-(n-t)$. 
    and let $\lceil t'/2\rceil\le i\le t'$. We want to show that $b_{2i-t'}^{(i)}$ is the size of some $(n-t)$-intersecting family $\m G \subset S_n$.
    Choose $X=(D,E)\subset \mb X'$ with $|X|=2i-t'$, such that $\mu(\mathcal B_i[X])=b_{2i-t'}^{(i)}$. (Such an $X$ exists by the definition of $b_{2i-t'}^{(i)}$). 
    Consider the family of partial permutations
\[
\m Q' := \{\sigma: \sigma \in \Sigma_n^{(j)}, j \in \{n-t,\ldots,n-t+2k\}, |D_{\sigma}|=i, I_{\sigma} \supset I, X \subset F_{\sigma}\}.
\]
Denoting by $\m F$ the family of subsets of $\mb X'$ corresponding to $\m Q'$, we have $\m F_i=\m B_i[X]$ and $\m F_j=\emptyset$ for all $j \neq i$. As $\m B_i[X]$ is $(2i-t')$-intersecting, Claim~\ref{clatranslation2} implies that $\m Q'$ is an $(n-t)$-intersecting family of partial permutations. Let $\m G \subset S_n$ be the family of all full permutations that contain some element of $\m Q'$. $\m G$ is clearly $(n-t)$-intersecting, and by the definition of $\mu$, we have $|\m G|=\mu(\m B_i[X])=b_{2i-t'}^{(i)}$. 
\end{proof}

\section{The Iterative Spread Approximation Lemma}
\label{sec:IterativeSpreadApproximation}

In this section, we present our second simplification, the iterative spread approximation lemma. This lemma is an enhancement of the spread approximation lemma of Kupavskii and Zakharov~\cite[Theorem~8]{KZ22}. Its advantage is that while~\cite[Theorem~8]{KZ22} has extra assumptions on the parameters which make it applicable only in the setting where the extremal example is a $t$-umvirate, the iterative spread approximation lemma (i.e., Theorem~\ref{thmapprox1} below) can be applied in the general setting where all families $\m F_{n,t,r}$ are candidates for being extremal. As was mentioned above, a similar procedure in the special case of families of $k$-subsets of $[n]$ was introduced by Frankl and Kupavskii~\cite{FranklK25} who used it to study the Hajnal-Rothschild problem. 
%\nathan{Andrey, can you please check that the sentences I added are fine?} \andrey{I think that it's fine!}
As the statement of the lemma and its proof are somewhat technical, we begin with an informal statement of the lemma and a proof outline, and then we present the formal statement and the full proof.

\subsection{Informal statement of the lemma and proof outline}

Throughout this section, we work with  an ambient `spread' family $\m A \subset \binom{[n]}{k}$, and use the notation 
\begin{equation*}
    a_i:=\max_{Z:|Z|=i}|\m A[Z]|, \qquad 0 \le i \le k.   
\end{equation*}
%The whole argument also applies for the \emph{weighted} setting, in which the elements of $\m A$ have positive weights, and for any $\m F \subset A$, $|\m F|$ is the sum of weights of the elements of $\m F$. \nathan{This does not agree with the definition of weights for partial permutations in Section~3.3!} 
Essentially, the iterative spread approximation lemma asserts the following. 
%\ohad{did we mean it to not have a number?} \nathan{Yes. The formal statement has a number (Theorem 14), and it is the one we use in later sections.}
\begin{theorem*}
[Iterative spread approximation lemma, informal statement]
      Let $n,k,t$ be integers, and let $\aaa\subset {[n]\choose k}$ be a %(possibly weighted) 
      family that is weakly $(R,t')$-spread for a `large' $R$ and all values of $t'$ around $t$, but is not `too spread'. 
      
      Let $\ff\subset \aaa$ be a $t$-intersecting family. Then there exists a $t$-intersecting family $\m S$ of subsets of $[n]$ such that:
      \begin{itemize}
        \item The size of each $S \in \m S$ is very close to $t$. 
        
        \item The family $\m R:=\m F \setminus\m F[\m S]$ is very small -- that is, almost every $F \in \m F$ contains some $S \in \m S$.

        \item For each $S \in \m S$, there exists $\m F_S \subset \m F$ such that $\m F_S(S)$ is $\frac{R}{2}$-spread.
      \end{itemize}
\end{theorem*}
    
%In words, the family $\m S$ is an approximation of $\m F$ by a family of sets of a small size, such that almost each set in $\m F$ contains a set in $\m S$, and for each $S \in \m S$, $\m F$ contains a subfamily $\m F_S$ such that $\m F_S[S]$ is spread. 
Formally, the statement of the theorem relies on several additional parameters: $\sigma \ge 0$, which measures how small is $\m R$ required to be; $\frac{1}{2} \le \alpha \le 1$, which measures how large is the spreadness parameter $R$ required to be, compared to $t$; and $R_1>R$, which quantifies the statement that $\m A$ is `not too spread' via the inequality $a_{t} \geq R_1^{t'-t}a_{t'}$ for values of $t'$ slightly smaller than $t$. Roughly speaking, if $R_1$ is not too large and $R>c(t^{\alpha}\log t+\sigma)$ for a sufficiently large constant $c$, then the theorem guarantees that all sets in $\m S$ are of size at most $t+O(t^{1-\alpha}+\sigma+\log t)$, and $|\m R| \leq O(2^{-\sigma}\log(Rt+\sigma) a_t)$. The exact relations between the parameters are cumbersome, and hence, will be presented later on. %\andrey{I think that, as a measure of efficiency of the theorem, we could add that, while the `original' spread approximation only works when the extremal example is trivial, this works in the regime when it's not necessarily trivial.}

    \medskip The proof of the theorem relies on three components:
    \begin{itemize}
        \item \emph{Theorem~\ref{thm2sets}: A sub-structure inside $t$-intersecting families.} This theorem asserts that if $\m G \subset \binom{[n]}{\le \ell}$ is a $t_1$-intersecting family and for some $\m F \subset \m A$,   $\m F[\m G]$ is `large', then there exists a set $X$ of size not much larger than $t_1$ such that $\m F[X]$ is `dense' (meaning that a large portion of the sets in $\m F$ contain the same `small' set $X$).

        \item \emph{Theorem~\ref{thmregularity}: Spread approximation.} This theorem asserts that if $\m F \subset \m A$ has the property that for each `not-too-small' subfamily $\m P \subset \m F$ there exists a `small' set $X$ such that a large portion of the sets in $\m P$ contain $X$, then there exists a family $\m S$ of `small' sets, such that $\m F$ can be partitioned as $\m F= (\bigsqcup_{S \in \m S} \m F_S[S]) \sqcup \m R$, where for any $S \in \m S$, the family $\m F_S(S)$ is `spread', and the remainder $\m R$ is `small'. This means that $\m F$ can be approximated by a union of spread pieces, where all sets in each piece contain the same small set $S \in \m S$. 
        The advantage of this step over the `usual' spread approximation is that we work with subfamilies of sets containing $X$, which are much denser in the corresponding ambient family than the family $\m F$ itself. This allows us to obtain an approximation with much better parameters -- specifically, with sets of a smaller size and a better spreadness.
        %\andrey{Maybe we want to add some words that, using this set $X$, we get a `bump in the density' and thus get an approx of better quality, as compared to the `original' spread approx} \nathan{To Andrey: I am not sure how to write this. Can you please suggest a formulation?}
        %\andrey{Maybe something along the following lines: The advantage over the usual spread approximation here is that we work with subfamilies of sets containing  $X$, which are much denser in the corresponding ambient family than the family $\m F$ itself. This allows us to have much better parameters of the approximation, that is, better spreadness and lower uniformity of the approximation.}

        \item \emph{Lemma~\ref{lemtint}: Partially preserving intersection.} This lemma asserts that if $\m F \subset \binom{[n]}{k}$ is $t$-intersecting and $\m S \subset \binom{[n]}{\leq \ell}$ has the property that for each $S \in \m S$, there is a family $\m F_S \subset \m F$ such that $\m F_S(S)$ is $r$-spread, then $\m S$ is $t'$-intersecting for some $t' \le t$ that depends on how large $r$ is, compared to $k$ and $\ell$.
        This allows deducing that the process of approximating $\m F$ by $\m S$, described in the previous step, partially preserves the intersection property of the family (though, degrading from $t$-intersecting to $t'$-intersecting). 
    \end{itemize}
The proofs of all three components are purely combinatorial and are not complicated. 

The proof of the iterative spread approximation lemma is a more complex iterative process in which Theorem~\ref{thm2sets}, Theorem~\ref{thmregularity} and Lemma~\ref{lemtint} are applied alternately. Given a $t$-intersecting family $\m F \subset \binom{[n]}{k}$, we first apply a simpler variant of Theorem~\ref{thmregularity} and Lemma~\ref{lemtint} to construct an initial $t^{(0)}$-intersecting approximating family $\m S^{(0)}$ for $\m F$, with $t^{(0)}=t-\lceil t^{1/2} \rceil$. Then, we apply Theorem~\ref{thm2sets} to each sufficiently large subfamily of $\m F[\m S^{(0)}]$ and deduce that a large portion of its elements contain a single `small' set $X_1$. This allows us to apply Theorem~\ref{thmregularity} to $\m F[\m S^{(0)}]$ and get a new approximating family $\m S^{(1)}$. Then, Lemma~\ref{lemtint} allows us to deduce that $\m S^{(1)}$ is $t^{(1)}$-intersecting, for some $t^{(1)}>t^{(0)}$. We then repeat the process with $\m S^{(1)}$ replacing $\m S^{(0)}$. We show that at each iteration of the process, the maximum size $q^{(i)}$ of a set in the approximating family $\m S^{(i)}$ decreases, while the guaranteed intersection size $t^{(i)}$ increases. Furthermore, the process converges quickly, and at the end, the approximating family becomes $t$-intersecting, and the maximum size of a set in it becomes only slightly larger than $t$. 

\subsection{A sub-structure inside $t$-intersecting families}

Our first component is the following.

\begin{thm}\label{thm2sets} Let $n,k,t_1,\ell \in \mathbb{N}$ be such that $n \geq k \geq t_1$ and $\ell\ge t_1\ge 1$, and let $\lambda>0$. Let $\m A \subset \binom{[n]}{k}$, and denote $a_j:=\max_{Z: |Z| = j} |\aaa(Z)|$ for each $j \leq k$. Let $\m G$ be a $t_1$-intersecting family of $(\le\ell)$-element sets, and let $\ff \subset \m A$ be a family such that $\ff[\m G]\subset \m A[\m G]$ satisfies $|\ff[\m G]|\geq  \lambda a_{t_1}$. 

Assume that $\aaa$ is weakly $(R,t_1)$-spread for some $R\ge 1$. Then there exists a set $X$ of size $x$,
\begin{equation}\label{eqsize2set0}
t_1\le x\le t_1+ 4\Big(\frac{t_1(\ell-t_1)^2}{R}\Big)^{1/3}+ \log_2\Big((t_1+1)\lambda^{-1}\Big),   
\end{equation}
such that, denoting $\beta = \frac{|\m F(X)|}{a_x}$, $0<\beta\le 1,$ we have
\begin{equation}\label{eqsize2set}
\beta \ge \left[(t_1+1)e^{3\big(\frac{t_1(\ell-t_1)^2}{R}\big)^{1/3}} \right]^{-1} \lambda.
%|\m F[\m G]|\le (t_1+1)e^{3\big(\frac{t_1(\ell-t_1)^2}{R}\big)^{1/3}} \beta a_{t_1}.
\end{equation}
\end{thm}
Roughly speaking, the theorem says that if a subfamily $\ff[\m G]$ of the $t_1$-intersecting family $\m A[\m G]$ is `large' then there exists $X$ of size not much larger than $t_1$ such that 
$\ff(X)$ is rather dense.

\begin{proof} Consider a $t_1$-intersecting family $\m G$ of $(\le \ell)$-element sets and take two sets  $A,B\in \m G$ that have the smallest intersection in $\m G$. Assume that they intersect in $t'\ge t_1$ elements.  Choose $I\subset A\cap B$, such that $|I| = t_1$, and put $D_1:=A\setminus B$, $D_2:=B\setminus A$. Note that $|D_1|,|D_2|\le \ell-t_1$. Then for each set $C \in \m G$  there is a value $i\in \{0,\ldots, t_1\}$ and sets $U,V,W$ such that 
\[
|U| = t_1-i, \quad |V|=|W| = i, \quad \mbox{and} \quad C\cap I = U, \quad C\cap D_1\supset V, \quad C\cap D_2\supset W.
\]
By the pigeonhole principle, there is a choice of $i$ and such sets $U,V,W$ such that
%$$\frac{|\m G(U\cup V\cup W)|}{|\m G|}\ge \frac 1{(t_1+1){t_1\choose t_1-i}{\ell-t_1\choose i}^2}.$$
%\andrey{New version:}
$$\frac{|\ff[\m G[U\cup V\cup W]]|}{|\ff[\m G]|}\ge \frac 1{(t_1+1){t_1\choose t_1-i}{\ell-t_1\choose i}^2}.$$
Indeed, for any given $i$ there are ${t_1\choose t_1-i}$ ways to choose $U\subset I$ and at most ${\ell-t_1\choose i}^2$ ways to choose $V\subset D_1$ and $W\subset D_2$. Thus, for one of the $t_1+1$ possible values of $i$ and one of the corresponding choices $U,V,W$ we must get the above inequality. 

Put $X = U\cup V\cup W$. We will show that $X$ satisfies the conditions of the theorem. Note that $x = |X|=t_1+i$ and recall that $|\ff(X)| = \beta a_x$. We bound the value of $i$ using the bound on the size of $\m F[\m G]$. We have
\begin{align}\label{eq65}
\begin{split}    
|\m F[\m G]|\le& (t_1+1){t_1\choose i}{\ell-t_1\choose i}^2 \beta a_x\\
\le& (t_1+1){t_1\choose i}{\ell-t_1\choose i}^2 R^{-i} \beta a_{t_1}\\
\le& (t_1+1) \Big(\frac{e^3t_1(\ell-t_1)^2}{R i^3}\Big)^i \beta a_{t_1}.
\end{split}
\end{align}
In the second inequality, we used the weak $(R,t_1)$-spreadness of $\m A$. In the last inequality, we used the inequality ${x\choose m}\le (ex/m)^m$, valid for any $x\ge m\ge 1$. If 
\[
i> 4\Big(\frac{t_1(\ell-t_1)^2}{R}\Big)^{1/3}+ \log_2\Big((t_1+1)\lambda^{-1}\Big),
\]
then the right hand side of \eqref{eq65} is less than
$$\beta\lambda a_{t_1}\le \lambda a_{t_1},$$
which contradicts our assumption on the size of $|\m F[\m G]|$. As $x:=|X|=t_1+i$, this shows that $x$ satisfies condition~\eqref{eqsize2set0} in the statement of the theorem.  
%Thus, we may assume that the opposite inequality on $i$ holds, i.e., that $i$ is `not too large'. In particular, this implies the inequality on $|X|$ stated in the theorem. 

Furthermore, as a function of $i$, the maximum in the right hand side of~\eqref{eq65} is attained for $i = \Big(\frac{t_1(\ell-t_1)^2}{R}\Big)^{1/3}$. This yields the bound
$$\lambda a_{t_1} \le |\m F[\m G]|\le (t_1+1)e^{3\big(\frac{t_1(\ell-t_1)^2}{R}\big)^{1/3}} \beta a_{t_1},$$
which implies condition~\eqref{eqsize2set} in the statement of the theorem by rearranging.
This completes the proof of Theorem~\ref{thm2sets}.
\end{proof}

\subsection{Spread approximation}

Our second component is the following. 
\begin{thm}\label{thmregularity}
  %Let $\eta,\theta,\ell_1,\ell_2, R,r>0$, $r\le R$. 
  Let $n,k,q,\ell_1,\ell_2,\eta \in \mathbb{N}$ and $R,r,\theta>0$ be such that $n \geq k$, $q \geq \ell_2 \geq \ell_1$ and $R \geq r$. Let $\m A \subset \binom{[n]}{k}$ be a family that is weakly $(R, l)$-spread for each $l\in [\ell_1,\ell_2]$, and denote $a_j:=\max_{Z: |Z| = j} |\aaa(Z)|$ for each $j \leq k$. Let $\ff \subset \m A$ be a family such that  
  for any $\m P \subset \m F$ of size 
  at least $\eta$, there is a set $X$ with $x$ elements, $\ell_1\le x\le \ell_2$, such that $|\m P| \le \frac{|\m P(X)|}{a_x}\cdot \theta$. 
  
  Then there exists a family $\m S$ of sets of size at most $q$ and a family $\m R\subset \ff$, such that the following holds.
  \begin{itemize}
    \item[(i)] $\ff= \m R\sqcup \bigsqcup_{S\in \m S} \m F_S[S]$;
    \item[(ii)] For any $S\in \s$ and the family $\ff_S\subset \ff$, the family $\ff_S(S)$ is $r$-spread;
    \item[(iii)] $|\m R|\le \max\big\{\eta, \theta \cdot (r/R)^{q+1-\ell_2}\big\}$.
  \end{itemize}
\end{thm}
Theorem~\ref{thmregularity} is a variant of the spread approximation theorem~\cite[Lemma~10]{KZ22} %\nathan{We should add a pointer to the theorem; I assume it is from~\cite{KZ22}.} \andrey{It should be Lemma~10, I believe.} 
that makes use of dense pieces within the set family we want to approximate. The idea is the following: Rather than searching for a spread approximation for the entire $\ff$, we find a set $X$ such that $\ff(X)$ is dense and then a spread piece inside it. Then we remove the spread piece from $\ff$ and repeat the process. The gain, compared to~\cite[Lemma~10]{KZ22}, is a better bound on the size of the remainder $\m R$.
%, stated in part (iii).

\medskip In the proof of Theorem~\ref{thmregularity}, we use the following simple yet important observations which assert that any sufficiently large family contains a spread subfamily.
\begin{obs}\label{obs13}
  Let $n,k \in \mathbb{N}$, let $\ff\subset {[n]\choose \le k}$ and let $r \geq 1$. If  $|\ff|>r^k$, then there exists $X \subset [n]$, $0\le |X|<k$, such that $\ff(X)$ is $r$-spread and contains at least two elements.
\end{obs}
\begin{proof}
  If $\ff$ is $r$-spread then we can put $X = \emptyset$. Otherwise, consider a maximal w.r.t.~inclusion set $X$ such that $|\ff(X)|\ge r^{-|X|}|\ff|$. Note that $|X|\le k-1$, since for any $X'$ with $|X'|=k$ we have $|\ff(X')| \leq 1<r^{-k}|\ff|$. 
  %\ohad{$|\ff(X)|\le 1<r^{-k}|\ff|$}. 
  By the maximality of $X$, for any $Y$ that is disjoint with $X$ we have $|\ff(X\cup Y)|< r^{-|X|-|Y|}|\ff|\le r^{-|Y|}|\ff(X)|.$ Thus, $\ff(X)$ is $r$-spread.
\end{proof}
The same proof implies the following:
\begin{obs}\label{obs34}
Let $n,k \in \mathbb{N}$, let $\ff\subset {[n]\choose \le k}$ and let $r \geq 1$.
If $X \subset [n]$ is a maximal w.r.t.~inclusion set such that 
$|\ff(X)|\ge r^{-|X|}|\ff|$, then  $\ff(X)$ is $r$-spread.
\end{obs}

Now we are ready to present the proof of the theorem.

\begin{proof}[Proof of Theorem~\ref{thmregularity}]
Consider the following iterative procedure for $i=1,2,\ldots $ with $\ff^1:=\ff$.
\begin{enumerate}
    \item If $|\ff^i|<\eta$ then stop.
    \item Applying the assumption of the theorem to $\m F^i \subset \m F$, take a set $X_i \in \ff^i$, 
    $\ell_1\le |X_i| \le \ell_2$, such that $|\m F^i| \le \frac{|\m F^i(X_i)|}{a_{|X_i|}}\cdot \theta$. 
    %as guaranteed in the statement of the theorem.
    \item Find a maximal for inclusion $S_i\supset X_i$ that  $|\ff^i(S_i)|\ge  r^{|X_i|-|S_i|}|\ff^i(X_i)|$.
    \item If $|S_i|> q$  then stop. Otherwise, put $\ff^{i+1}:=\ff^i\setminus \ff^i[S_i]$.
\end{enumerate}
%\ohad{This looks like the same process from the first paper on permutations, except choosing $X$ everytime. What do we get from the added $X$? } \andrey{It allows us to start the approximation from a dense piece of the family, and thus get a much better approximation (smaller uniformity, bigger spreadness).}
Let $N$ be the step at which the procedure stops. Set 
\[
\m S=\{S_1,\ldots,S_{N-1}\}, \quad \m F_S:=\m F^i[S_i] \quad \mbox{for all} \quad 0 \le i \le N-1, \quad \mbox{and} \quad \m R:=\m F_N.
\]
By the construction, all sets in $\m S$ are of size at most $q$ and we have $\ff= \m R\sqcup \bigsqcup_{S\in \m S} \m F_S[S]$. For each $1 \leq i \leq n$, Observation~\ref{obs34} and the maximality of $S_i$ imply that $\ff^i(S_i)$ is $r$-spread. 
%(We can argue as in Observation~\ref{obs13}.) %if for some nonempty $X, X\cap S_i = \emptyset$, we have $\ff^i(S_i\cup X)\ge r^{-|X|}|\ff^i(S_i)|$, then $\ff^i(S_i\cup X)\ge r^{-|X|+|S_i|}|\ff^i|$, a contradiction with maximality of $S_i$.
%Let $N$ be the step of the procedure for $\ff$ at which we stop. The family $\s$ is defined as follows: $\s:=\{S_1,\ldots, S_{N-1}\}$. Clearly, $|S_i|\le q$ for each $i\in [N-1]$. The family $\ff_{A}$ promised in (ii) is defined to be $\ff^i[S_i]$ for $A=S_i$. Next, we put $\m R:=\ff^N$. 
Furthermore, if $|\ff^N|\ge \eta$, then by the stopping rule of the procedure, we have $|S_N|>q$, and hence, 
%\textcolor{red}{inequality 3 below???? need to check or not?}
\begin{align*}|\ff^N|\le&\ \theta\frac{|\ff^N(X_N)|}{a_{X_N}}\le \theta\frac{ r^{|S_N|-|X_N|}  |\ff^{N}(S_N)|}{a_{X_N}}\\
\le&  \theta \frac{ r^{|S_N|-|X_N|}  |\aaa(S_N)|}{a_{X_N}} \le \theta \frac{r^{|S_N|-|X_N|}}{R^{|S_N|-|X_N|}} \le \theta (r/R)^{q+1-\ell_2}.%\\
%\le&  \theta \frac{ r^{|S_N|-\ell_2}   |\aaa(S_N)|}{\max_{X', |X'|=|X_N|} |\aaa(X')|}\le  \theta \frac{ r^{|S_N|-\ell_2}  {n-|S_N|\choose k-|S_N|}}{{n-\ell_2\choose k-\ell_2}}\\
%\le& \theta \frac{ r^{q+1-\ell_2}  {n-q-1\choose k-q-1}}{{n-\ell_2\choose k-\ell_2}}
\end{align*}
In the first inequality, we used the way $X_N$ was chosen in Step~(2). In the second inequality, we used the way $S_N$ was chosen in Step~(3). In the fourth inequality, we used the assumption that $\aaa$ is weakly $(R, |X_N|)$-spread. 
The last inequality holds since $|S_N|\ge q+1, |X_N|\le \ell_2$ and $R\ge r$. 
%Since either $|S_N|>q$ or $|\ff^N|<\eta$, we get the claimed inequality on $|\m R|$.
This shows that the bound (iii) on the size of $\m R$ holds as well, completing the proof.
\end{proof}

\subsection{Partially preserving intersection}

Our third component is the following lemma, which will allow us to deduce that the approximating family $\m S$ constructed in Theorem~\ref{thmregularity} partially inherits the intersection property of $\m F$. 

\begin{lem}\label{lemtint} 
Let $n,k,\ell,t \in \mathbb{N}$ be such that $n \geq k \geq \ell \geq t$ and let $r \geq 1$. Let $\ff\subset {[n]\choose k}$ be a $t$-intersecting family. Let $\mathcal S\subset {[n]\choose \le \ell}$ be a family such that for each $S\in \mathcal S$, there is a subfamily $\ff_S\subset \ff$ such that $\ff_S(S)$ is $r$-spread. Assume that for some $t'$, $t' \le t$, the following conditions are satisfied: \begin{align}\label{eqint1} r\ge&  \frac{24(\ell-t'+1)}{t-t'+1};\\
\label{eqint2} r\ge& 2^{14}\log_2(4k).
\end{align}
Then $\m S$ is $t'$-intersecting. \end{lem}
This lemma is the case $s=1$ of~\cite[Lemma~24]{FranklK25} and hence we omit its proof (which is a fairly easy spreadness argument, similar to the proof of Lemma~\ref{lemregime2}(3) below). %\nathan{Say something about the proof -- how simple it is and what tools it uses. Check the constants in~\eqref{eqint1} and~\eqref{eqint2} -- it seems that in~\cite{FranklK25} they are better by a factor of 2.} \andrey{the proof is very similar to what we use, say, in Lemma 22(3). The main challenge is to pass to from families $\m F_{S_1}, \m F_{S_2}$ corresponding to $S_1,S_2$, to families $\m X_{S_1},\m X_{S_2}$ that live on $[n]\setminus (S_1\cup S_2)$. We cannot guarantee them to be disjoint with $S_i\setminus S_j$, but we can guarantee small intersection with it. The key observation is that it's easier to ``kill'' with $r^x$ a term of the form ${|S_2\setminus S_1\choose x}$ for some large $x$, since it's $\Big(\frac{e|S_2\setminus S_1|}{x}\Big)^x$, rather than $|S_2\setminus S_1|^x$. Here, $x$ is essentially $(t-t')/2$. As I wrote, it's really not hard to reprove as well.}

\subsection{The iterative approximation lemma and its proof}

Now, we are ready to state and prove the main result of this section. Recall that in the `weighted' setting, each set $F \in \m A$ is given a positive weight, and for any $\m F \subset \m A$, $|\m F|$ is the sum of weights of the elements of $\m F$.

\begin{thm}[The iterative spread approximation lemma]
\label{thmapprox1}
  Let $n,k,t \in \mathbb{N}$ be such that $n \geq k \geq t$ and let $R_1,R,\sigma>0$ be such that $R_1 \geq R \geq 2$. Let $\m A \subset \binom{[n]}{k}$ be a (possibly weighted) family of sets, and denote $a_j:=\max_{Z: |Z| = j} |\aaa(Z)|$ for each $j \leq k$. 
%  Let $n,k,t$ be integers, $\sigma\ge 0$ and $R_1\ge R\ge 2$ real numbers, and consider a (possibly weighted) family $\aaa\subset {[n]\choose k}$. Putting $a_i:=\max_{Z: |Z| = i} |\aaa(Z)|$, 
Assume that $\aaa$ satisfies the following:
  \begin{itemize}
      \item $\m A$ is weakly $(R,t')$-spread for each $t'$ such that 
      \begin{equation}\label{Eq:Aux-Iterative1}
      t-\lceil t^{\frac{1}{2}}\rceil\le t'\le t+\tfrac 1{100} \big(t^{\frac{1}{2}}(k-t+\lceil t^{\frac{1}{2}}\rceil)^{2}\big)^{\frac{1}{3}}+2\sigma+2\lceil t^{\frac{1}{2}}\rceil \log_2 R_1+2\log_2(t+1);
      \end{equation}
      %\ohad{I think we actually prove $+2\log_2(t+1)$} \nathan{Right, thanks!}
      %the family $\aaa$ is weakly $(R,t')$-spread, i.e., $a_{t'+s}\le R^{-s}a_{t'}$;
      \item For each $t'$ such that $t-\lceil t^{1/2}\rceil\le t'\le t$, the family $\aaa$ is `not too weakly spread', concretely,  $a_{t'}\le R_1^{t-t'}a_{t}$. 
\end{itemize}
Furthermore, assume that $R$ satisfies the following:
\begin{itemize}
    \item $R\ge 2^{30}(t^{1/2}\log_2 t+\log_2 R_1)+200\sigma$;

    \item $R\ge 2^{15}\log_2(4k)$.
\end{itemize}
Then for every $t$-intersecting family $\ff\subset \aaa$, 
  %\ohad{I thought we need to change the inequality on $R$ a bit}
there exists a $t$-intersecting family $\m S$ such that the following hold: 
\begin{itemize}
    \item[(i)] Each set in $\m S$ is of size at most
    $$t+t^{1/2}+4\sigma+4\log_2(t+1);$$ 
    \item[(ii)] For any $S\in \m S$, there exists $\ff_S\subset \ff$ such that $\ff_S(S)$ is $R/2$-spread;
    \item[(iii)] The family
    $\m R:= \ff\setminus \ff[\m S]$ satisfies $$|\m R|\le 2^{-\sigma} \cdot 4\log_2(Rt+\sigma) a_t.$$
  \end{itemize}
%of sets of size at most
%   $$t+t^{1/2}+4\sigma+4\log_2(t+1),$$
%   such that the family $\m R:= \ff\setminus \ff[\m S]$ satisfies $$|\m R|\le 2^{-\sigma} \cdot 4\log_2(Rt+\sigma) a_t.$$ %\nathan{It seems from the proof that $R_1$ should be replaced by $R$.}\andrey{fixed}
%   Moreover, for each $A\in \m S$ there is $\ff_A\subset \ff$ such that $\ff_A(A)$ is $R/2$-spread.

  If in addition, we have $R\ge 2^{30}t^{\alpha}$ for some $\alpha\in (1/2,1]$ (that may depend on $n,t$ as long as it is between $1/2$ and $1$),
  then the bound on the sizes of sets in $\m S$ improves to $t+t^{1-\alpha}+4\sigma+4\log_2(t+1),$ with the same bound on the size of the remainder $\m R$.
\end{thm}

\begin{proof}[Proof of Theorem~\ref{thmapprox1}] 
We concentrate on proving the first part of the theorem, i.e., the case $\alpha=1/2$. As we show at the end of the proof, the case $1/2<\alpha \leq 1$ can be proved by a small variation of the proof for $\alpha=1/2$. 

If $k\le t+t^{1/2}+4\sigma+4\log_2(t+1)$, then the family $\ff$ itself can serve as the family $\m S$. Hence, in what follows, we assume that $k>t+t^{1/2}+4\sigma+4\log_2(t+1)$.

The proof of the theorem is an iterative  bootstrapping argument that goes back and forth between application of  Theorem~\ref{thm2sets} and combination of Theorem~\ref{thmregularity} with Lemma~\ref{lemtint}. At the $i$'th iteration, we denote the application of  Theorem~\ref{thm2sets} by Step~$A(i)$ and the combination of Theorem~\ref{thmregularity} with Lemma~\ref{lemtint} by Step~$B(i)$. As our argument requires the assumption $t-\lceil t^{1/2} \rceil>0$ (since we will apply Theorem~\ref{thm2sets} with this value of $t_1$, see below) and this assumption holds only for $t\ge 3$, we treat the cases $t=1, 2$ separately at the end of the proof. We also relay to the end of the proof the verification that all families $\m A$ to which we apply Theorems~\ref{thm2sets} and~\ref{thmregularity} during the proof satisfy the required spreadness assumptions.

\subsubsection{Preparation phase.}
If $k\ge  t+t\log_2 R+\sigma$ then we first perform the auxiliary step $B(0)$. If not, then we skip it and put 
\[
\m S^{(0)} = \ff, \quad q^{(0)} =  t+t\log_2 R+\sigma, \quad \mbox{and} \quad t'^{(0)} = t-\lceil t^{1/2}\rceil.
\]

%\nathan{Probably it will be better to not call this step A(0), since in the sequel, Step A is an application of Theorem 5, while Step B is an application of Theorem 6 and Lemma 7. Here we apply only Theorem 6, so maybe it's better to call the step B(0) and divide it into two parts.} \andrey{Ok!}

\noindent {\bf Step B(0).} This step, aimed at reducing the sizes of the sets in the examined family, consists of two parts. 

\medskip \noindent \emph{Application of a variant of Theorem~\ref{thmregularity}.} We apply the process described in Theorem~\ref{thmregularity} to $\ff$,  with 
\[
\ell_1 = \ell_2 = x=0, \quad \theta = |\aaa|, \quad \eta = 0, \quad r = R/2, \quad \mbox{and} \quad q=q^{(0)} = t+t\log_2 R+\sigma.
\]
%(Basically, we ignore the part of the statement that looks
%for the `dense' part of $\ff$ and do a `normal' spread approximation.)
Importantly, we use the notation and the iterative process of Theorem~\ref{thmregularity}, but use a different bound on the remainder instead of the bound used in the proof of Theorem~\ref{thmregularity}. As a result, we can carry out this argument, although the $(R, 0)$-spreadness of $\aaa$, required by Theorem~\ref{thmregularity}, is not assumed to hold in our setting. 

We obtain a family $\m S^{(0)} = \m S$ of sets, each of size at most $t+t\log_2 R+\sigma$, and a family $\m R \subset \m F$, such that:
\begin{itemize}
    \item[(i)] We have $\ff= \m R\sqcup \bigsqcup_{S\in \m S} \m F_S[S]$;

    \item[(ii)] For any $S\in \m S$ and the family $\ff_S\subset \ff$, the family $\ff_S(S)$ is $r$-spread;
    
    \item[(iii)] We have
\begin{align*}|\m R|&\le r^{|S_N|}a_{|S_N|} \le r^{|S_N|} R^{t-|S_N|}a_t\\
&= 2^{t-|S_N|} r^t a_t\le 2^{-t\log_2R-\sigma} r^t a_t \le 2^{-\sigma} a_t.
\end{align*}
  \end{itemize}
%$\ff= \m R\sqcup \bigsqcup_{S\in \m S} \m F_S[S]$ and for any $S\in \s$ and the family $\ff_S\subset \ff$, the family $\ff_S(S)$ is $r$-spread, and   \begin{align*}|\m R|&\le r^{|S_N|}a_{|S_N|} \le r^{|S_N|} R^{t-|S_N|}a_t\\
%&= 2^{t-|S_N|} r^t a_t\le 2^{-t\log_2R-\sigma} r^t a_t \le 2^{-\sigma} a_t.
%\end{align*}
In (iii), the first inequality holds by the way $S_N$ is constructed since $X_N=\emptyset$, the second inequality uses the weak $(R,t)$ spreadness of $\aaa$, and the third inequality uses the definition of $q$ and the fact that $|S_N|>q$.
%  $$|\m R|\le \Big(\frac {n-t}{2(k-t)}\Big)^{q}{n-q\choose k-q}\le 2^{-q}\Big(\frac {n-t}{k-t}\Big)^{t} {n-t\choose k-t}\le 2^{-\sigma}{n-t\choose k-t}.$$
%   Here the last inequality is by our choice of $q$.
%\begin{align*}|\m R|&\le r^{q+1}a_{q+1} \le r^{q+1} R^{t-q-1}a_t\\
%&\le 2^{t-q} r^t a_t\le 2^{-\sigma} a_t.
%\end{align*}
%\nathan{The third inequality looks incorrect. It seems this should be $2^{t-q-1}$ instead of $2^{q-t}$. If this is indeed the case, this might affect arguments in the sequel.}\andrey{This is a typo, corrected. The last bound should be fine, as well as all the rest.}
%\Big(2^{9} (2t)^{1/2}\big(1+\frac \sigma t\big)
%\log_2(2t)\Big)^{q} \Big(\frac{1}{2^{12} (st)^{1/2}\big(1+\frac \sigma t\big)\log_2(st)}\Big)^{q-t}{n-t\choose k-t}\\
%& =
%8^{-q+t} \Big(2^9 (st)^{1/2}\big(1+\frac \sigma t\big)\log_2(st)\Big)^{t}{n-t\choose k-t}\\
%&\le
%2^{-\sigma}{n-t\choose k-t}.
%   Here, we used the weak $(R,t)$-spreadness of $\aaa$ and the definition of $q$. (Exceptionally, we do not need the weak $(R, 0)$-spreadness of $\aaa$ for this application, as required by Theorem~\ref{thmregularity}.) %the third inequality is due to the bound $n\ge t+C(st)^{1/2}\big(1+\frac\sigma t\big)(k-t)\log_2\frac{n-t}{k-t}$, which in particular implies $\log_2\frac{n-t}{k-t}\ge \frac 12\log_2(st)$. The
%   last inequality is due to our choice of $q$ and the fact that $(1+\frac \sigma t)^t\le 4^\sigma$.

%{\bf Step B(0) } 

\medskip \noindent \emph{Application of Lemma~\ref{lemtint}.} We would like to apply Lemma~\ref{lemtint} to $\m S$, with 
\[
t'=t'^{(0)} = t-\lceil t^{1/2}\rceil, \quad \ell=q^{(0)}, \quad \mbox{and} \quad r=R/2.
\]
Let us verify that the assumptions of the lemma are satisfied.
For each $S\in \m S$ there is a family $\ff_S\subset \ff$ such that $\ff_S(S)$ is $r$-spread. \eqref{eqint2} is satisfied by the assumption on $R$ in Theorem~\ref{thmapprox1}, since $r=\frac{R}{2}$. 
As for \eqref{eqint1}, it is implied by the inequality
  $$\frac R2\ge 24\frac{\ell}{t^{1/2}} = 24\left(t^{1/2}\log_2 (2R)+\frac{\sigma}{t^{1/2}}\right),$$
  which is again guaranteed by our assumption on $R$. Hence, we can apply Lemma~\ref{lemtint} to conclude that $\m S$ is $t'^{(0)}$-intersecting.

\subsubsection{The iterative process}

The heart of the proof is the following  sequence of bootstrapping steps that alternate between Theorem~\ref{thm2sets}, in which we gradually increase the size of the family $\m F[X]$ lower bounded in~\eqref{eqsize2set},  and a combination of Theorem~\ref{thmregularity} with Lemma~\ref{lemtint}, in which we decrease the maximum size $q$ of sets in $\m S$ and increase the intersection parameter $t'$. For each $i = 1,2,\ldots, i_0$ (where $i_0$ will be defined below) we perform the following steps.

\medskip \noindent {\bf Step A(i).} We apply Theorem~\ref{thm2sets} to each subfamily $\ff'$ of $\ff[\m S^{(i-1)}]$ of size larger than $2^{-\sigma}a_t$, with $\m S^{(i-1)}$ playing the role of $\m G$, $t_1^{(i)} = t'^{(i-1)}$ playing the role of $t_1$,  $\ell^{(i)} = q^{(i-1)}$ playing the role of $\ell$, and 
\[
\lambda^{(i)}:= 2^{-\sigma} \frac{a_t}{a_{t_1^{(i)}}} \le \frac{|\m F'[\m S^{(i-1)}]|}{a_{t_1^{(i)}}}
\]
playing the role of $\lambda$. Note that since $\ff' \subset \ff[\m S^{(i-1)}]$, we have $\ff'[\m S^{(i-1)}]=\ff'$. Also note that the application of Theorem~\ref{thm2sets} requires $t_1 \geq 1$. Throughout the process, we will maintain $t_1^{(1)}\le t_1^{(2)}\le\ldots\le t_1^{(i_0)}\le t$ (as will be shown below). Hence, the theorem is applied with $t_1=t'^{(i-1)} \geq t'^{(0)}=t-\lceil t^{1/2} \rceil$, which is $\geq 1$ for all $t \geq 3$. The cases $t=1,2$ and the spreadness assumption on $\m A$ are treated separately at the end of the proof, as was written above.

By Theorem~\ref{thm2sets}, there exists a set $X$ such that $x=|X|$ satisfies $t_1^{(i)}\le x\le \ell_2^{(i)}$, where
\begin{equation*} 
\ell_2^{(i)}:= t_1^{(i)}
+4\Big(\frac{t_1^{(i)} (q^{(i-1)}-t_1^{(i)})^{2}}{R}\Big)^{1/3}+ \log_2\big((t_1^{(i)}+1)(\lambda^{(i)})^{-1}\big),
\end{equation*}
%where for the last summand we used that $\lambda^{(i)}\ge 2^{-\sigma}\big(\frac{k-t}{n-t}\big)^{t-t_1^{(i)}}$.
and denoting $\beta^{(i)} = \frac{|\m F'(X)|}{a_x}$, we have
%For this $X$, we get a relation between the sizes of the families $\ff'=\ff'[\m S^{(i-1)}]$
%and $\ff'[X]$. Putting  $|\ff'[X]| = \beta^{(i)}a_x,$ we have
\begin{equation}\label{boundonf}
|\m F'[\m S^{(i-1)}]|\le (t_1^{(i)}+1)e^{3\big(\frac{t_1^{(i)} (q^{(i-1)}-t_1^{(i)})^2}{R}\big)^{1/3}} \beta^{(i)}a_{t_1^{(i)}}.
\end{equation}

\medskip \noindent {\bf Step B(i).} We apply Theorem~\ref{thmregularity} to the family $\ff[\m S^{(i-1)}]$, with $\ell_1^{(i)} = t_1^{(i)}$ playing the role of $\ell_1$, $\ell^{(i)}_2$ as defined in Step A(i) playing the role of $\ell_2$, 
%(we will show that the assumptions of the theorem on weak spreadness of $\aaa$ are satisfied at the end of the proof), 
$$\eta^{(i)}=2^{-\sigma} a_t = \lambda^{(i)}a_{t_1^{(i)}}$$
playing the role of $\eta$,
$$\theta^{(i)} = (t_1^{(i)}+1)e^{3\big(\frac{t_1^{(i)} (q^{(i-1)}-t_1^{(i)})^2}{R}\big)^{1/3}} a_{t_1^{(i)}}$$
playing the role of $\theta$,
$r=R/2$, and
\begin{align}\label{boundq}
\begin{split}
q^{(i)} =& \ell_2^{(i)}+5\Big(\frac{t_1^{(i)} (q^{(i-1)}-t_1^{(i)})^{2}}{R}\Big)^{1/3}+ \log_2\big((t_1^{(i)}+1)(\lambda^{(i)})^{-1}\big)\\
=& t_1^{(i)}+ 9\Big(\frac{t_1^{(i)} (q^{(i-1)}-t_1^{(i)})^{2}}{R}\Big)^{1/3}+ 2\log_2\big((t_1^{(i)}+1)(\lambda^{(i)})^{-1}\big).
\end{split}
\end{align}
The assumption of Theorem~\ref{thmregularity} that for any $\m P\subset \ff[\m S^{(i-1)}]$ of size at least $\eta$ there exists a set $X$ with $\ell_1 \leq |X| \leq \ell_2$ such that $|\m P| \leq \frac{|\m P(X)|}{a_{|X|}} \cdot \theta$ is satisfied, due to Step~A(i) and the definition of $\theta^{(i)}$. We show at the end of the proof that the assumption on weak spreadness of $\m A$ is satisfied as well.
%Recall that, by Step A(i), for any family $\ff'\subset \ff[\m S^{(i-1)}]$ ($\ff'$ is playing the role of $\m P$) of size at least $\eta$ the bound \eqref{boundonf} holds. Thus, the condition on the size of $\ff'$ vs the size of $\ff'[X]$ in Theorem~\ref{thmregularity} is automatically satisfied due to our choice of $\theta^{(i)}$. 
%The condition on $r$ is also clearly satisfied, and thus 
Thus, Theorem~\ref{thmregularity} can indeed be applied. By the theorem, there exists a family $\m S^{(i)}$ of sets, each of size at most $q^{(i)}$, and a family $\m R^{(i)} \subset \m F[\m S^{(i-1)}]$, such that:  
\begin{itemize}
    \item[(i)] $\ff[\m S^{(i-1)}]= \m R^{(i)}\sqcup \bigsqcup_{S\in \m S^{(i)}} \m F[\m S^{(i-1)}]_S[S]$;
    \item[(ii)] For any $S\in \s^{(i)}$ and the family $\m F[\m S^{(i-1)}]_S$, the family $\m F[\m S^{(i-1)}]_S(S)$ is $(R/2)$-spread;
    \item[(iii)] $|\m R|\le \max\big\{\eta^{(i)}, \theta^{(i)} \cdot (1/2)^{q^{(i)}+1-\ell_2^{(i)}}\big\}$.
  \end{itemize}
We claim that the choice of $q^{(i)}$ guarantees that 
$$|\m R^{(i)}|\le 2^{-\sigma}a_t.$$ 
%Recall Theorem~\ref{thmregularity} (iii). The first expression in the maximum, $\eta^{(i)}$, immediately gives exactly this upper bound. The second expression in the maximum in the RHS of Theorem~\ref{thmregularity} (iii) is at most 
Indeed, we have $\eta^{(i)}=2^{-\sigma}a_t$ and 
$$
\theta^{(i)} \cdot (1/2)^{q^{(i)}+1-\ell_2^{(i)}}\le 2^{-\log_2((t_1^{(i)}+1)(\lambda^{(i)})^{-1})}\cdot (t_1^{(i)}+1)a_{t_1^{(i)}} = 2^{-\sigma} a_t.$$ 
In the first inequality, we used the definition of $q^{(i)}$ and the inequality $2^5>e^3.$ In the second inequality, we used the definition of $\lambda^{(i)}$.
%Theorem~\ref{thmregularity} also yields a new spread approximation of $\ff$, that is, a family $\m S^{(i)}$ of uniformity at most $q^{(i)}$. 

\medskip At this stage, we apply Lemma~\ref{lemtint} to $\ff$ and $\m S^{(i)}$, with $r=\frac{R}2$ and  $t'^{(i)}$ defined implicitly by the following equation:
\begin{equation}\label{boundt} t'^{(i)} = \min \{ t-\Big\lfloor 24\cdot \frac{q^{(i)}-t'^{(i)}+1}{r}\Big\rfloor, t\}.
\end{equation}
The lemma can indeed be applied, as Assumption~\eqref{eqint1} is satisfied by the definition of $t'^{(i)}$ and Assumption~\eqref{eqint2} is satisfied by the definition of $R$ (note that the values of $R,r,k$ are not changed throughout the process).
%the lemma with these parameters. 
%Condition \eqref{eqint2} is satisfied again (as the values of $R,r,k$ are not changed throughout the process). 
%Also, the definition of $t'^{(i)}$ is chosen  so that~\eqref{eqint1} is satisfied.
By Lemma~\ref{lemtint}, the family $S^{(i)}$ is $t'^{(i)}$-intersecting.

\subsubsection{The termination of the process}
As was written above, throughout the process, the value $q^{(i)}$ that upper bounds the size of sets in the approximating family $\m S^{(i)}$ decreases, while the value $t^{(i+1)}_1=t'^{(i)}$ that lower bounds the size of the intersection of any two elements of $\m S^{(i)}$ increases (we prove this formally below). We continue the process as long as the condition
\begin{equation}\label{qdec} q^{(i)}-t'^{(i)}_1\ge t^{1/2}+4\sigma+4\log_2(t+1)
\end{equation}
holds, which means that $q^{(i)}$ and $t'^{(i)}=t^{(i+1)}_1$ are `not sufficiently close' yet. At the smallest value $i_0$ for which~\eqref{qdec} fails, we stop the procedure, and set $\m S:=\m S^{(i_0)}$. 

We claim that $t'^{(i_0)}=t$, which means that the approximating family $\m S$ is $t$-intersecting. Indeed, as~\eqref{qdec} does not hold for $i_0$, in the definition of $t'^{(i_0)}$ in~\eqref{boundt} we have 
$$\Big\lfloor 24\cdot \frac{q^{(i_0)}-t'^{(i_0)}+1}{r}\Big\rfloor\le \Big\lfloor 24\cdot \frac{t^{1/2}+4\sigma+4\log_2(t+1)+1}{r}\Big\rfloor = 0,$$
since $r = R/2\ge 2^{29}t^{1/2}+100\sigma\ge 25(t^{1/2}+4\sigma+4\log_2(t+1))$, and therefore, $t'^{(i_0)}=t$. 

We now claim that the approximating family $\m S$ satisfies all the requirements of the theorem. 
To this end, we will show below that during the process, the values $t-t'^{(i)}$ and $q^{(i)}-t'^{(i)}$ decrease at an exponential rate, and as a result, the process terminates after at most $2\log_2(Rt+\sigma)$ steps. %Furthermore, at the end of the process we have $t'^{(i_0)}=t$, which means that $\m S$ is $t$-intersecting. 
Due to the failure of~\eqref{qdec}, the maximum size of a set in $\m S$ is at most $t+t^{1/2}+4\sigma+4\log_2(t+1)$.  
Consider the size of the remainder $\m R=\m F \setminus \m F[\m S]$. Note that at every application of Theorem~\ref{thmregularity}, we have $|\m R^{(i)}|\le 2^{-\sigma}a_t$. Assuming the number of steps is at most $2\log_2 (Rt+\sigma)$, this implies that 
\[
|\m R| = |\cup_{i=1}^{i_0} \m R^{(i)}| \le 2^{-\sigma}a_t\cdot 2\log_2 (Rt+\sigma),
\]
which implies the bound on the size of $\m R$ in the statement of the theorem. 
%\nathan{In the statement of the theorem, we had the constant $4$ instead of $2$. I fixed this in the statement of the theorem.} 
This shows that the family $\m S=\m S^{(i_0)}$ indeed satisfies all the requirements of the theorem.

\medskip Therefore, it only remains to show that during the process $q$ decreases and $t'$ increases and that the process ends after at most $2\log_2(Rt+\sigma)$ steps.
%, and that at the end of the process we have $t'^{(i_0)}=t$.
To show this, we first
%analyze Equations~\eqref{boundq} and~\eqref{boundt}. 
prove the following auxiliary claim.
\begin{claim}
    Let $t,t^{(i)}_1,q^{(i)}_1,\sigma$ be as defined above. 
    If the condition~\eqref{qdec} 
    is satisfied for $i-1$ (namely, if
    $q^{(i-1)}-t'^{(i-1)}_1\ge t^{1/2}+4\sigma+4\log_2(t+1)$), then 
    \[
        q^{(i)}-t^{(i)}_1\le \tfrac 23 (q^{(i-1)}-t^{(i)}_1).    
    \]
\end{claim}
%It should be clear that $t'^{(i)}$ does not decrease as $i$ grows, provided the value of $q^{(i)}$ does not increase as $i$ grows.
Intuitively, this means that as long as $q^{(i-1)}-t^{(i)}_1$ is `large', $q^{(i)}$ is much closer to $t^{(i)}_1$ than $q^{(i-1)}$.

\begin{proof}[Proof of the claim]
    Consider the definition of $q^{(i)}$ in~\eqref{boundq}. First, since by the assumption of the theorem, $R\ge 2^{30}t^{1/2}$, we have
\begin{equation}\label{Eq:Aux3.1}
    9\Big(\tfrac{t_1^{(i)} (q^{(i-1)}-t_1^{(i)})^{2}}{R}\Big)^{1/3} \le \tfrac 1{100} \big(t^{1/2}(q^{(i-1)}-t_1^{(i)})^{2}\big)^{1/3}\overset{\eqref{qdec}}{\le} \tfrac 1{100} (q^{(i-1)}-t_1^{(i)}).
\end{equation} 
Next, by the definition of $\lambda^{(i)}$ we have 
\begin{equation}\label{Eq:Aux3.2}
2\log_2\big((t_1^{(i)}+1)(\lambda^{(i)})^{-1})=2\log_2(t_1^{(i)}+1)+2\sigma +2\log_2 \frac {a_{t_1^{(i)}}}{a_t},
\end{equation}
and by the assumption of the theorem that $\m A$ is `not too weakly spread' we have
\begin{equation}\label{Eq:Aux3.3}
    \log_2 \frac {a_{t_1^{(i)}}}{a_t}\le (t-t_1^{(i)})\log_2 R_1.
\end{equation}
Note that this is the only place in the proof where we use this assumption on $\m A$.

By the definition of $t'^{(i-1)}$ in~\eqref{boundt} and the identification $t'^{(i-1)} = t_1^{(i)}$, we have 
\begin{equation}\label{Eq:Aux3.3.5}
t-t_1^{(i)} \le 24 \frac{q^{(i-1)}-t_1^{(i)}+1}{r}\le \frac 1{200}\frac{q^{(i-1)}-t_1^{(i)}}{\log_2 R_1},    
\end{equation}
where the second inequality holds since by the assumption of the theorem, we have $r = R/2\ge 2^{29}\log_2 R_1.$ By~\eqref{Eq:Aux3.3} and~\eqref{Eq:Aux3.3.5}, we have 
\begin{equation}\label{Eq:Aux3.3.6}
2\log_2 \frac {a_{t_1^{(i)}}}{a_t}\le \frac{q^{(i-1)}-t_1^{(i)}}{100}.    
\end{equation}
Substituting  Equations~\eqref{Eq:Aux3.1},~\eqref{Eq:Aux3.2}, and~\eqref{Eq:Aux3.3.6} into the definition of $q^{(i)}$ in~\eqref{boundq} and using the inequality $t_1^{(i)} \leq t$, we get
\begin{align*}
q^{(i)}-t^{(i)}_1&\overset{\eqref{boundq}}{=}  9\Big(\frac{t_1^{(i)} (q^{(i-1)}-t_1^{(i)})^{2}}{R}\Big)^{1/3}+ 2\log_2\big((t_1^{(i)}+1)(\lambda^{(i)})^{-1}\big)\\
&\overset{\eqref{Eq:Aux3.1},\eqref{Eq:Aux3.2}}\le \frac 1{100} (q^{(i-1)}-t_1^{(i)}) + 2\log_2(t_1^{(i)}+1)+2\sigma + 2\log_2 \frac {a_{t_1^{(i)}}}{a_t}\\
&\overset{\eqref{Eq:Aux3.3.6}}\le \frac 1{50} (q^{(i-1)}-t_1^{(i)}) + 2\log_2(t_1^{(i)}+1)+2\sigma\\
&\overset{\eqref{qdec}}{\le} \frac 1{50} (q^{(i-1)}-t_1^{(i)}) + \frac 1{2} (q^{(i-1)}-t_1^{(i)})<\frac 2{3} (q^{(i-1)}-t_1^{(i)}),
\end{align*}
which proves the assertion of the claim.
\end{proof}
The claim implies that as long as \eqref{qdec} holds for $i-1$, we have $q^{(i)}<q^{(i-1)}$, %\ohad{since $q^{(i)}<t_1^{(i)}$} \nathan{I guess you meant $q^{(i)}>t_1^{(i)}$; this is indeed the reason.}\ohad{sure}
which in turn implies by the definition of $t'^{(i)}=t_1^{(i+1)}$ in~\eqref{boundt} that $t_1^{(i+1)} \geq t_1^{(i)}$, and thus, we have 
\[
q^{(i)}-t_1^{(i+1)}\le q^{(i)}-t_1^{(i)}< \tfrac 23 (q^{(i-1)}-t_1^{(i)}).
\]
This means that the difference  $q^{(i)}-t_1^{(i+1)}$ decays exponentially with $i$. In particular, we reach the point where \eqref{qdec} does not hold any longer in at most $2\log_2 (Rt+\sigma)$ steps, since $q^{(0)}-t^{(1)}_1\le \lceil t^{1/2} \rceil+t\log_2 R+\sigma\le t\log_2(2R)+\sigma$ %\nathan{It seems that $t^{(0)}$ should be replaced by $t_1^{(1)}$. Also, in the right hand side, it seems that an additive summand of $\lceil t^{1/2} \rceil$ is missing; the same also in the left hand side of the next inequality. Though, this seems easy to fix.}\andrey{Done} 
and 
\[
(t\log_2(2R)+\sigma)\cdot \left(\tfrac{2}{3}\right)^{2\log_2 (Rt+\sigma)}<1.
\]
Thus, $i_0 \leq 2\log_2 (Rt+\sigma)$, as was claimed. 
%This completes the proof of the theorem, modulo a few `remainders' that will be filled up below.

%\ohad{Isnt the right place for that is before this subsection?} \nathan{Andrey wrote at the beginning of the proof that this part is postponed to the end. I think it's fine.} 

\subsubsection{Completing the proof.}
In order to complete the proof of the theorem, we have to fill in a few claims that were left for the end.

\medskip \noindent \emph{Spreadness of the family $\aaa$.}
We show that the spreadness requirements on the family $\aaa$ in the applications of Theorems~\ref{thm2sets} and~\ref{thmregularity} throughout the proof are satisfied. Theorem~\ref{thm2sets} requires weak $(R,t_1)$-spreadness of $\aaa$, and we have applied it with values of $t_1$ in the range $[t-\lceil t^{1/2}\rceil,t].$ Theorem~\ref{thmregularity} requires weak $(R,x)$-spreadness for the values of $x\in [\ell_1,\ell_2].$ The value of $\ell_1$ in our applications is always at least $t-\lceil t^{1/2}\rceil$. The value of $\ell_2$ we used is $q^{(i)}$ for $i\ge 1$. Using the trivial bound $q^{(i-1)}\le k$, the bound \eqref{boundq} and the inequalities~\eqref{Eq:Aux3.1}, \eqref{Eq:Aux3.2} and~\eqref{Eq:Aux3.3} above, we obtain
$$q^{(i)}\le t+ \frac 1{100} \big(t^{1/2}(k-t+\lceil t^{1/2}\rceil)^{2}\big)^{1/3}+2\sigma+2\lceil t^{1/2}\rceil \log_2 R_1+2\log_2(t+1).$$
By the spreadness assumption on $\aaa$ in the statement of the theorem, $\aaa$ is weakly $(R,x)$-spread for all $x$ in the said range.

\medskip \noindent \emph{The proof for $t=1,2$.}
As written in the beginning of the proof, the above argument works for $t\ge 3$. For $t=1,2$, the argument cannot be carried out directly, since it requires applying Theorem~\ref{thm2sets} with $t'^{(0)}=t-\lceil t^{1/2} \rceil =0$ in place of $t_1$, while Theorem~\ref{thm2sets} assumes $t_1 \ge 1$.  
%\ohad{Why? we should talk about it here, not in page 3.} 
Instead, we apply the same procedure, but with $t_1^{(i+1)}=t'^{(i)}=t$ for all $i$. 

Inspection of the proof shows that the only places where this change makes a difference are the applications of Lemma~\ref{lemtint}, where we have to check that the assumptions~\eqref{eqint1},~\eqref{eqint2} hold for $t=1,2$, where $t'^{(i)}$ is replaced by $t$. The assumption~\eqref{eqint2} is not affected by the change. As for~\eqref{eqint1}, since during the iterative process $q^{(i)}$ decreases, we have 
%The only condition to check is that \eqref{eqint1} always holds.\ohad{need to explain why its the only condition, but seems true} We have 
%$q^{(i)} \le \max(q^{(0)},q^{(i_0)}) \le \max\{4\log 2 R+\sigma, 4+4\sigma+4\log_2(3)\}$, 
$q^{(i)} \le q^{(0)} \le 4\log_2 R+\sigma,$ 
%\ohad{If $t$ was replaced elsewhere it should be replaced in the log, also maybe its $4+4\sigma+4\log_2(t+1)$}, 
and by the assumption on $R$, this clearly implies
%it is easy to check that then 
$r = R/2>24 q^{(i)}$. Thus,~\eqref{eqint1} holds for $t'^{(i)} = t$. 
%The rest of the argument stays the same.\\

\medskip \noindent \emph{The proof for $\frac{1}{2}<\alpha \leq 1$.} To complete the proof of the theorem, we show that if $R\ge 2^{30}t^{\alpha}$ for some $\alpha\in (1/2,1]$,
then the bound on the sizes of sets in $\m S$ improves to $t+t^{1-\alpha}+4\sigma+4\log_2(t+1),$ without affecting the bound on the size of the remainder $\m R$.

%Let us now prove the `more generally' part of the statement of the theorem. 
First, we perform the same steps as for $\alpha=1/2$, which can be done since the assumptions on $R$ are valid. % Assume that $n\ge Ct^{\alpha}s^{\beta}(k-t)\log_2 \frac{n-t}{k-t}$ with some $\alpha,\beta\ge 1/2$. Then the bound \eqref{boundq} improves to
%\begin{multline}\label{boundq2}
%  q^{(i)}\le   t_1^{(i)}+2t^{(1-\alpha)/3}s^{-\beta/3}\Big(\frac{q^{(i-1)}-t_1^{(i)}} {100\log_2\frac{n-t}{k-t}}\Big)^{2/3} \\ +2\sigma+\log_2(8s^4t^2)+2(t-t_1^{(i)})\log_2 \frac{n-t}{k-t}
%\end{multline}
%Thus, we may perform exactly the same analysis as above (only using the $\alpha=\beta=1/2$ case),
We obtain an approximating family $\m S^{(i_0)}$, with  $q^{(i_0)}\le t+t^{1/2}+4\sigma+4\log_2(t+1)$ and $t'^{(i_0)} = t$. 

Then we continue applying the iterative process with $t^{(i)}=t$ for all $i$ (instead of the above definition), as long as the following condition holds:
\begin{equation}\label{qdec2} 
q^{(i-1)}-t\ge t^{1-\alpha}+4\sigma+4\log_2(t+1).
\end{equation}
By a virtue of the above argument, with~\eqref{qdec2} in place of~\eqref{qdec} and using the inequality 
$$9\Big(\frac{t (q^{(i-1)}-t)^{2}}{R}\Big)^{1/3} \le \frac 1{100} \big(t^{1-\alpha}(q^{(i-1)}-t)^{2}\big)^{1/3}\overset{\eqref{qdec2}}{\le} \frac 1{100} (q^{(i-1)}-t),$$
we obtain that as long as~\eqref{qdec2} holds, we have $q^{(i)}-t\le \frac 23 (q^{(i-1)}-t)$. Hence, after at most $2\log_2 t$ steps 
%up until step $i_1$, and for which we use the bound
%$$9\Big(\frac{t (q^{(i-1)}-t)^{2}}{R}\Big)^{1/3} \le \frac 1{100} \big(t^{1-\alpha}(q^{(i-1)}-t)^{2}\big)^{1/3}\overset{\eqref{qdec2}}{\le} \frac 1{100} (q^{(i-1)}-t).$$
%Otherwise,
%the same analysis works, and we get $q^{(i)}-t\le \frac 23 (q^{(i-1)}-t)$. After at most $2\log_2 t$ steps 
we reach the point where~\eqref{qdec2} does not hold any longer. 

At this point, we terminate the process and obtain a $t$-intersecting approximating family $\m S$ consisting of sets of size at most  $t+t^{1-\alpha}+4\sigma+4\log_2(t+1)$, as required.
Finally, we can bound the size of the remainder $\m R$ as it was done above, modifying the number of steps to $2\log_2 t+2\log_2 (Rt+\sigma)\le 4\log_2 (Rt+\sigma)$. This completes the proof of the theorem.
\end{proof}

\section{The Peeling Simplification}
\label{sec:peeling_off}

In this section we present our third simplification, called `peeling simplification'. This is an iterative process, which given a family of sets, gradually removes from it `spread' parts, and eventually, replaces the family with a `kernel' whose size can be bounded efficiently.
As was mentioned above, a similar procedure was introduced by Kupavskii and Zakharov~\cite{KZ22} and refined by Kupavskii~\cite{Kupavskii26EJC}. In fact, somewhat similar methods can be traced back to works of Frankl and F\"{u}redi in the 80's; see a detailed discussion in~\cite[Section~1.7]{kupavskii25survey}. 
%\nathan{Andrey, please check whether the sentences I wrote here (based on your comment) are fine.}\andrey{Yes, this is fine!}
%\nathan{I think we will have to explain what is the relation between this procedure and similar procedures used in the `spread approximation' papers.}\andrey{We used a (worse version of) this procedure in our initial paper with Dima Zakharov. I later refined it in \cite{Kupavskii26EJC}. Here, we are using some parts of it. (Technically, what corresponds the most to, say, the procedure in \cite{Kupavskii26EJC} is the analysis we do in Case 3 of the proof, where we change the spreadness parameter. Somewhat similar things corresponding to fixed spreadness $r$ (but with some other notions, not spreadness) were used  back in the 80s by Furedi and Frankl. We can refer to my recent survey ``Delta-system method - a survey'', where I discuss these things in detail.}

\medskip \noindent \emph{The $r$-simplification process.} Given a family of sets $\m X \subset \m P([n])$, and a parameter $r>0$, we perform the following iterative process:
\begin{itemize}
\item Set $\m X_1:= \m X$, and perform the following steps for $i=1,2,\ldots$, until we reach the termination condition stated below.

\item At the $i$'s step, find a set $Y_{i}$ and a subfamily  $\m Z_i\subset \m X_i$ such that $\m Z_i(Y_i)$ is $r$-spread and contains at least two elements. Set $\m X_{i+1}:=\m X_i\setminus \m X_i[Y_i] \cup \{Y_i\}$.

\item If no such $Y_i$ and $\m Z_i$ can be chosen, stop the iterative process and set $\m S':=\m X_i$.  

\item Set $\m S$ to be the family of inclusion-minimal sets in $\m S'$. 
\end{itemize}
As by construction, $|\m X_{i+1}| \leq |\m X_i|-1$ for all $i$, the process ends after at most $|\m X|$ steps. The family $\m S$ is called \emph{an $r$-simplification of $\m X$}. We note that the resulting family $\m S$ depends on the order in which subfamilies are removed. We take one of the possible orderings arbitrarily.

\medskip \noindent \emph{Properties of the $r$-simplification process.} The following lemma proves several properties of the process that will be used in the sequel.
\begin{lem}\label{lemsimplsimple} Let $n,t,p,r \in \mathbb{N}$ be such that $r>p-t+1$. Let $\m X \subset \m P([n])$ be a $t$-intersecting family in which each set is of size at most $p$. Let $\m S$ be an $r$-simplification of $\m X$, and denote the family of all $i$-element sets in $\m S$ by $\m S^{(i)}$. Then:
\begin{enumerate}
    \item $\m X = \m X[\m S]$;

    \item $\m S$ is $t$-intersecting; and

    \item For any $i \geq t$, $|\m S^{(i)}|\le {i\choose t} \cdot r^{i-t}$.
\end{enumerate}
Furthermore, if $\m X,\m W \subset \m P([n])$ are cross $t$-intersecting families, both consisting of sets of size at most $p$, and $\m S_{\m X},\m S_{\m W}$ are their respective $r$-simplifications, then each of the families $\m X,\m S_{\m X}$ cross $t$-intersects each of the families $\m W, \m S_{\m W}$.
\end{lem}
%Note that, for any $B\in \m X$, the family $\m X(B)$ is $r$-spread for the reason that $B$ is not a subset of any other set in $\m X$. \nathan{Why cannot $B$ be a subset of another set in the family? Also, it is not clear why this remark is needed.} Thus such $B$ may be included in $\m S$ as long as it has no proper subset that gives rise to an $r$-spread family. \ohad{I didnt understand this paragraph at all} \andrey{We can safely remove this paragraph, it's not needed.}
\begin{proof}
  To prove the assertion~(1), we show that for each $i$, we have $\m X = \m X[\m X_i]$. This is sufficient, as this implies that $\m X[\m S]=\m X[\m S']=\m X[\m X_{i_0}]$ for $i_0$ being the step at which the process is terminated.

  To show that $\m X = \m X[\m X_i]$ for all $i$, note that at the $i$'s step of the process, we replace $\m X_i$ by $\m X_{i+1}:=\m X_i\setminus \m X_i[Y_i] \cup \{Y_i\}$. As all sets removed from $\m X_i$ contain $Y_i \in \m X_{i+1}$, we have $\m X_i=\m X_i[\m X_{i+1}]$, and thus, $\m X[\m X_i]=\m X[\m X_{i+1}]$. The assertion follows by induction on $i$.

  \smallskip The assertion~(2) follows from the `Furthermore' assertion on a pair $\m X,\m W$ of cross $t$-intersecting families and their corresponding $r$-simplifications $\m S_{\m X}, \m S_{\m W}$. Hence, we prove the `furthermore' assertion. 

  First, we prove that $\m S_{\m X}$ and $\m W$ are cross $t$-intersecting. Assume on the contrary they are not, and let $j$ be the earliest step of the simplification process that led from $\m X$ to $\m S_{\m X}$ such that $\m X_{j+1}$ and $\m W$ are not cross $t$-intersecting. As in the $j$'s step, we replace $\m X_{j}$ by $\m X_j\setminus \m X_j[Y_j] \cup \{Y_j\}$ for some $Y_j$, there exists $W \in \m W$ such that $|W\cap Y_j|<t$. 
  Denote 
  \[
  x_1:=|W\cap Y_j|<t \qquad \mbox{and} \qquad x_2:= |W\setminus Y_j|,
  \]
  and note that as $|W| \leq p$, we have $x_1 \leq p-x_2$, and thus,
  \begin{equation}\label{Eq:Aux_Simplification0}
      x_1 \leq \min \{p-x_2,t-1\}.
  \end{equation}
  By the definition of the $r$-simplification process, there is a family $\m Z_j\subset \m X_j$ such that $\m Z_j(Y_j)$ is $r$-spread (where $r>p-t+1$). This implies that for any $V \subset [n]$ of size $\le p-t+1$, there exists $Z \in \m Z_j(Y_j)$ such that $V \cap Z=\emptyset$, as otherwise, we would have 
  \begin{equation}\label{Eq:Aux_Simplification0.5}
  |(\m Z_j(Y_j))[\{v\}]| \geq \tfrac{1}{p-t+1}|\m Z_j(Y_j)|
  \end{equation}
  for some $v \in V$, contradicting the $r$-spreadness of $\m Z_j(Y_j)$. Applying this to any subset $V$ of size $\min \{p-t+1,x_2\}$ of $W\setminus Y_j$, we deduce that there exists $Z \in \m Z_j(Y_j)$ such that $Z \cap V=\emptyset$, and hence, 
  $|Z \cap (W \setminus Y_j)| \leq \max \{x_2-(p-t+1),0\}$. Let $Z' \in Z_j[Y_j]$ be such that $Z' \supset Z$. Since $|W\cap Y_j|=x_1$, we have 
  \begin{align*}
  |Z' \cap W| &= |Z \cap (W \setminus Y_j)| + |(Z' \setminus Z) \cap (W \cap Y_j)| \\
  &\leq \max \{x_2-(p-t+1),0\} + x_1 \leq t-1,     
  \end{align*}
  where the last inequality holds by~\eqref{Eq:Aux_Simplification0}. As $Z' \in \m X_j$, this implies that $\m X_j$ and $\m W$ are not $t$-intersecting, a contradiction. This shows that $\m S_{\m X}$ and $\m W$ are cross $t$-intersecting, as asserted.

  To show that $\m S_{\m X}$ and $\m S_{\m W}$ are cross $t$-intersecting, we apply the above argument to the cross $t$-intersecting families $\m S_{\m X}$ and $\m W$, examining   
  the $r$-simplification process that led from $\m W$ to $\m S_{\m W}$. %\ohad{$S_{\m W}$?}\nathan{Yes}. 
  The cross $t$-intersection assertion about the pairs of families $(\m X, \m W)$ and $(\m X, \m S_{\m W})$ follows by Assertion~(1).

\smallskip To prove the assertion~(3), it is clearly sufficient to show that the same bound holds for the family $\m S'$ that contains $\m S$. We use the $t$-intersection property of $\m S'^{(i)}$ that follows from assertion~(2). Let $F\in \m S'^{(i)}$. Each $F' \in \m S'^{(i)}$ intersects $F$ in at least $t$ elements, and hence, by averaging, there exists a subset $X\in {F\choose t}$ such that 
\begin{equation}\label{Eq:Aux_Simplification1}
    |\m S'^{(i)}(X)|\ge {i\choose t}^{-1}|\m S'^{(i)}|.
\end{equation}
Next, we show that $|\m S'^{(i)}(X)|\le r^{i-t}$, which will conclude the proof. 

The family $\m S'^{(i)}(X)$ consists of sets of size $i-t$, and there is no set $Y \subset [n]$ such that the family $(\m S'^{(i)}(X))(Y)$ is $r$-spread and contains at least two elements. Indeed, should such a $Y$ exist, by the definition of $r$-simplification, the process would not terminate with some $\m X_{i_0}= \m S'$, but rather, more peeling steps would be performed.
%the definition of $\m S$ implies that we should have replaced all the sets in $\m S^{(i)}(X)$ by $X\cup Y$. 
Therefore, Observation~\ref{obs13} implies that $|\m S'^{(i)}(X)|\le r^{i-t}$. Combining this with~\eqref{Eq:Aux_Simplification1} completes the proof.
\end{proof}

\section{Finding a Simple Sub-Structure Within $\m F$, for a Small $t$}
\label{sec:Small_t}

In this section, we study $(n-t)$-intersecting families, for $t \leq n^{\epsilon}$. Theorem~\ref{Thm:Main} asserts that in this range, the maximum size of an $(n-t)$-intersecting family is attained by the family 
\[
\m F_{n,n-t,t/2}=\{\sigma \in S_n: |\mathrm{Moving}(\sigma)| \leq t/2\}
\]
for an even $t$ and by the family
\[
\m F_{n,n-t,(t-1)/2}=\{\sigma \in S_n: |\mathrm{Moving}(\sigma) \cap \{1,2,\ldots,n-1\}| \leq \tfrac{t-1}{2}\}
\]
for an odd $t$. We consider an $(n-t)$-intersecting family $\m G$ of a maximum size and transform it to a family $\m F$ of sets, using the representation of permutations by sets presented in Section~\ref{sec:PermutationstoSets}. Decompose $\m G = \cup_m \m G_m$, where $\m G_m$ is the set of permutations in $\m G$ that have exactly $m$ moving points, and denote by $\m F_m$ the family of sets that corresponds to $\m G_m$. We show that there exists $m$ such that the family $\m F_m$ is `large' and $(2m-t)$-intersecting and most of the sets in $\m F_m$ contain the same set $F$ of size $2m-t$ that has a very special structure. Formally, we prove the following theorem.
\begin{thm}\label{thmfm}
    For any $\epsilon \leq 0.01$, there exists $n_0 \in \mathbb{N}$ such that the following holds for all $n \geq n_0$ and all $3 \leq t \leq n^{\epsilon}$. 
    Let $\m G \subset S_n$ be a maximum-size $(n-t)$-intersecting family of permutations. Let $\m F \subset \m P(\mb [n] \cup \{(i,j):i,j \in [n], i \neq j\})$ %\ohad{$\m F$ is a family of subsets} \nathan{Right!} 
    be the family of sets that corresponds to $\m G$ in the representation of permutations by sets presented in Section~\ref{sec:PermutationstoSets}. Then there exists $t/2 \leq m \leq t$ such that:
    \begin{itemize}
        \item $\m F_m:=\{F_{\sigma}=(D_{\sigma},E_{\sigma}) \in \m F: |D_{\sigma}|=m\}$ satisfies $|\m F_m| \geq n^{-2\epsilon} |\m F|$;

        \item $\m F_m$ is $(2m-t)$-intersecting.
        
        \item $|\m F_m(F)| \geq (1-n^{-\epsilon})|\m F_m|$, where $F$ is a set of size $2m-t$ of the form $F = \{D,(M,\sigma_0(M))\}$, where $\sigma_0 \in S_n$, $M\subset D$, $\sigma_0(M) \subset D$, $|D|=m-\lfloor t/2\rfloor$, and $|M|=m-\lceil t/2\rceil$.
    \end{itemize} 
\end{thm}
Note that in the statement of the theorem, we make the assumption $t \geq 3$. 
For $t=1,2$, the characterization of $(n-t)$-intersecting families is easy: The only $(n-1)$-intersecting families are $1$-element families, and the remaining $(n-2)$-intersecting families consist of two permutations that differ by a transposition.

%There is no need to consider the range $t<c_0$, as in this range (i.e., $t$-intersecting families in $S_n$ for $t>n-c$), the assertion of Theorem~\ref{Thm:Main} was proved for all $n \geq n_0(c)$ by Deza and Frankl~\cite{DF77}.

Before we prove the theorem, let us examine the `special structure' admitted by most elements of $\m F_m$. Assume that $t$ is even (the situation for an odd $t$ is only slightly more complex). The family of sets $\m F_m[F]$ corresponds to the family $\m G'$ of all permutations in $\m G_m$ that agree on a specific set $M$ of size $m-\frac{t}{2}$ with a bijection $\sigma'_0:M \to M$. Putting aside the set $M$ on which all these permutations agree, $\m G'$ corresponds to the family of all permutations on $[n] \setminus M$ that have exactly $t/2$ moving points. For $M = \emptyset$, this family is almost the same as the family $\m F_{n,n-t,t/2}$ (the latter also contains all permutations with less than $t/2$ moving points, whose total amount is negligible). Hence, Theorem~\ref{thmfm} asserts that $\m F_m$, which takes up a large part of $\m F$, essentially corresponds to a `copy' of $\m F_{n,n-t,t/2} \subset S_n$. In Section~\ref{sec:Completing_the_Proof} we will show how this allows us to deduce that $\m G \subset \m F_{n,n-t,t/2}$, as asserted in Theorem~\ref{Thm:Main}.    

\begin{proof}[Proof of Theorem~\ref{thmfm}]
Let $n,t,\epsilon$ be as in the statement of the theorem, let $\m G \subset S_n$ be a maximum-size $(n-t)$-intersecting family of permutations, and let $\ff = \cup_{i=0}^t\ff_i$ be the corresponding family of sets. The proof consists of several steps.

\medskip \noindent \textbf{Step~1: Finding a large $(2m-t)$-intersecting family $\m F_m \subset \m F$}. Recall that by Lemma~\ref{lemadecay}(v), for each $\lceil t/2 \rceil \leq i \leq t$, there exists an $(n-t)$-intersecting family $\m G' \subset S_n$ of size $a^{(i)}_{2i-t}$. Hence, by the maximality of $|\m G|$, we have
\begin{equation}\label{Eq:Aux-Small-t1}
|\ff|\ge \max_{i\in \{\lceil t/2\rceil,\ldots, t\}} a^{(i)}_{2i-t}.    
\end{equation}
Consider the families $\{\m F_i\}_{i=0,\ldots,t}$. We claim that the size of all $\{\m F_i\}_{i=0,\ldots,\lceil t/2 \rceil -1}$ is `small' compared to $|\m F|$. Indeed, as for each $i$ we have $\m F_i \subset \m A_i$ and $|\m A_i|=a^{(i)}_0$, Lemma~\ref{lemadecay}(vi) implies 
\begin{equation}
\label{Eq:Aux-Small-t1.1}
\sum_{i=0}^{\lceil t/2 \rceil-1} |\m F_i| \leq \sum_{i=0}^{\lceil t/2 \rceil-1} a_0^{(i)} \le \tfrac{6}{n} \cdot\,a_0^{\lceil t/2 \rceil}.    
\end{equation}
%and similarly, $\sum_{i=0}^{\lceil t/2 \rceil-1} |\m F_i| \leq \tfrac{6}{n} \cdot a_0^{\lceil t/2 \rceil}$. 
If $t$ is even, we have $|\ff|\ge a^{(t/2)}_{0}$ by~\eqref{Eq:Aux-Small-t1}, and thus, $\sum_{i=0}^{(t/2)-1} |\m F_i| \le \tfrac{6}{n} |\m F|$. 

For an odd $t$, we need a somewhat more delicate argument. Observe that 
$a_1^{\lceil t/2 \rceil}+\sum_{i=0}^{\lceil t/2 \rceil-1} a_0^{(i)}$ is the size of the family $\m F_1 \cup \m F_2 \subset \mb X$, where $\m F_1:=\{F=(D,E):|D| \leq \lceil {t/2} \rceil-1\}$ and $\m F_2:=\{F=(D,E):|D|=\lceil {t/2} \rceil, 1 \in D\}$. $\m F_1 \cup \m F_2$ clearly corresponds to an $(n-t)$-intersecting family of permutations, and hence, by the maximality of $\m G$, we have
\begin{equation}\label{Eq:Aux-Small-t1.11}
|\m F| \geq 
a_1^{\lceil t/2 \rceil}+\sum_{i=0}^{\lceil t/2 \rceil-1} a_0^{(i)}.
\end{equation} 
By Lemma~\ref{lemadecay}(vi,viii) , we have 
\[
\sum_{i=0}^{\lceil t/2 \rceil-1} a_0^{(i)} \leq (1+\tfrac{6}{n})a_0^{(\lceil t/2 \rceil -1)} \leq \frac{1+\tfrac{6}{n}}{(1- \tfrac{\lceil t/2 \rceil -1}{n})\cdot (\lceil t/2 \rceil -1)}a_{1}^{(\lceil t/2 \rceil)}. 
\]
Hence,~\eqref{Eq:Aux-Small-t1.11} implies
\[
|\m F| \geq a_1^{\lceil t/2 \rceil}+\sum_{i=0}^{\lceil t/2 \rceil-1} a_0^{(i)} \geq \left(1+
\frac{(1- \tfrac{\lceil t/2 \rceil -1}{n})\cdot (\lceil t/2 \rceil -1)}{1+\tfrac{6}{n}
} \right)\sum_{i=0}^{\lceil t/2 \rceil-1} a_0^{(i)}.
\]
This implies that for all $t \geq 3$ and all sufficiently large $n$, we have 
\begin{equation}\label{Eq:Aux-Small-t1.2}
\sum_{i=0}^{\lceil t/2 \rceil-1} |\m F_i| \leq \sum_{i=0}^{\lceil t/2 \rceil-1} a_0^{(i)} \leq 0.51 |\m F|.
\end{equation}

%\begin{equation}
%\label{Eq:Aux-Small-t1.2}
%\sum_{i=0}^{\lceil t/2 \rceil-1} |\m F_i| \leq \tfrac{6}{n} \cdot\,a_0^{\lceil t/2 \rceil} =\tfrac{6}{\lceil t/2 \rceil} \cdot\,a_1^{\lceil t/2 \rceil} \leq \tfrac{12}{c_0} |\m F| \leq \tfrac{|\m F|}{3},
%\end{equation}
%\ohad{I think we can have $|\ff|\ge a^{\lceil t/2 \rceil}_{1} + \sum_{i < t/2} a_0^{(i)}$ (which is n-t intersecting), and also  $a^{\lceil t/2 \rceil}_{1} \ge \lfloor t/2 \rfloor \cdot a_0^{\lfloor t/2 \rfloor}$ (we can have a "bijection"), and then we get better bounds. But also I dont think we need bellow $ \le |F|/3$, so it should work even for $t=3$.}
%provided $c_0 \geq 36$. 
Therefore, for both even and odd values of $t \geq 3$, we have
\[
\sum_{i=0}^{\lceil t/2 \rceil-1} |\m F_i| \leq 0.51|\m F|,
\]
for a sufficiently large $n$.
As $\sum_{\ell=1}^\infty \frac{1}{\ell^2}=\frac{\pi^2}{6}$, this implies that there exists $\lceil t/2 \rceil \leq m \leq t$ such that
\begin{equation}\label{eqlbf}
|\ff_m|\ge \frac{6/\pi^2}{(m-(t/2)+1)^2}\cdot 0.49|\m F| \geq t^{-2}
|\ff|\ge n^{-2\epsilon} |\m F| \geq n^{-2\epsilon} a^{(m)}_{2m-t}.
\end{equation}
We use the somewhat complex term $\frac{6/\pi^2}{(m-(t/2)+1)^2}$ in~\eqref{eqlbf} instead of a simple term like $1/t$, as this will be helpful for handling the case where $2m-t$ is small; see below. 

By Claim~\ref{clatranslation}, $\m F_m$ is $(2m-t)$-intersecting. Therefore, $\m F_m$ satisfies the first two assertions of the theorem.

\medskip \noindent \textbf{Step~2: Proving the theorem in cases where $2m-t$ is small.} We show that if~\eqref{eqlbf} holds for $m$ such that $2m-t \leq 2$, then $\m F_m$ satisfies the assertion of the theorem. If $2m-t=0$, the assertion holds trivially, with $F=\emptyset$. We consider the case $2m-t=2$ and then the slightly harder case $2m-t=1$. We shall use the classical Erd\H{o}s-Ko-Rado theorem~\cite{EKR61} and its classical `stability' version, the Hilton-Milner theorem~\cite{HM67}.
\begin{thm}[Hilton and Milner]
\label{Hilton-Milner}
    Let $n,k \in \mathbb{N}$ be such that $k \leq \frac{n}2$, and let 
    $\m F \subset \binom{[n]}{k}$ be an intersecting family. If $|\m F|\geq \binom{n-1}{k-1}-\binom{n-k-1}{k-1}+2$, then there exists $i \in [n]$ such that $\m F \subset \{S \in \binom{[n]}{k}:i \in S\}$.
\end{thm}
The theorem straightforwardly implies the following:
\begin{cor}
\label{Cor:HM}
    For any $\epsilon>0$, there exists $n_0 \in \mathbb{N}$ such that the following holds for all $n \geq n_0$. Let $k \leq n^{\frac{1}{2}-\epsilon}$, and let $\m F \subset \binom{[n]}{k}$ be an intersecting family, such that $|\m F|\geq \frac{1}{20}\binom{n-1}{k-1}$. Then there exists $i \in [n]$ such that $\m F \subset \{S \in \binom{[n]}{k}:i \in S\}$.    
\end{cor}

\medskip \noindent \emph{The case $2m-t=2$.} In this case, all permutations in the family $\m G_m \subset \m G$ that corresponds to $\m F_m$ have $m=\frac{t}{2}+1$ moving points. By~\eqref{Eq:Aux-Small-t1} and~\eqref{eqlbf}, we have 
\begin{equation}\label{Eq:Aux-Small-t7}
|\m G_m| =|\ff_m|\ge \frac{6/\pi^2}{(m-(t/2)+1)^2}\cdot \frac{2|\m F|}{3} \geq \frac{6/\pi^2}{4} \cdot \frac{2}{3} \cdot \max_{t/2 \leq m' \leq t} a^{(m')}_{2m'-t} \geq \frac{1}{\pi^2}a_{0}^{(t/2)}.
\end{equation}
The term $a_{0}^{(t/2)}$ is equal to the number of permutations $\sigma \in S_n$ with exactly $t/2$ moving points, which is $\binom{n}{t/2}d_{t/2}$, where $d_{t/2}$ is the number of derangements on $t/2$ elements. By~\eqref{eq:der-short}, we have $d_{t/2} \geq \frac{1}{3}(t/2)!$. Hence,~\eqref{Eq:Aux-Small-t7} yields
\[
|\m G_m| \geq \frac{1}{\pi^2}a_{0}^{(t/2)} \geq \frac{1}{\pi^2} \cdot \binom{n}{t/2} \cdot \frac{1}{3}(t/2)!.
\]
Consider the family $\m V_m=\{\mathrm{Moving}(\sigma): \sigma \in \m G_m \} \subset \binom{[n]}{(t/2)+1}$. As each set of $\frac{t}{2}+1$ moving points corresponds to at most $(\frac{t}{2}+1)!$ permutations in $\m G_m$, we have
\begin{equation}\label{Eq:Aux-Small-t6}
    |\m V_m| \geq \frac{|\m G_m|}{((t/2)+1)!} \geq \frac{1}{3\pi^2((t/2)+1)} \cdot \binom{n}{t/2}.
\end{equation}
On the other hand, for any $\sigma,\sigma' \in \m G_m$, we have $|\mathrm{Moving}(\sigma) \cap \mathrm{Moving}(\sigma')| \geq 2$, as otherwise, $\sigma$ and $\sigma'$ disagree on at least $(\frac{t}{2}+1)+(\frac{t}{2}+1)-1=t+1$ elements, contradicting the $(n-t)$-intersection property of $\m G$. Hence, $\m V_m$ is $2$-intersecting. By the case $t=2$ of the Erd\H{o}s-Ko-Rado theorem~\cite{EKR61}, for a sufficiently large $n$ this  implies that 
\[
|\m V_m| \leq \binom{n-2}{\frac{t}{2}-1}.
\]
As $t \leq n^{\epsilon}$, this contradicts~\eqref{Eq:Aux-Small-t6} for a sufficiently large $n$. Therefore,~\eqref{eqlbf} cannot hold when $2m-t=2$.

\medskip \noindent \emph{The case $2m-t=1$.} In this case, all permutations in the family $\m G_m \subset \m G$ that corresponds to $\m F_m$ have $m=\frac{t+1}{2}$ moving points. By~\eqref{eqlbf}, we have 
\begin{equation}\label{Eq:Aux-Small-t8}
|\m G_m| =|\ff_m|\ge \frac{6/\pi^2}{(m-(t/2)+1)^2}\cdot \frac{2|\m F|}{3} \geq \frac{6/\pi^2}{2.25} \cdot \frac{2}{3} \cdot a^{(m)}_{2m-t} \geq \frac{1}{6}a_{1}^{((t+1)/2)}.
\end{equation}
The term $a_{1}^{((t+1)/2)}$ is equal to the number of permutations $\sigma \in S_n$ whose moving-points-set is of size $(t+1)/2$ and contains a specific element $x$, which is $\binom{n}{\frac{t-1}{2}}d_{(t+1)/2}$. Hence,~\eqref{Eq:Aux-Small-t8} yields
\[
|\m G_m| \geq \frac{1}{6}a_{1}^{((t+1)/2)} \geq \frac{1}{6} \cdot \binom{n}{\frac{t-1}{2}} \cdot \frac{1}{3}(\tfrac{t+1}{2})!,
\]
%As each set of $\frac{t+1}{2}$ moving points corresponds to at most $(\frac{t+1}{2})!$ permutations in $\m G_m$, 
which in turn implies
\begin{equation}\label{Eq:Aux-Small-t9}
    |\m V_m| \geq \frac{|\m G_m|}{((t+1)/2))!} \geq \frac{1}{18} \cdot \binom{n}{\frac{t-1}2}.
\end{equation}
On the other hand, for any $\sigma,\sigma' \in \m G_m$, we have $|\mathrm{Moving}(\sigma) \cap \mathrm{Moving}(\sigma')| \geq 1$, as otherwise, $\sigma$ and $\sigma'$ disagree on at least $\frac{t+1}{2}+\frac{t+1}{2}=t+1$ elements, contradicting the $(n-t)$-intersection property of $\m G$. Hence, $\m V_m$ is intersecting. 
Since $|\m V_m| \geq \frac{1}{18} \cdot \binom{n}{\frac{t-1}2}$, by Corollary~\ref{Cor:HM} this implies that assuming $n$ is sufficiently large, there exists $i \in [n]$ such that $i \in D_{\sigma}$ for all $\sigma \in \m G_m$. Therefore, for $F=\{i\}$, we have $\m F_m[F]=\m F_m$. As $F=(D,E)$, where $D=\{i\}$ and $E=\emptyset$, the family $\m F_m[F]$ satisfies the assertion of the theorem.

\smallskip As we have resolved the cases $2m-t \leq 2$, we assume from now on that~\eqref{eqlbf} holds for $m$ such that $2m-t \geq 3$. This will allow us to apply Lemma~\ref{lemadecay}(ii,iii,vii) to compare the sizes of the families we obtain during the proof.

\medskip \noindent \textbf{Step~3: Applying `peeling simplification'.} We apply Lemma~\ref{lemsimplsimple} to the family $\ff_m$, with $\mb X = [n] \cup \{(i,j):i,j \in [n], i \neq j\}$ in place of $[n]$, $2m$ in place of $p$, $2m-t$ in place of $t$, and $2n^{\epsilon}$ in place of $r$. It is clear that the assumptions of the lemma are satisfied. The simplification yields a $(2m-t)$-intersecting family $\m S_m \subset \m P(\mb X)$ such that 
\[
\ff_m = \ff_m[\m S_m] = \cup_{i=0}^{2m}\ff_m[\m S_m^{(i)}],
\]
where $\m S_m^{(i)} = \{S \in \m S_m: |S|=i\}$, and for any $i \geq 2m-t$, $|\m S_m^{(i)}| \leq \binom{i}{2m-t} \cdot (2n^{\epsilon})^{i-(2m-t)}$.

Recall that each $F \in \mb X$ has the form 
$F = (D,(M,\sigma(M)))$. We claim that for any such $F \in \m S_m$, we have $M \subset D$ and $\sigma(M) \subset D$. This clearly holds if $F \in \m F_m$, as all elements of $\m F_m$ are of the form $(D,(D,\sigma(D)))$, where being viewed as sets, $\sigma(D)=D$. Each set $F=(D,(M,\sigma(M))) \in \m S_m \setminus \m F_m$ was created during the peeling simplification process as a replacement of a spread family. If we had $M \not \subset D$ or $\sigma(M) \not \subset D$, then no family of the form $\m X_j(F)$ could be $r$-spread, as for any $u \in (M\cup \sigma(M))\setminus D$ and any family $\m X_j$, we have $(\m X_j(F))[\{u\}] = \m X_j(F)$. Therefore, $M \subset D$ and $\sigma(M) \subset D$ for any $F = (D,(M,\sigma(M))) \in \m S_m$.

Let us bound the contribution of each family $\m F[\m S_m^{(2m-t+i)}]$ to the size of $\ff_m$. As $\m F_m \subset \m A_m$, for any $2 \le i \le t$ we have
\begin{align}\label{Eq:Aux-Small-t2}
\begin{split}|\ff_m\big[\m S_m^{(2m-t+i)}\big]|&\le a^{(m)}_{2m-t+i}\cdot |\m S_m^{(2m-t+i)}| \\
&\le 3\left(\tfrac{2m}{n}\right)^{\lfloor i/2 \rfloor}
a^{(m)}_{2m-t}\cdot {2m-t+i\choose i}(2n^{\epsilon})^{i}\\
&\le 3\left(\tfrac{2m}{n}\right)^{\lfloor i/2 \rfloor} a^{(m)}_{2m-t}\cdot  (2m-t+i)^{i} (2n^\epsilon)^{i} \\
&\le n^{-0.9}a^{(m)}_{2m-t},
\end{split}
\end{align}
provided that $n$ is sufficiently large. The second inequality uses Lemma~\ref{lemadecay}(iii), applied with $m$ in place of $i$, $2m-t$ in place of $\ell_0$ and $i$ in place of $\ell$, and Lemma~\ref{lemsimplsimple}. Note that Lemma~\ref{lemadecay}(iii) can be applied here since we assumed $2m-t \geq 3$ and $m \leq t < \sqrt{n/2}$.
%\ohad{We also need $2m-t+i \le m-2$ is it obvious?} 
%\nathan{To Ohad: This was indeed a flaw! To repair this, I replaced Lemma~\ref{lemadecay}(iii) with an entirely different statement. Can you please verify that now this is fine?}\ohad{Looks Good!}
%and yields $a_\ell^{(i)}/a_{\ell+2}^{(i)} \geq \frac{n}{6e}$ in our range of values of $i$, as $i \leq t \leq n^{\epsilon}$. 
%\ohad{we somehow replaced here $2^i$ by $2^{i\epsilon}$} \nathan{Right, fixed.}

The inequalities~\eqref{eqlbf} and~\eqref{Eq:Aux-Small-t2} imply that 
\begin{equation}\label{Eq:Aux-Small-t3}
|\ff_m[\m S_m^{(2m-t)}\cup \m S_m^{(2m-t+1)}]| \geq (1-n^{-0.9+3\epsilon}) |\m F_m| \geq n^{-3\epsilon} a^{(m)}_{2m-t}.
\end{equation}
In words, most of the sets in $\ff_m$ contain an element of $\m S_m^{(2m-t)}\cup \m S_m^{(2m-t+1)}$, which is a $(2m-t)$-intersecting family that consists only of sets of size $2m-t$ and $2m-t+1$, and hence, must have a very specific structure. 

\medskip \noindent \textbf{Step~4: Analyzing the family $\m S_m^{(2m-t)}\cup \m S_m^{(2m-t+1)}$.} At this step, we consider several cases, according to the structure of the family $\m S_m^{(2m-t)}\cup \m S_m^{(2m-t+1)}$. In each case, we either reach a contradiction or find $F$ such that $\m F_m[F]$ satisfies the assertion of the theorem.

\medskip \noindent \emph{Case~1: $\m S_m^{(2m-t)} \neq \emptyset$.} Being a non-empty $(2m-t)$-intersecting family in which each set contains exactly $2m-t$ elements, $\m S_m^{(2m-t)}$ must consist of a single set $F$. Moreover, in this case, $\m S_m^{(2m-t+1)} = \emptyset$, since by the $(2m-t)$-intersection property of $\m S_m$, each set in $\m S_m^{(2m-t+1)}$ must contain $F$, while by Lemma~\ref{lemsimplsimple}, all sets in $\m S_m$ are inclusion-minimal. Furthermore, we have $\m F_m[F] = \m F_m$, as by Lemma~\ref{lemsimplsimple}, the family $\m F_m$ cross $(2m-t)$-intersects $\m S_m$, and thus, any set in $\m F_m$ must contain $F$. Therefore, 
\begin{equation}\label{Eq:Aux-Small-t4}
   |\m F_m[F]| =|\m F_m| \geq  n^{-2\epsilon} a^{(m)}_{2m-t}.
%    |\m F_m[F]| \geq (1-n^{-0.9+3\epsilon}) |\m F_m| \geq (1-n^{-\epsilon})|\m F_m|    
%    \geq n^{-3\epsilon} a^{(m)}_{2m-t}.
\end{equation} 
%\andrey{We write $n^{-3\epsilon}$ here, so probably need to adjust it in the statement?} \nathan{I think it's fine since the $n^{-3\epsilon}$ refers to the entire family $\m F$. I added one term that hopefully clarifies this.}

Denote $F=(D,(M,\sigma(M)))$. As was shown above, $M \subset D$ and $\sigma(M) \subset D$. If $|M| \leq |D|-2$, then by Lemma~\ref{lemadecay}(vii), which can be applied since $t \geq 3$ and $2m-t \geq 3$, 
we have 
\[
|\ff_m[F]| \le |\aaa_m[F]|\le \frac{3m^2}{n} \cdot a^{(m)}_{2m-t} \leq 3n^{-1+2\epsilon}a^{(m)}_{2m-t},
\]
in contradiction to~\eqref{Eq:Aux-Small-t4}. Therefore, if $|F|=2m-t$ is even, then we have $M=D$, and if $|F|$ is odd, then we have $|M|=|D|-1$. In both cases, we have $|D|=m- \lfloor t/2 \rfloor$ and $|M|=m- \lceil t/2 \rceil$. Hence, $\m F_m[F]$ satisfies the assertion of the theorem.

\medskip \noindent \emph{Case~2: $\m S_m^{(2m-t)} = \emptyset$ and $2m-t$ is even.} In this case, $2m-t \geq 4$, and thus, Lemma~\ref{lemadecay}(ii) can be applied to deduce that  $a_{2m-t+1}^{(m)}/a_{2m-t}^{(m)}\le \frac{3m}{n} \leq n^{-1+2\epsilon}$. By Lemma~\ref{lemsimplsimple}(3), we have
$|\m S_m^{(2m-t+1)}| \leq \binom{2m-t+1}{2m-t} \cdot (2n^\epsilon) \leq 4n^{2\epsilon}$. 
Thus, 
%\ohad{We assumed here $| S_m^{(2m-t+1)}| =1$ why is it true?} \nathan{Of course, I don't remember anymore... This should be anyway easy to fix since in this case we are not close to be tight. If you can find a fix, this will be nice. Otherwise, I'll try to check later on.}\ohad{WE use lemma~\ref{lemsimplsimple} to bound the size, $s=2m-t$\[
%|\mathcal S_m^{(s+1)}|
%\le \binom{s+1}{s}(2n^\epsilon)
%=2(s+1)n^\epsilon
%\le 4n^{2\epsilon},
%\]}
\begin{align*}
|\ff_m\big[\m S_m^{(2m-t)} \cup \m S_m^{(2m-t+1)}\big]|&=
|\ff_m\big[\m S_m^{(2m-t+1)}\big]|\le a_{2m-t+1}^{(m)}|\m S_m^{(2m-t+1)}| \\
&\leq n^{-1+2\epsilon}a^{(m)}_{2m-t} \cdot 4n^{2\epsilon},
\end{align*}
in contradiction to~\eqref{Eq:Aux-Small-t3}.

\medskip \noindent \emph{Case~3: $\m S_m^{(2m-t)} = \emptyset$ and $2m-t$ is odd.} In this case, we write $\m S_m^{(2m-t+1)}= \m T_1\sqcup \m T_2$, where $\m T_1$ contains all sets $F=(D,(M,\sigma(M)))$ with $M=D$ and $\m T_2$ contains all other sets. 

Let $F=(D,(M,\sigma(M))) \in \m T_2$. Note that as $2m-t+1$ is even and $M \subset D$, we have $|D|\ge |M|+2$. By Lemma~\ref{lemadecay}(i,vii),
%, which can be applied since $2m-t \geq 3$ and 
%and Lemma~\ref{lemadecay}(i), 
we have 
\[
|\m A_m[F]|\le \frac{3m^2}{n} a^{(m)}_{2m-t+1} \leq 3n^{-1+2\epsilon} a^{(m)}_{2m-t+1}\le 3n^{-1+2\epsilon} a^{(m)}_{2m-t}.
\]
By Lemma~\ref{lemsimplsimple}, we have $|\m T_2| \leq |S_m^{(2m-t+1)}| \leq (2m-t+1)(2n^{\epsilon}) \leq 2n^{2\epsilon}$. Hence, as in the previous case, we conclude that 
\begin{equation}\label{Eq:Aux-Small-t5}
 |\ff_m\big[\m T_2\big]|\le 6n^{-1+4\epsilon}a^{(m)}_{2m-t} \leq n^{-1+6\epsilon} |\m F_m|.   
\end{equation}
Now, we show that $\m T_1$ consists of a single set. 
Let $F_1=(D_1,E_1)$ and $F_2=(D_2,E_2)$ be elements of $\m T_1$, where $E_1=(D_1,\sigma_1(D_1))$ and $E_2=(D_2,\sigma_2(D_2))$. As $|F_1|=|F_2|=2m-t+1$ and these two sets $(2m-t)$-intersect, we must have $D_1=D_2$ and $|E_1\cap E_2|\ge |D_1|-1.$ As for each $j \in \{1,2\}$,  $\sigma_j(D_j) \subset D_j$, we have $\sigma_j(D_j)=D_j$, and hence, $\sigma_j|_{D_j}$ is a permutation on $D_j$. Thus, $\sigma_{1}|_{D_1}$ and $\sigma_{2}|_{D_2}$ are two permutations on the same set of size $|D_1|$ that agree on at least $|D_1|-1$ elements. This means that they agree on all elements, and hence, $F_1=F_2$.

Let $F'=(D',(D',\sigma(D'))$ be the single element of $\m T_1$. By~\eqref{Eq:Aux-Small-t3} and~\eqref{Eq:Aux-Small-t5}, we have
\[
|\m F_m[F']| \geq (1-n^{-0.9+3\epsilon}) |\m F_m| \geq (1-n^{-\epsilon})|\m F_m|.
\]
Setting $D=D'$, $E=(D',\sigma(D')) \setminus \{(u,\sigma(u))\}$ for some $u \in D'$, and $F=(D,E)$, we have $|\m F_m[F]| \geq (1-n^{-\epsilon})|\m F_m|$ (since 
$\ff_m[F]\supset\ff_m[F']$), $|D|=m-\lfloor t/2 \rfloor$ and $|E|=m-\lceil t/2 \rceil$. Therefore, $\m F_m[F]$ satisfies the assertion of the theorem.
\end{proof}

\section{Finding a Simple Sub-Structure Within $\m F$, for a Medium $t$}
\label{sec:Medium_t}

In this section, we study $(n-t)$-intersecting families, for $n^{\epsilon} \leq t \leq n^{(1+\epsilon)/2}$. Theorem~\ref{Thm:Main} asserts that in this range, the maximum size of an $(n-t)$-intersecting family is attained by a family of the form
\[
\m F_{n,n-t,\tfrac{t-r}{2}}=\{\sigma \in S_n: |\mathrm{Moving}(\sigma) \cap \{1,2,\ldots,n-r\}| \leq \tfrac{t-r}{2}\},
\]
for some $r \geq 0$. For values of $t$ up to $\approx \sqrt{n}$, the maximum is attained for $r=0$ for an even $t$ and for $r=1$ for an odd $t$, and for larger values of $t$, the value $r$ gradually increases, up to $r \approx n^{(1+\epsilon)/2}$ for $t=n^{(1+\epsilon)/2}$ (see Lemma~\ref{lemsizeofmaxFrankl} below).
%\ohad{Isnt $r$ actually $\approx t$ in this range?} \nathan{Right, thanks! What I wrote is $\frac{r-t}{2}$.}

We consider an $(n-t)$-intersecting family $\m G$ of a maximum size and transform it to a family $\m F$ of sets, using the representation of permutations by sets presented in Section~\ref{sec:PermutationstoSets}. %Decompose $\m G = \cup_m \m G_m$, where $\m G_m$ is the set of permutations in $\m G$ that have exactly $m$ moving points, and denote by $\m F_m$ the family of sets that corresponds to $\m G_m$. 
We show that $\m F$ is almost entirely contained in an $(n-t)$-intersecting family $\m H$ of a simple structure.  
To state the result formally, we need an additional definition.
\begin{definition}
    A set $X \subset [n] \cup \{(i,j):i,j \in [n], i \neq j\}$ is called \emph{extendable} if $X \subset F_{\sigma}$ for some $\sigma \in S_n$.
\end{definition}
In words, this means that $X$ represents part of the information on the set of moving points of some $\sigma \in S_n$ and the places they move to.
\begin{thm}\label{thmfm2}
    For any $\epsilon \leq 0.01$, there exists $n_0 \in \mathbb{N}$ such that the following holds for all $n \geq n_0$ and all $n^{\epsilon} \leq t \leq n^{(1+\epsilon)/2}$. 
    Let $\m G \subset S_n$ be a maximum-size $(n-t)$-intersecting family of permutations. Let $\m F \subset \mb [n] \cup \{(i,j):i,j \in [n], i \neq j\}$ be the family of sets that corresponds to $\m G$ in the representation of permutations by sets presented in Section~\ref{sec:PermutationstoSets}. Then there exist $t/2 \leq m' \leq t$, an extendable set $F'$ of size $2m'-t$ of the form $F' = \{D,(M,\sigma_0(M))\}$, where $\sigma_0 \in S_n$, $M\subset D$ and $\sigma_0(M) \subset D$, and a family $\m H=\cup_{i=0}^{t}\m H_i$,  
    where 
    \[\m H_i:=\{A\in \m A_i: |A\cap F'|\ge i-(t-m')\},
    \]
    such that $|\m F \cap \m H| \geq (1-n^{-\epsilon/2})|\m F|$. 
    \end{thm}
Before we prove the theorem, let us examine the `special structure' admitted by most elements of $\m F$. Consider the `special' set $F'$ whose existence is asserted by the theorem. If $F'$ consists of $2m'-t$ singleton-elements,
%\andrey{maybe specify, what point-elements are?}
%\nathan{Right, we called them   `singleton-elements' elsewhere. Fixed now.}
then for each $i$, $\m H_i$ corresponds to the family of permutations $\sigma \in S_n$ that have $i$ moving points, where at least $i-(t-m')$ of these points belong to $F$. This means that $\m H$ corresponds to the family of permutations $\{\sigma \in S_n: |\mathrm{Moving}(\sigma) \cap ([n] \setminus F)| \leq t-m'\}$, which is a double translate of the family $\m F_{n,n-t,t-m'}$. %\andrey{I believe that this family should be $\m F_{n,n-t,t-m'}$, right?}
%\nathan{Right, thanks!}
Hence, in this case the theorem asserts that $\m G$ is almost entirely contained in a double translate of $\m F_{n,n-t,t-m'}$. In Section~\ref{sec:Completing_the_Proof} we will show that this allows us to deduce that $\m G$ is entirely contained in a double translate of $\m F_{n,n-t,t-m'}$, as asserted in Theorem~\ref{Thm:Main}, and that a similar assertion can be deduced in the case where some of the elements in $F'$ are pair-elements.    

\begin{proof}[Proof of Theorem~\ref{thmfm2}]
Let $n,t,\epsilon$ be as in the statement of the theorem, let $\m G \subset S_n$ be a maximum-size $(n-t)$-intersecting family of permutations, and let $\ff = \cup_{i=0}^t\ff_i$ be the corresponding family of sets. Recall that each $\ff_i$ is $(2i-t)$-intersecting. The proof consists of several steps.

\medskip \noindent \textbf{Step~1: Applying the `iterative spread approximation' simplification.} A central difference between this range of values of $t$ and the `small $t$' range considered in Section~\ref{sec:Small_t} is that the size of each set in $\m F_i$ (which is $2i$) is much larger than the `intersection size' $2i-t$, which implies that the intersection property allows deducing much less information on the families $\m F_i$. To overcome this, we apply the `iterative spread approximation lemma', which essentially allows replacing each $\m F_i$ with a $(2i-t)$-intersecting family of sets of sizes not much larger than $2i-t$. We apply the lemma (i.e., Theorem~\ref{thmapprox1}) to each $\ff_i$, for $i = \lceil t/2\rceil,\ldots t$. (As we shall show below, the contribution of the $\m F_i$'s for $i< \lceil t/2\rceil$ to $|\m F|$ is negligible in this range of values of $t$). The following statement describes the approximating family we obtain. Recall the notation $\mb X=[n] \cup \{(i,j):i,j \in [n], i \neq j\}$.
\begin{lem}\label{lemregime2}
    Let $n,\epsilon,t, \m F$ be as defined above, and for all $\lceil t/2\rceil \leq i \leq t$, let $\m F_i$ be as defined above. There exist families $\m S_i\subset \m P({\mb X})$ such that:
    \begin{enumerate}
        \item Each set in $\m S_i$ is of size at most $(2i-t)+k_i$, where $k_i\le 50\log n$ for $t\le n^{1/2}$ and $k_i\le n^{2\epsilon}$ for $t\ge n^{1/2}$;

        \item $|\ff_i\setminus \ff_i[\m S_i]| \leq n^{-9}|\m F|$; 

        \item For any $i,j$, the families $\m S_i$ and $\m S_j$ are cross $(i+j-t)$-intersecting.
    \end{enumerate} 
\end{lem}

In the proof of Lemma~\ref{lemregime2}, 
we use the following theorem that lies in the basis of the spread approximation technique. The theorem is due to Alweiss, Lovett, Wu and Zhang~\cite{ALWZ21}, and we use here a sharpening due to Tao~\cite{TaoSunflower}, following~\cite[Thm.~4]{KZ22}. 

The formulation uses the following standard definition. A $p$-random subset of $[n]$ is $X \subset [n]$ obtained by
picking each $i \in [n]$ with probability $p$, independently over $i$.
\begin{thm}[The spreadness lemma] \label{thmtao}
  Let $n,k,r \in \mathbb{N}$, let $\beta,\delta >0$, and let $\m X\subset {[n]\choose \le k}$ be an $r$-spread family. Let $W$ be a $(\beta \delta)$-random subset of $[n]$. Then 
  \[
  \Pr[\exists F\in \m X: F\subset W]\ge 1-\Big(\frac 5{\log_2(r\delta)} \Big)^\beta k.
  \]
\end{thm}
Intuitively, the spreadness lemma asserts that if $\m X \subset \m P([n])$ is a `spread' family then a random `reasonably large' set $W \subset [n]$ is expected to contain a set from $X$.

\begin{proof}[Proof of Lemma~\ref{lemregime2}]
For each $i$, we would like to apply Theorem~\ref{thmapprox1} with $n^2$ in place of $n$, $2i$ in place of $k$, $2i-t$ in place of $t$, $n$ in place of $R_1$, $\min(\frac{n}{t},\frac{t}{11})$ in place of $R$, $10\log_2 n$ in place of $\sigma$,
the set $\mb X$ in place of $[n]$, the family $\m A_i \subset \binom{\mb X}{2i}$ in place of $\m A \subset \binom{[n]}{k}$, $a_j^{(i)}$ in place of $a_j$, and the $(2i-t)$-intersecting family $\m F_i \subset \m A_i$ in place of $\m F \subset \m A$. Let us verify that for each $ \lceil t/2\rceil\le i\le t$, the assumptions of the theorem are satisfied. 

The weak spreadness assumptions on $\m A_i$ are satisfied, as by  Lemma~\ref{lemadecay}(i), for any $ \lceil t/2\rceil\le i\le t$, the family $\m A_i \subset \binom{\mb X}{2i}$ is weakly $(R,\ell)$-spread for any $\ell$ with $R=\min(\frac{n}{i},\frac{i}{2e}) \ge \min(\frac{n}{t},\frac{t}{11})$, and on the other hand, by  Lemma~\ref{lemadecay}(iv), for any $\ell$ we have $a_\ell^{(i)}/a_{\ell+1}^{(i)}\le n.$ The assumption $R \geq 2^{15}\log_2(4k)$, which in our case reads 
\[
\min(\tfrac{n}{t},\tfrac{t}{11}) \geq 2^{15}\log_2(8i),
\]
clearly holds assuming $n$ is sufficiently large, as $n^{\epsilon} \leq t \leq n^{(1+\epsilon)/2}$ and $i \leq t$.
The assumption $R \geq 2^{30}(t^{1/2}\log_2 t+\log_2 R_1)+200\sigma$, which in our case reads
\[
\min(\tfrac{n}{t},\tfrac{t}{11}) \geq 2^{30}((2i-t)^{1/2}\log_2(2i-t)+\log_2 n)+2000 \log_2 n,
\]
also holds assuming $n$ is sufficiently large, as $n^{\epsilon} \leq t \leq n^{(1+\epsilon)/2}$ and $2i-t \leq t$.

Therefore, we can apply Theorem~\ref{thmapprox1} to approximate $\m F_i$ by a $(2i-t)$-intersecting family $\m S_i$ consisting of `small' sets. Assertion~(iii) of the theorem states that $\m R_i:= \m F_i \setminus \m F_i[\m S_i]$ satisfies 
\[
|\m R_i| \leq n^{-10} \cdot 4 \log_2(n+10\log_2 n)a_{2i-t}^{(i)} \leq n^{-9} a_{2i-t}^{(i)} \leq n^{-9}|\m F|,
\]
where the last inequality holds as by Lemma~\ref{lemadecay}(v), there exists an $(n-t)$-intersecting family of permutations of size at least $a_{2i-t}^{(i)}$, while $|\m F|=|\m G|$ is assumed to be the maximum size of an $(n-t)$-intersecting family in $S_n$. This proves Assertion~(2) of the lemma.

\smallskip As for the sizes of the sets in $\m S_i$, we can obtain an improved bound using the `in addition' part of Theorem~\ref{thmapprox1}. 
For $t \le n^{1/2}$, the additional assumption $R \geq 2^{30}t^{\alpha}$, which in our case reads as 
\begin{equation}\label{Eq:Aux-Medium-t1}
\min(\tfrac{n}{t},\tfrac{t}{11}) \geq 2^{30} (2i-t)^{\alpha},
\end{equation}
holds for $\alpha=1-\frac{34}{\log_2 t}$, since $2i-t \leq t$ and in this range we have $\frac{n}{t} \geq t$. Thus, the theorem asserts that the size of each set in $\m S_i$ is at most 
\[
(2i-t)+(2i-t)^{34/\log_2 t}+40\log_2 n+4\log_2(2i-t+1) \leq (2i-t)+50 \log n,
\]
where the inequality holds assuming $n$ is sufficiently large. 

For $n^{1/2} \leq t \leq n^{(1+\epsilon)/2}$, the inequality~\eqref{Eq:Aux-Medium-t1} holds for $\alpha=1-2\epsilon-\frac{34}{\log_2 n}$, since in this range we have $2i-t \leq t \leq n^{(1+\epsilon)/2}$ and $\min(\tfrac{n}{t},\tfrac{t}{11}) \geq \frac{n}{11t} \geq n^{(1-\epsilon)/2}/11$.    
Thus, the size of each set in $\m S_i$ is at most 
\[
(2i-t)+(2i-t)^{2\epsilon+34/\log_2 t}+40\log_2 n+4\log_2(2i-t+1) \leq (2i-t)+n^{2\epsilon},
\]
where the inequality holds assuming $n$ is sufficiently large.
Therefore, in both cases, each set in $\m S_i$ is of size at most $(2i-t)+k_i$, where $k_i\le 50\log n$ for $t\le n^{1/2}$ and $k_i\le n^{2\epsilon}$ for $t\ge n^{1/2}$. This proves Assertion~(1) of the lemma.

\smallskip
It is left to show that for any $i_1,i_2$, the families $\m S_{i_1}, \m S_{i_2}$ are cross $(i_1+i_2-t)$-intersecting. For this, we use Assertion~(ii) of Theorem~\ref{thmapprox1} which states that for any $S \in \m S_i$, there exists $\m F_S \subset \m F_i$ such that $\m F_S(S)$ is $\min(\frac{n}{2t},\frac{t}{22})$-spread. 

Suppose on the contrary that there exist $i_1,i_2$ and $S_1\in \m S_{i_1},$ $S_2\in \m S_{i_2}$ such that $|S_1\cap S_2| = i_1+i_2-t-j$ for some $j>0$. Assume w.l.o.g.~that $i_2\ge i_1$. As was shown above, we have $|S_{1}|\le 2i_1-t+k_{i_1}$. %\ohad{$|S_{1}|\le 2i_1-t+k_{i_1}$?}
%\nathan{Yes.}
%\le 2i_1-t+n^{0.02}$, where the second inequality holds assuming $n$ is sufficiently large, since $\epsilon \leq 0.01$. 
Hence,
\[
0\le |S_1\setminus S_2| \le 2i_1-t+k_{i_1}-(i_1+i_2-t-j)=i_1+k_{i_1}+j-i_2,
\]
and consequently, $i_2\le i_1+k_{i_1}+j$. Applying the same argument with the roles of $i_1,i_2$ interchanged and using the bound $i_2-i_1 \leq k_{i_1}+j$, we obtain 
\[
|S_2\setminus S_1| \le 2i_2-t+k_{i_2}-(i_1+i_2-t-j)=i_2+k_{i_2}+j-i_1\le k_{i_1}+k_{i_2}+2j.
\]
Now we apply a spreadness argument, similar to the argument used in the proof of Lemma~\ref{lemsimplsimple}(3) above. Denote $r=\min(\frac{n}{2t},\frac{t}{22})$ and consider the $r$-spread families $\m F_{S_1}(S_1) \subset \m F_{i_1}(S_1)$, $\m F_{S_2}(S_2) \subset \m F_{i_2}(S_2)$, 
%\ohad{, $\m F_{S_2}(S_2) \subset \m F_{i_2}(S_2)$?}\nathan{Yes.}
whose existence is guaranteed by Assertion~(ii) of Theorem~\ref{thmapprox1}. We claim that the fraction of sets in $\m F_{S_1}(S_1)$ that intersect $S_2\setminus S_1$ in at least $\lceil \frac{j}{2} \rceil$ elements is at most $n^{-\epsilon/2}$. Indeed, by the $r$-spreadness of $\m F_{S_1}(S_1)$ and a union bound, this fraction is at most
\[
\binom{|S_2\setminus S_1|}{\lceil \frac{j}{2} \rceil}r^{\lceil -j/2 \rceil}\le \Big(\frac{2e |S_2\setminus S_1|}{jr}\Big)^{\lceil j/2 \rceil}\le \Big(\frac{2e (k_{i_1}+k_{i_2}+2j)}{jr}\Big)^{\lceil j/2 \rceil}.
\]
For $n^{\epsilon} \leq t \leq n^{1/2}$, we have $r=t/22$ and $k_{i_1},k_{i_2} \leq 50\log n$, and thus, the numerator is bounded from above by $2e(100\log n+2j)$ while the denominator is bounded below by $n^{\epsilon}\cdot j/22$. For $n^{1/2} \leq t \leq n^{(1+\epsilon)/2}$, we have $r \geq n^{(1-\epsilon)/2}/22$ and $k_{i_1},k_{i_2} \leq n^{2\epsilon}$, and thus, the numerator is bounded from above by $2e(2n^{2\epsilon}+2j)$ while the denominator is bounded by below by $n^{(1-\epsilon)/2}\cdot j/22$. Clearly, in both cases the whole expression is bounded from above by $n^{-\epsilon/2}$, assuming $n$ is sufficiently large.
The same bound holds for the fraction of sets in $\m F_{S_2}(S_2)$ that intersect $S_1\setminus S_2$ in at least $\lceil \frac{j}{2} \rceil$ elements. Denote
\[
\m X_1:=\{S' \in \m F_{S_1}(S_1): |S' \cap (S_2 \setminus S_1)| < \lceil \tfrac{j}{2} \rceil \}, 
\]
and define $\m X_2 \subset \m F_{S_2}(S_2)$ in the same way. By the above argument, for $\ell=1,2$ we have $|\m X_\ell| = (1-n^{-\epsilon/2})|\m F_{S_\ell}(S_{\ell})|$, and thus, $\m X_1,\m X_2$ are $(r/2)$-spread families. This allows us to deduce that there exist 
$X_1\in \m X_1,X_2\in \m X_2$ such that $X_1 \cap X_2 = \emptyset$. Specifically, applying Theorem~\ref{thmtao}, with $n^2$ in place of $n$, $2i_1$ in place of $k$, $(r/2)$ in place of $r$, $\mb X$ in place of $[n]$, the family $\m X_1$ in place of $\m F$, $(r/2^{12})$ in place of $\beta$ and $(2^{11}/r)$ in place of $\delta$, we deduce that if $W$ is a $(1/2)$-random subset of $\mb X$, then 
\[
\Pr[\exists X_1\in \m X_1: X_1\subset W]\ge 1-(\tfrac 1{2})^{r/2^{12}} \cdot (2i_1) \geq 1- (\tfrac 1{2})^{n^{\epsilon}/(11 \cdot2^{12})} \cdot (2n^{(1+\epsilon)/2})>\tfrac{1}{2}.
\]
By the same argument, if $W'$ is a $(1/2)$-random subset of $\mb X$, then 
$\Pr[\exists X_2\in \m X_2: X_2\subset W']>1/2$. Therefore, there exists $W \subset \mb X$ and $W'=\mb X \setminus W$, such that there exist $X_1 \in \m X_1$ with $X_1 \subset W$ and $X_2 \in \m X_2$ with $X_2 \subset W'$. These two sets satisfy $X_1 \cap X_2 = \emptyset$.
 
Consider the sets $X_1\sqcup S_1\in \ff_{i_1},X_2\sqcup S_2\in \ff_{i_2}$.
We have
\begin{align*}|(X_1\sqcup S_1)\cap&(X_2\sqcup S_2)|=|S_1\cap S_2|+|X_1\cap (S_2\setminus S_1)|+|X_2\cap (S_1\setminus S_2)|+|X_1\cap X_2|\\
&\le(i_1+i_2-t-j)+(\lceil\tfrac j2 \rceil -1)+(\lceil\tfrac j2 \rceil -1)+0<i_1+i_2-t.
\end{align*}
This contradicts the assumption that $\m F$ corresponds to an $(n-t)$-intersecting family, as by Lemma~\ref{clatranslation}, this assumption implies that the families $\m F_{i_1},\m F_{i_2}$ are $(i_1+i_2-t)$-intersecting.

Therefore, for every $\lceil t/2 \rceil \leq i_1,i_2 \leq t$, the families $\m S_{i_1}, \m S_{i_2}$ are $(i_1+i_2-t)$-intersecting, which proves Assertion~(3) of the lemma.
\end{proof}

We claim that $\cup_{\lceil t/2 \rceil \leq i \leq t} \m S_i$ is a `good approximation' for $\m F$, in the sense that  
\begin{equation}
\label{Eq:Aux-Medium-t2}
\Big|\ff\setminus\bigcup_{i=\lceil t/2\rceil}^t \ff_i[\m S_i]\Big| \leq  n^{-\epsilon/2}|\ff|.
\end{equation}
Indeed, recall that $\m F=\sqcup_{i=0}^t \m F_i$. Assertion~(2) of Lemma~\ref{lemregime2} implies that 
\[
|\sqcup_{i=\lceil t/2\rceil}^t \m F_i \setminus \ff_i[\m S_i]| \leq  n^{-8}|\ff|.
\]
As for values $i < \lceil t/2\rceil$, 
Equations~\eqref{Eq:Aux-Small-t1} and~\eqref{Eq:Aux-Small-t1.1} hold in our range of $t$ as well, showing that for an even $t$, we have $\sum_{i=0}^{\lceil t/2 \rceil-1} |\m F_i| \leq \frac{6}{n} |\m F|$. %Furthermore, the inequality~\eqref{Eq:Aux-Small-t1.2} is even stronger in our range than in the `small $t$' range, \nathan{Check whether this sentence should be revised.} showing that 
By a similar argument, for an odd $t$ we have
\[
\sum_{i=0}^{\lceil t/2 \rceil-1} |\m F_i| \leq \tfrac{6}{n} \cdot\,a_0^{\lceil t/2 \rceil} =\tfrac{6}{\lceil t/2 \rceil} \cdot\,a_1^{\lceil t/2 \rceil} \leq \tfrac{12}{n^{\epsilon}} |\m F|.
\]
Combining these bounds together yields~\eqref{Eq:Aux-Medium-t2}, assuming $n$ is sufficiently large.

Let $\lceil t/2\rceil \leq p \leq t$ be the smallest such that $\m S_{p} \neq \emptyset$. The families $\m S_{p+k_p+i}$, $i\ge 1$, must be empty, as by Assertion~(3) of Lemma~\ref{lemregime2}, each $S_1 \in \m S_p$ and $S_2 \in \m S_{p+k_p+i}$ satisfy $|S_1 \cap S_2| \geq p+(p+k_p+i)-t=(2p-t)+k_p+i$, while by Assertion~(1) of Lemma~\ref{lemregime2}, $|S_1| \leq (2p-t)+k_p$. 
Combining with the above, we get
\begin{equation}\label{eqmostf-small}\Big|\ff\setminus \bigcup_{i=0}^{k_p}\ff_{p+i}[\m S_{p+i}]\Big| \leq n^{-\epsilon/2}|\ff|.
\end{equation}

\medskip \noindent \textbf{Step~2: Applying the `peeling simplification' procedure.} In order to further simplify the approximating family, we apply the `peeling simplification' process to each $\m S_i$, $i\in [p, p+k_p]$. Specifically, we set 
\[
k'=\max_{p \leq i \leq p+k_p} k_i+2
\]
(so, $k' \leq 100\log n$ for $t \leq n^{1/2}$ and $k' \leq 2n^{2\epsilon}$ for $n^{1/2} \leq t \leq n^{(1+\epsilon)/2}$), and
apply  Lemma~\ref{lemsimplsimple} with $n^2$ in place of $n$, $2i-t$ in place of $t$, $(2i-t)+k_{i}$ in place of $p$, $k'$ in place of $r$, $\mb X$ in place of $[n]$, and the $(2i-t)$-intersecting family $\m S_i \subset \binom{\mb X}{\leq (2i-t)+k_{i}}$ in place of $\m X$. The lemma can indeed be applied, as $k'>k_{i}+1$ for all $p \leq i \leq p+k_p$. The lemma yields $(2i-t)$-intersecting families $\m W_i \subset \m P(\mb X)$ such that $\m S_i=\m S_i[\m W_i]$ and for any $p \leq i,j \leq p+k_p$, each of the families $\m S_i, \m W_i$ cross $(i+j-t)$-intersects each of the families $\m S_j, \m W_j$. The construction used in the proof of the lemma guarantees that each $\m W_i$ admits spreadness properties that will be used below. Furthermore, for each $i$, denoting by $\m W_i^{(2i-t+j)}$ the family of sets of size $2i-t+j$ in $\m W_i$, the lemma provides the bound 
\[
|\m W_i^{(2i-t+j)}| \leq \binom{2i-t+j}{2i-t} \cdot (k')^j.
\]
However, in this range (unlike the `small $t$' range), this bound is not sufficiently strong, and we will bound the $|\m W_i|$'s using a more complex argument below.

By~\eqref{eqmostf-small} and the conclusion of Lemma~\ref{lemsimplsimple}, we have 
\[
\Big|\bigcup_{i=p}^{p+k_p}\bigcup_{j=0}^{k_{i}}\ff_i[\m W_i^{(2i-t+j)}]\Big| \geq (1-n^{-\epsilon/2})|\ff|.
\]
By the pigeonhole principle, %and the fact that $k_p\le 2n^{\epsilon}$,
there exist $m\in [p, p+k_p], k\in [0,k_{m}]$ such that
\begin{equation}\label{eqchoicem} |\ff_{m}[\m W_m^{(2m-t+k)}]|\ge (2k')^{-2} |\ff|.\end{equation}
Let us put $\m W:= \m W_m^{(2m-t+k)}$ for shorthand.

\medskip \noindent \textbf{Step~3: Finding a simple sub-structure within $\m W$.}
If~\eqref{eqchoicem} is obtained for $k=0$, then $\m W$ consists of a single set $F \subset \mb X$ of size $2m-t$, and thus, all sets in $\m F_m[\m W]$ contain the same set of size $2m-t$. In this step we show that if~\eqref{eqchoicem} is obtained for some $k>0$, then  there exists a set $F'$ of size $2m-t+2k$, such that almost all sets in $\m F_m[\m W]$ contain a subset of $F'$ of size $2m-t+k$. The proof consists of three sub-steps:
\begin{enumerate}[label=(\alph*)]
    \item Using a spreadness argument to upper bound $|\m W|$ in terms of quantities denoted by $\{f_{m,k}(j)\}_{j=0,1,\ldots,k}$, where $f_{m,k}(k)$ corresponds to the set of elements of $\m W$ that contain a subset of size $2m-t+k$ of a certain set $F'$ of size $2m-t+2k$.

    \item Using the assumption that $|\m F|$ is the maximum size of an $(n-t)$-intersecting family in $S_n$ and comparisons between the quantitites $a_j^{(i)}$ to deduce a lower bound on $|\m F|$ in terms of $f_{m,k}(k)$.

    \item Showing that $\sum_{j<k} f_{m,k}(j)$ is much smaller than $f_{m,k}(k)$, which allows deducing that most elements in $\m F_m[\m W]$ contain a subset of size $2m-t+k$ of $F'$. 
\end{enumerate}
Throughout the proof of this step, we assume that~\eqref{eqchoicem} is obtained for $k>0$, as our structural statement holds trivially where~\eqref{eqchoicem} is obtained for $k=0$.

\medskip \noindent \emph{Step~3a: Bounding $|\m W|$ from above.}
Let $A,B\in \m W$ be 
such that $I:=A\cap B$ satisfies $|I|=2m-t$.
(We may assume that there exist such sets, as otherwise, $\m W$ is $t'$-intersecting for some $t'>2m-t$, and then we can provide even better bounds on its size).
Any set $C\in \m W$ intersects both $A$ and $B$ in at least $2m-t$ elements. Denote 
\[
C_0=C\cap I, \qquad C_1=C\cap (A\setminus I), \qquad C_2 = C\cap (B\setminus I).
\]
Putting $|C_0|=(2m-t)-j$, we must have $|C_1|, |C_2|\ge j$. Choose $C_1'\subset C_1,C_2'\subset C_2$ such that $|C_1'| = |C_2'| = j$. For any $X \subset \mb X$ of size at most $k-j-1$, if $|\m W(C_0\cup C_1'\cup C_2'\cup X)|>1$, then $\m W(C_0\cup C_1'\cup C_2'\cup X)$ is not $k'$-spread by the construction of $\m W$. Indeed, otherwise one could simplify $\m W$ further by replacing $\m W[C_0\cup C_1'\cup C_2'\cup X]$ with $\{C_0\cup C_1'\cup C_2'\cup X\}$.  %\nathan{I guess the reason is that otherwise, we could apply Lemma 13 to replace all sets containing $C_0\cup C_1'\cup C_2'\cup X$ by the single set $C_0\cup C_1'\cup C_2'\cup X$. But in order to apply the lemma, we need at least two sets that contain it and not one.}\andrey{I made a slight change above} 
%(The sets in $\m W$ are inclusion-minimal w.r.t. having of the peeling procedure.\ohad{why is that true?} \andrey{Not sure if it's needed, but added this to peeling - this is for free.} 
Therefore, by Observation~\ref{obs13} we have  $|\m W(C_0\cup C_1'\cup C_2')|\le (k')^{k-j}$. %\nathan{It is not clear to me why we can apply Observation~\ref{obs13} here. In order to apply it, we need the sets in $\m W(C_0\cup C_1'\cup C_2')$ to be of size at most $k-j$. Why does this hold?}\ohad{I think its just that $|C_0\cup C_1'\cup C_2'| = 2m-t-j+j+j = 2m-t+j$, and sets in $W$ have size $2m-t+k$}\nathan{Right, thanks.} 
As this holds for any $C_0,C'_1,C'_2$, we can upper bound the size of $\m W$ as follows:
%\nathan{It seems that we should multiply also by $\binom{2m-t+k}{2m-t}$, which is the number of ways to choose $I$. If this is indeed the case, this will require significant technical adjustments in the sequel.} \ohad{I don't think so, I think $A,B$ are chosen before (so naturally also $I$), only the $C$'s need to be chosen}\nathan{Right, thanks!}
\begin{align}\label{Eq:Aux-Medium-t3}
\begin{split}
    |\m W|&\le \sum_{j=0}^{k}{2m-t\choose 2m-t-j}{|A|-(2m-t)\choose j}{|B|-(2m-t)\choose j} (k')^{k-j} \\
    &\le \sum_{j=0}^{k}f_{m,k}(j),
\end{split}
\end{align}
where
\begin{equation}\label{eqfj} f_{m,k}(j):= {2m-t\choose 2m-t-j}{k\choose j}^2 (k')^{k-j}.\end{equation}
%We may bound $$f_{m,k}(j)\le \Big(\frac{(2m-t)k^2}{j^3 r}\Big)^j r^k.$$
Put 
\[
\m Y=\m W\cap{A\cup B\choose 2m-t+k}.
\]
Note that if for $C \in \m W$, we have $|C \cap I| = 2m-t-k$, then $|C \cap (A \setminus I)|= |C \cap (B \setminus I)|=k$, and consequently, $C \subset A\cup B$, as $|C|=2m-t+k$. Hence, all such sets $C$ are contained in $\m Y$. Therefore, we have 
\begin{equation}\label{Eq:Aux-Medium-t3.5}
|\m W\setminus \m Y|\le \sum_{j=0}^{k-1}f_{m,k}(j).
\end{equation}
The following steps will allow us to show that almost all sets in $\m F_m[\m W]$ contain a set from $\m Y$ (i.e., contain a subset of size $2m-t+k$ of the $(2m-t+2k)$-element set $F':=A\cup B$).

\medskip \noindent \emph{Step~3b: Bounding $|\m F|$ from below in terms of $f_{m,k}(k)$.} The bound is summarized in the following lemma. 
\begin{lem}\label{lemdensity2}Let $f_{m,k}(k)$ be as defined above. We have
\begin{equation}\label{eqrelationa}
|\m F| \geq a^{(m+k)}_{2m-t+2k}\ge \tfrac{1}{64}f_{m,k}(k)a^{(m)}_{2m-t+k}.\end{equation}
\end{lem}
\begin{proof}[Proof of Lemma~\ref{lemdensity2}]
The first inequality holds since by Lemma~\ref{lemadecay}(v), there exists an $(n-t)$-intersecting family in $S_n$ of size at least $a^{(m+k)}_{2m-t+2k}$, while $|\m F|$ is assumed to be the maximum size of an $(n-t)$-intersecting family in $S_n$. Thus, it is left to prove the second inequality.

%Recall \eqref{eqrestirctsize},~\eqref{eqrestrictsize2}.  \ohad{Fix the references, but the inequality is obvious.} \nathan{The reference should be to the equality case of~\eqref{eq:canon-lower}, plus the trivial claim that $d_s \leq s!$. However, this is valid under the assumption that the maximum $a_{2m-t+k}^{(m)}$ is obtained for $\m A_m[F]$ where $F=(D,(M,\sigma(M))$ with $M=\sigma(M) \subset D$. Is it clear that this is the case? We didn't claim this at the beginning of the proof of Lemma 10...}\ohad{I dont see why we need it here, isnt it just chhosing the size of $D,E$?} 
By Observation~\ref{obs:max-structure}, the value $a^{(m)}_{2m-t+k}$ is attained (also) by $\aaa_{m}[X]$ for some $X=(D,E)$ such that $E=(M,\sigma(M))$, where $M \subset D$ and $\sigma(M) \subset D$. Take such an $X$, and denote $x:=|D|$. 
We clearly have
$$a^{(m)}_{2m-t+k} = |\m A_m[X]|\le {n-x\choose m-x}(m-(2m-t+k-x))! = {n-x\choose m-x}(t-m-k+x)!.
$$
(See the proof of~\eqref{eq:canon-lower} above).
%This corresponds to fixing a set $F = (D,E)$, such that $|D| = x$ and $|E| = 2m-t+k-x$. 
%As $M \subset D$, we have $|D|\ge |E|$, and thus,  $x\ge m-\frac {t-k}2$. 
Let $X'=(D',E') \subset \mb X$ be obtained from $X$ by adding $k$ singleton-elements.
%Now, consider a set $F' = (D',E')$ with $|D'| = x+k$ and $|E'| = |E| = 2m-t+k-x$. (Also note that if $|D|\le m$ then $|D'|\le m+k$, and that $|D'|\ge |E'|=|E|$, which are the only inequalities we needed to check to verify the validity of such a choice.) 
As $M \subset D'$ and $\sigma(M) \subset D'$, it follows from~\eqref{eq:der-short} and~\eqref{eq:canon-lower} that 
%\nathan{The references here should be~\eqref{eq:der-short} and~\eqref{eq:canon-lower}, but $1/e$ should be replaced by $1/3$, and more importantly, this again uses the assumption that $E'=(M',\sigma(M'))$, where $M', \sigma(M') \subset D'$. Why can we assume this?} \andrey{Since the choice of $F'$ is completely in our hands, cannot we simply assume this when choosing it?}
\begin{align*}a^{(m+k)}_{2m-t+2k}&\ge |\aaa_{m+k} [X']| \\
    &\ge \tfrac 13{n-x-k\choose m+k-(x+k)}(m+k-(2m-t-x+k))!\\
    &= \tfrac 13{n-x-k\choose m-x}(t-m+x)!.
    \end{align*}
We would like to bound from below the ratio $a^{(m+k)}_{2m-t+2k}/a^{(m)}_{2m-t+k}$. As $k < k' \leq 2n^{2\epsilon}$ and $t,x,m \leq n^{(1+\epsilon)/2}$, we have 
\[
{{n-x-k}\choose{m-x}} \Big/{{n-x}\choose{m-x}}=1-o(1) \geq \tfrac{1}{2},
\]
for a sufficiently large $n$. As $M \subset D$, we have $x=|D|\geq \frac{2m-t+k}2$, and thus, $t-m+x-k \geq \frac{t-k}{2}>\frac{t}{4}$. Hence, $\frac{(t-m+x)!}{(t-m-k+x)!} > (t-m-k+x)^k > (t/4)^k$. Therefore,  
\begin{align*}\frac{a^{(m+k)}_{2m-t+2k}}{a^{(m)}_{2m-t+k}}&\ge \frac{\frac 13{n-x-k\choose m-x}(t-m+x)!}{{n-x\choose m-x}(t-m-k+x)!}
    \ge \tfrac{1}{3} \cdot \tfrac{1}{2} \cdot (t/4)^k = \tfrac{1}{6} \cdot (t/4)^k.
    \end{align*}
% In the second inequality, we used that $m = O(t)$ and thus $mk = o(n)$, as well as the following chain of inequalities 
%\begin{align*}
%    \frac{(t-m+x)!}{(t-m-k+x)!}&\ge ((t-m+x)!)^{\frac{k}{t-m+x}}\ge \Big(\Big(\frac{t-m+x}e\Big)^{t-m+x}\Big)^{\frac{k}{t-m+x}} \\ &=\Big(\frac{t-m+x}e\Big)^k.
%\end{align*}
%In the last inequality, we used that $x\ge m-\frac{t}2.$ 
Finally, we have 
\[
f_{m,k}(k) =\binom{2m-t}{2m-t-k}\le \frac{(2m-t)^k}{k!}\le \frac{t^k}{k!},
\]
and thus,
$$\frac{a^{(m+k)}_{2m-t+2k}}{f_{m,k}(k) a^{(m)}_{2m-t+k}}\ge \frac{\tfrac{1}{6}(t/4)^k}{t^k/k!} = \frac{k!}{6 \cdot 4^k} \ge \frac{1}{64},$$
as asserted.
%with an absolute constant $c>0$. The last inequality holds since the minimum of this expression is attained for $k = 5,$ for which the ratio is  bigger than $1/120.$
\end{proof}

\medskip \noindent \emph{Step~3c: Bounding $\sum_{j<k} f_{m,k}(j)$ from above in terms of $f_{m,k}(k)$.} In order to prove the desired upper bound, we first show that if~\eqref{eqchoicem} is obtained for $k>0$, then $2m-t$ must be `somewhat large'. 
\begin{lem}\label{lemk>0}
Assume that~\eqref{eqchoicem} is obtained for $k>0$. We have the following. 
\begin{enumerate}[label=(\alph*)]
  \item If $t\le n^{1/2}$ (and thus, $k' \leq 100\log n$), then $2m-t \ge n^{\epsilon/2}$;
  \item If $n^{1/2}\le t\le n^{(1+\epsilon)/2}$ (and thus, $k' \leq 2n^{2\epsilon}$), then $2m-t\ge n^{20\epsilon}$.
\end{enumerate}
In particular, in both cases we have
\begin{equation}\label{eq2m-t}
2m-t\ge 100n^{\epsilon/4} (k')^9,
\end{equation}
assuming $n$ is sufficiently large.
\end{lem}
\begin{proof}[Proof of Lemma~\ref{lemk>0}]
  The inequality~\eqref{eq2m-t} follows immediately from the inequalities (a) and (b), using the upper bounds on $k'$ in the two ranges of $t$. To prove (a) and (b), we consider the two ranges of values of $t$ separately.
  
  \smallskip \noindent \emph{Case~I: $t \leq n^{1/2}$.} Assume on the contrary that $2m-t< n^{\epsilon/2}$. As in this range, $k \le k'\le 100\log n$, for each $0 \leq j \leq k$ we have
\[
f_{m,k}(j) = {2m-t\choose 2m-t-j}{k\choose j}^2 (k')^{k-j}\le (n^{\epsilon/2} \cdot 100^3 (\log n)^3)^k\le n^{2\epsilon k/3},
\]
where the second inequality holds assuming $n$ is sufficiently large. By~\eqref{Eq:Aux-Medium-t3}, this implies 
\[
|\m W|\le (k+1)n^{2\epsilon k/3}.
\]
At the same time, as $m \geq p \geq \frac{t}{2} \geq \frac{n^{\epsilon}}{2}$, by Lemma~\ref{lemadecay}(i) we have
\[
\frac{a^{(m)}_{2m-t}}{ a^{(m)}_{2m-t+k}}\ge \left(\frac{m}{2e} \right)^k \geq (4e)^{-k} n^{\epsilon k}.
\]
As for each $W \in \m W$ we have $|\m F_m[W]| \leq a^{(m)}_{2m-t+k}$, this (together with the upper bound on $|\m W|$ obtained above) implies 
\begin{equation}\label{Eq:Aux-Medium-t4}
\frac{|\ff_m[\m W]|}{a^{(m)}_{2m-t}}\le \frac{|\m W|a^{(m)}_{2m-t+k}}{ a^{(m)}_{2m-t}}\le
\frac{(k+1)n^{2\epsilon k/3}}{(4e)^{-k} n^{\epsilon k}}\le   n^{-\epsilon/4},\end{equation}
assuming $n$ is sufficiently large. 

On the other hand, by Lemma~\ref{lemadecay}(v) there exists an $(n-t)$-intersecting family in $S_n$ of size at least $a^{(m)}_{2m-t}$, and thus, $|\ff|\ge a^{(m)}_{2m-t}$. By~\eqref{eqchoicem}, this implies
\[
|\ff_m[\m W]|\ge (2k')^{-2} |\m F| \geq (200\log n)^{-2} a^{(m)}_{2m-t},
\]
which contradicts~\eqref{Eq:Aux-Medium-t4} for a sufficiently large $n$. 

\smallskip \noindent \emph{Case~II: $t \geq n^{1/2}$.} The argument in this case is similar to the argument in the previous case. We assume on the contrary that $2m-t< n^{20\epsilon}$. As in this range, $k \le k'\le 2n^{2\epsilon}$, for each $0 \leq j \leq k$ we have
$f_{m,k}(j) \le (n^{20\epsilon} \cdot 2^3 \cdot ( n^{6\epsilon}))^k\le n^{27\epsilon k},$ and thus, $|\m W| \leq (k+1)n^{27\epsilon k}$.

By Lemma~\ref{lemadecay}(i), we have 
$a^{(m)}_{2m-t} / a^{(m)}_{2m-t+k}\ge \left(\frac{n^{(1-\epsilon)/2}}{2e} \right)^k \geq (2e)^{-k} n^{(1-\epsilon) k/2},$ and hence, we obtain
\[
\frac{|\ff_m[\m W]|}{a^{(m)}_{2m-t}}\le \frac{|\m W|a^{(m)}_{2m-t+k}}{ a^{(m)}_{2m-t}}\le
\frac{(k+1)n^{27\epsilon k}}{(2e)^{-k} n^{(1-\epsilon) k/2}}\le   n^{-10\epsilon},
\]
where the last inequality holds since $\epsilon \leq 0.01$. 

On the other hand, by Lemma~\ref{lemadecay}(v) we obtain 
\[
|\ff_m[\m W]|\ge (2k')^{-2} |\m F| \geq (2n^{2\epsilon})^{-2} a^{(m)}_{2m-t},
\]
a contradiction. This completes the proof of the lemma.
\end{proof}

%We can similarly analyze the case if $n^{1/2}\le t\le n^{(1+\epsilon)/2}$ and $(2m-t)\le n^{10\epsilon}$.\ohad{$f_{m,k}(j)\le (n^{10\epsilon}\cdot 4 \cdot k)^k\le n^{12\epsilon k}$, $\frac{a^{(m)}_{2m-t}}{ a^{(m)}_{2m-t+k}}\ge 2^{-k} n^{k/2}$, I think we can say the calculations in short} The calculations and conclusion are the same. We only need to substitute the value $k_p = n^{\epsilon}$ and bound the ratio $\frac{a^{(m)}_{2m-t}}{ a^{(m)}_{2m-t+k}}\ge n^{1/2-\epsilon}.$

%The verification of \eqref{eq2m-t} is a straightforward calculation.
%\end{proof}

Now we are ready to accomplish the goal of Step~3 -- showing that most of the sets in $\ff_m[\m W]$ admit a simple common structure. This is done by bounding $\sum_{j<k} f_{m,k}(j)$ from above in terms of $f_{m,k}(k)$ and combining this bound with the lower bound on $|\m F|$ in terms of $f_{m,k}(k)$ obtained in Step~3b. 
\begin{lem}\label{lemy=w}
  Let $\m W:= \m W_m^{(2m-t+k)}$ be a family that satisfies~\eqref{eqchoicem}. There exists a set $F'$ of size $2m-t+2k$, such that for $\m Y=\m W \cap \binom{F'}{2m-t+k}$, we have %\ohad{$\ge$?} \nathan{To Ohad: I don't understand this comment, can you please explain?}\ohad{$|\m F_m[\m Y]|\ge(1-n^{-\epsilon/5})|\m F_m[\m W]|$} \nathan{Right!}
  $|\m F_m[\m Y]| \geq (1-n^{-\epsilon/5})|\m F_m[\m W]|$.
\end{lem}
\begin{proof}[Proof of Lemma~\ref{lemy=w}]
If~\eqref{eqchoicem} is obtained for $k=0$, the assertion holds trivially, since in this case $\m W$ consists of a single set $F$ of size $2m-t$, and thus, $\m F_m[\m W]=\m F_m[F]$. Hence, we assume that~\eqref{eqchoicem} is obtained for $k>0$. We define $A,B$ as in Step~3a above, set $F':=A \cup B$ and $\m Y:= \m W \cap \binom{F'}{2m-t+k}$.

%We assume that we are in one of the regimes stated in Lemma~\ref{lemk>0}. Note that 
By the definition of $f_{m,k}(j)$, for any $j\in [0,k-1]$ we have
\begin{equation}\label{eqratiof}
\frac {f_{m,k}(j)}{f_{m,k}(j+1)} = \frac {j+1}{2m-t-j}\Big(\frac {j+1}{k-j}\Big)^2 \cdot k' = \frac{(j+1)^3 k'}{(2m-t-j)(k-j)^2}.
%\le \frac{(k_p+1)^3 r}{(2m-t-k_p)}.
\end{equation}
As $2m-t\ge 100n^{\epsilon/4} (k')^9$ by~\eqref{eq2m-t}, this implies 
%we get that for any $j\in [0,k-1]$ the equality \eqref{eqratiof} implies
$$\frac {f_{m,k}(j)}{f_{m,k}(j+1)}\le n^{-\epsilon/4}(k')^{-5}.$$
Summing over $j$ and using~\eqref{Eq:Aux-Medium-t3.5}, we get
\begin{equation}\label{eqfk}
|\m W \setminus \m Y| \leq \sum_{i=0}^{k-1}f_{m,k}(i) \le  n^{-\epsilon/4}(k')^{-4} f_{m,k}(k).
\end{equation}
As for each $W \in \m W$ we have $\m F_m[W] \leq a_{2m-t+k}^{(m)}$, this implies
\begin{equation}\label{Eq:Aux-Medium-t5}
|\ff_m[\m W\setminus \m Y]|\le \Big(\sum_{i=0}^{k-1}f_{m,k}(i)\Big)a_{2m-t+k}^{(m)} {\le} n^{-\epsilon/4}(k')^{-4} f_{m,k}(k) a_{2m-t+k}^{(m)}.
\end{equation}
On the other hand, by~\eqref{eqchoicem} and Lemma~\ref{lemdensity2}, we have
\begin{equation}\label{eqlbf1}|\ff_m[\m W]|\ge (2k')^{-2} |\ff|\ge (2k')^{-2} \cdot \tfrac{1}{64}f_{m,k}(k)a_{2m-t+k}^{(m)}.
\end{equation}
Combination of~\eqref{Eq:Aux-Medium-t5} and~\eqref{eqlbf1} yields
\begin{equation}\label{Eq:Aux-Medium-t6} 
|\m F_m[\m Y]| \geq (1-n^{-\epsilon/5})|\m F_m[\m W]|
\geq (1-n^{-\epsilon/5}) \cdot
(2k')^{-2} \cdot \tfrac{1}{64}f_{m,k}(k)a_{2m-t+k}^{(m)}
, 
\end{equation}
%\ohad{I think it is written wierd this way, I would write it as Combination of~\eqref{Eq:Aux-Medium-t5} and~\eqref{eqlbf1} yields
%$|\ff_m[\m W]| \ge n^{-\epsilon/5}|\ff_m[\m W\setminus \m Y]| $  } \nathan{I wrote in this form since I want to refer to this formula later on, and this form helps the later reference.}
provided $n$ is sufficiently large. This completes the proof of the lemma.
\end{proof}

\medskip \noindent \textbf{Step~4: Finding a simple sub-structure within $\m F$.}
In this step we leverage the simple structure found inside $\m W$ into a simple structure inside the entire family $\m F$, thus completing the proof of Theorem~\ref{thmfm2}. 

Recall that in Step~3 we showed that if~\eqref{eqchoicem} holds for $\m W:=\m W_m^{(2m-t+k)}$, where $m\in [p, p+k_p], k\in [0,k_{m}]$, then there exists a set $F'$ of size $2m-t+2k$, such that almost all sets in $\m F_m[\m W]$ contain a set belonging to $\m Y:= \m W \cap \binom{F'}{2m-t+k}$. In the case where~\eqref{eqchoicem} is obtained for $k=0$, $F'$ is the unique set in $\m W$, and otherwise, $F'=A \cup B$, where $A,B \in \m W$ satisfy $|A \cap B|=2m-t$. 
Our goal is to prove the assertion of Theorem~\ref{thmfm2} -- namely, that there exist $t/2 \leq m' \leq t$, an extendable set $F'$ of size $2m'-t$ of the form $F' = \{D,(M,\sigma_0(M))\}$, where $\sigma_0 \in S_n$, $M\subset D$ and $\sigma_0(M) \subset D$, and a family $\m H=\cup_{i=0}^{t}\m H_i$,  
where 
\[
\m H_i:=\{A\in \m A_i: |A\cap F'|\ge i-(t-m')\},
\]
such that $\m H$ corresponds to an $(n-t)$-intersecting family in $S_n$ and $|\m F \cap \m H| \geq (1-n^{-\epsilon/2})|\m F|$. We will show that the assertion holds for $m'=m+k$ and the set $F'$ defined above (or a slight modification of it).

\smallskip \noindent \emph{Showing the structural requirements on $F'$.} First, we observe that $\m Y$ is `fairly dense' inside $\binom{F'}{2m-t+k}$. Indeed, as for any $Y\in \m Y$ we have $|\aaa_m[Y]|\le a^{(m)}_{2m-t+k}$,~\eqref{Eq:Aux-Medium-t6} implies 
\begin{align}\label{eqlbony}
\begin{split}
|\m Y|&\ge (1-n^{-\epsilon/5}) \cdot
(2k')^{-2} \cdot \tfrac{1}{64}f_{m,k}(k) \geq 2^{-9} (k')^{-2} {2m-t\choose 2m-t-k} \\
&\geq 2^{-10}(k')^{-2}{2m-t+2k\choose 2m-t+ k}, 
\end{split}
\end{align}
where the last inequality holds due to~\eqref{eq2m-t}.

This allows us to deduce that $F'$ is extendable. In the case where~\eqref{eqchoicem} is obtained for $k=0$, the assertion is clear, as in this case $F'=F \in \m W$, and thus by the construction of $\m W$, $F'$ is included in some element of $\m F_m$, which implies that it is extendable. Assume that~\eqref{eqchoicem} is obtained for $k>0$, and assume on the contrary that $F'=A \cup B$ is not extendable. Denoting $A \cup B=(D,E)$, this means that there exist distinct pairs $(x_1,y_1),(x_2,y_2) \in E$ such that either $x_1=x_2$ or $y_1=y_2$. As all elements of $\m W$ are extendable and $\m Y \subset \m W$, this implies that no set in $\m Y$ can contain both $(x_1,y_1)$ and $(x_2,y_2)$. Hence, using a union bound and~\eqref{eq2m-t}, we obtain
\begin{align*}
|\m Y|&\le 2{2m-t+2k-1\choose 2m-t+k}=\frac {2k}{2m-t+2k}{2m-t+2k\choose 2m-t+k} 
\\ &\leq \tfrac{1}{50}n^{-\epsilon/4} (k')^{-8}{2m-t+2k\choose 2m-t+k},
\end{align*}
which contradicts \eqref{eqlbony} for a sufficiently large $n$. 

The extendability of $F'$ implies that we can write $F'=(D,E)$, where $E=(M,\sigma_0(M))$ for some $\sigma_0 \in S_n$. We claim that we can assume w.l.o.g.~that  
$M\subset D$ and $\sigma_0(M) \subset D$. Indeed, if $M \cup \sigma_0(M) \not \subset D$, then for any $u \in M \cup \sigma_0(M) \setminus D$, we can replace $F'$ with $F''=(D \cup \{u\}, E \setminus \{(u,\sigma_0(u))\})$ if $u \in M$, or $F''=(D \cup \{u\}, E \setminus \{(\sigma_0^{-1}(u),u)\})$ if $u \in \sigma_0(M)$. $F''$ is an extendable set of size $2m'-t$ and any   
set representation $F_\sigma \subset \mb X$ of a permutation $\sigma \in S_n$ such that $F_\sigma \supset F'$, must satisfy $F_\sigma \supset F''$ (as in this case, any element in $M \cup \sigma_0(M)$ is a moving point of $\sigma$). Hence, $F'$ can be replaced by $F''$ in the above argument. We can perform further such replacements, until there are no elements in $M \cup \sigma_0(M) \not \subset D$. Therefore, we can indeed assume that $M\subset D$ and $\sigma_0(M) \subset D$, as asserted.

\smallskip \noindent \emph{Showing the requirements on $m'$ and $\m H$.} Regarding $m'$, as $\m W_m^{(2m-t+k)} \neq \emptyset$, we have $2m-t+k \leq m$, 
%\ohad{Isnt it $2m-t+k \le 2m$?} \nathan{No, it is $2m-t+k \leq m$, which holds since $\m W_{2m-t+k}^{(m)} \neq \emptyset$.}
and thus, $m'=m+k \leq t$. On the other hand, $m' \geq m \geq p \geq t/2$. Hence, $t/2 \leq m' \leq t$, as asserted. 

Regarding $\m H$, for any $i_1,i_2$ and any sets $S_1\in \m H_{i_1}$, $S_2\in \m H_{i_2}$, by the definition of $\m H$ we have
\begin{align*}
    |(A\cup B)\cap S_1\cap S_2|&\ge (m'+i_1-t)+(m'+i_2-t)-(2m'-t) \\
    &= i_1+i_2-t.
\end{align*} 
Hence, Claim~\ref{clatranslation} implies that $\m H$ corresponds to an $(n-t)$-intersecting family of permutations.

\smallskip \noindent \emph{Showing that most elements of $\m F$ are contained in $\m H$.} Finally, we use the density of $\m Y$ in $\binom{2m-t+2k}{2m-t+k}$ to show that 
for any $p \leq i \leq p+k_p$ and any $S\in \m S_{i}$, we have $|S \cap F'| \geq m'+i-t,$ and thus, $S \in \m H$. By~\eqref{eqmostf-small}, this implies that $|\m F \cap \m H| \geq (1-n^{-\epsilon/2})|\m F|$, as asserted in the theorem. 

%\ohad{I guess the intuition without the details is that since $S,Y$ are $m+m'-t$ intersecting, and $Y$ can be almost anything in $A\cup B $, Then if $ L=S\cap (A\cup B) < m+m'-t+k$, We can choose $Y$ to be $2m-t+2k-|L|$ coordinates out of $(A\cup B) \setminus L$ and $|L|-k$ coordinates from $L$, but then it is not $m+m'-t$ intersecting with $L$ (and so with S). Only need to verify that choosing all the coordinates that are possible out of $L$ is the probable thing for random $Y$. But this is since if $|L| =m+m'-t+k-1 $, $2m-t+2k-|L| = m-m'+k \approx 2k_p$ is very small compares to $2m-t$, so not taking everything is not likely}
As $\m Y \subset \m W$, it follows from Step~2 that the families $\m S_{i}$ and  $\m Y$ are cross $(m+i-t)$-intersecting. For $k=0$, this immediately implies the conclusion. Hence, we assume that~\eqref{eqchoicem} is obtained for $k>0$ and assume on the contrary that for some $p \leq i \leq p+k_p$ and some $S \in \m S_i$, we have $|S \cap F'| \leq m'+i-t-1$. 
Fix $X\subset F'\setminus S$ of size $(2m'-t)-(m'+i-t-1)= k+(m-i)+1$. As $k \leq k_m$ and $m,i \in [p,p+k_p]$, we have 
\[
|X| \le k_m+k_p+1 \le 2k'.
\]
Hence, by a union bound and~\eqref{eq2m-t}, the fraction of sets in ${F'\choose 2m-t+k}$ that do not contain $X$ is at most 
\[
\frac{k \cdot 2k'}{2m-t+2k} \leq \tfrac{1}{50}n^{-\epsilon/4} (k')^{-7}.
\] 
By~\eqref{eqlbony}, this is much less than the fraction of sets in ${F'\choose 2m-t+k}$ that belong to $\m Y$. Hence, there exists $S'\in \m Y$ that contains $X$. As $X \cap S = \emptyset$, we have
\[
|S'\cap S|\le |S'|-|X| = (2m-t+k)-(k+m-i+1)<m+i-t,
\]
which contradicts the fact that $\m S_i$ and $\m Y$ are cross $(m+i-t)$-intersecting. This completes the proof of Theorem~\ref{thmfm2}.
\end{proof}

\section{Finding a Simple Sub-Structure Within $\m F$, for a Large $t$}
\label{sec:Large_t}

In this section, we study $(n-t)$-intersecting families, for $n^{(1+\epsilon)/2} \leq t \leq n-n^{1-\frac{\epsilon}{8}}$. Theorem~\ref{Thm:Main} asserts that in this range, the maximum size of an $(n-t)$-intersecting family is attained by a family of the form
\[
\m F_{n,n-t,(t-r)/2}=\{\sigma \in S_n: |\mathrm{Moving}(\sigma) \cap \{1,2,\ldots,n-r\}| \leq \tfrac{t-r}{2}\}.
\]
The value of $r$ for which the maximum is attained increases monotonically from $r \approx n^{(1+\epsilon)/2}$ for $t=n^{(1+\epsilon)/2}$ to $r=t$ for all $t > n/2$ (see Lemma~\ref{lemsizeofmaxFrankl} below). 

The approximation result we obtain in this range is almost exactly the same as in the `medium $t$' range (namely, Theorem~\ref{thmfm2} above; the only difference is replacing $n^{-\epsilon/2}$ in the upper bound by $n^{-\epsilon/6}$). The intuitive explanation of the result is the same as the explanation given right after  Theorem~\ref{thmfm2}, and thus we omit it here. 
\begin{thm}\label{thmfm3}
    For any $\epsilon \leq 0.01$, there exists $n_0 \in \mathbb{N}$ such that the following holds for all $n \geq n_0$ and all $n^{(1+\epsilon)/2} \leq t \leq n-n^{1-\frac{\epsilon}{8}}$. 
    Let $\m G \subset S_n$ be a maximum-size $(n-t)$-intersecting family of permutations. Let $\m F' \subset \mb [n] \cup \{(i,j):i,j \in [n], i \neq j\}$ be the family of sets that corresponds to $\m G$ in the representation of permutations by sets presented in Section~\ref{sec:PermutationstoSets}. Then there exist $t/2 \leq m' \leq t$, an extendable set $F'$ of size $2m'-t$ of the form $F' = \{D,(M,\sigma_0(M))\}$, where $\sigma_0 \in S_n$, $M\subset D$ and $\sigma_0(M) \subset D$, and a family $\m H=\cup_{i=0}^{t}\m H_i$,  
    where 
    \[\m H_i:=\{A\in \m A_i: |A\cap F'|\ge i-(t-m')\},
    \]
    such that $|\m F' \cap \m H| \geq (1-n^{-\epsilon/6})|\m F'|$. 
    \end{thm}

\begin{proof}[Proof of Theorem~\ref{thmfm3}]
Let $n,t,\epsilon$ be as in the statement of the theorem and let $\m G \subset S_n$ be a maximum-size $(n-t)$-intersecting family of permutations. The proof consists of several steps.

\medskip \noindent \textbf{Step~1: Applying the `iterative spread approximation' simplification to the family $\m G$.} In this range of values of $t$, transforming $\m G$ to a family $\m F=\cup \m F_i$ of sets and applying to it the `iterative spread approximation' (like we did in the proof of Theorem~\ref{thmfm2}) is not sufficient, since the size of each set in $\m F_i$ (which is $2i$) is so much larger than the `intersection size' $2i-t$, that even after the simplification, the size of the sets remains too large for being exploited. To overcome this, we apply the iterative spread approximation lemma directly to the family $\m G$, at the expense of moving to the more complex setting of partial permutations. We view each $\sigma \in \m G$ as a set of $n$ ordered pairs, and thus, we view $\m G$ as an $n$-element subset of $\mb X'':= [n] \times [n] = \{(i,j):i,j \in [n]\}$. For any $0 \leq j \leq n$, we denote 
\[
\alpha_n^{(j)}:=|\max_{Z \subset \mb X'', |Z|=j}\{\sigma \in S_n:Z \subset \sigma\}|.
\]
The following statement describes the approximating family we obtain for $\m G$. 
\begin{lem}\label{lemregime3step1}
    Let $n,\epsilon,t,$ and $\m G$ be as defined above. View $\m G$ as a subset of $\binom{[n] \times [n]}{n}$ in the natural way. There exists an $(n-t)$-intersecting family $\m Q \subset \m P([n] \times [n])$, such that:
    \begin{enumerate}
        \item Each set in $\m Q$ is of size at most $(n-t)+(n-t)^{\frac{1}{2}-\frac{\epsilon}{8}}$, and 

        \item $|\m G\setminus \m G[\m Q]| \leq n^{-9}|\m G|$. 
    \end{enumerate} 
\end{lem}

\begin{proof}[Proof of Lemma~\ref{lemregime3step1}]
We would like to apply Theorem~\ref{thmapprox1} with $n^2$ in place of $n$, $n$ in place of $k$, $n-t$ in place of $t$, $n$ in place of $R_1$, $n^{\frac{1+(\epsilon/2)}{2}}$ in place of $R$, $10\log_2 n$ in place of $\sigma$,
the set $\mb X''$ in place of $[n]$, the family $S_n$ viewed as a subset of $ \binom{\mb X''}{n}$ in place of $\m A \subset \binom{[n]}{k}$, $\alpha_n^{(j)}$ in place of $a_j$, and the $(n-t)$-intersecting family $\m G \subset S_n$ in place of $\m F \subset \m A$. Let us verify that the assumptions of the theorem are satisfied. 

To see that the weak spreadness assumptions on $S_n$ are satisfied, note that for each $Z \subset \mb X''$, we have $|\{\sigma \in S_n:Z \subset \sigma\}|=(n-|Z|)!$, if $Z$ does not contain two distinct pairs $(i_1,j_1),(i_2,j_2)$ such that either $i_1=i_2$ or $j_1=j_2$, and $\{\sigma \in S_n:Z \subset \sigma\}=\emptyset$ otherwise. Hence, $\alpha_n^{(j)}=(n-j)!$ for all $j$, and thus, for any $s,t$, we have 
\[
\alpha_n^{(s+t)}= [(n-s)(n-s-1)\cdot \ldots \cdot (n-s-t+1)]^{-1} \alpha_n^{(s)}.
\]
As $t \geq n^{(1+\epsilon)/2}$, it follows that for a sufficiently large $n$, the family $S_n \subset \binom{\mb X''}{n}$ is weakly $(n^{\frac{1+(\epsilon/2)}{2}},\ell)$-spread for any $\ell \leq (n-t)+(n-t)^{1/6}(t+\sqrt{n})^{2/3}$, which includes the range of values of $t'$ for which weak $(R,t')$-spreadness is required in the theorem (see~\eqref{Eq:Aux-Iterative1}). 
%\nathan{This suffices only up to $t < 3n/4$. We can make it work until any $c'n$, by replacing $10$ with an appropriate constant, but if we want larger values of $t$, we cannot reach $c't$.} \andrey{I don't understand, what is the problem here. In order to apply the spread approx theorem, we only need to make sure that $R\ge Const\cdot n^{0.5+\epsilon,}$ right? And we can always guarantee such $(R,(n-t))$-spreadness in our regime. Actually, the smaller $n-t$ is, the easier it is, I believe. What am I missing?} \nathan{In your draft, you wrote that we apply the iterative spreadness approximation with $R=t/10$. I thought you need $R$ to be that large, in order to use the $R$-spreadness of the pieces guaranteed by the theorem at a later stage (like we did in the medium $t$ case). If we are fine with $R=n^{(1+\epsilon)/2}$, then indeed there is no problem here. Are we fine with this?}\andrey{I was not  careful with the choice of parameters in my draft of this statement - since I was anyway concerned with the cases of $n-t$ large. But what we need, say, for the application of the new argument (pp 38-39) is that $k\ll \sqrt n$. $k$ is essentially the uniformity after the set interpretation, but is also the maximum uniformity of sets in the iterative spread approximation, minus $(n-t)$ (if we talk about an $(n-t)$-intersecting family).   For that in the iterative spread approximation is sufficient that $R\gg \sqrt n.$ In particular, $R = n^{(1+\epsion)/2}$ would be fine.}
In the other direction, it is clear that for any $t' \leq t$, we have $\alpha_n^{(t')} \leq n^{t-t'} \alpha_n^{(t)}$. Hence, the two spreadness requirements on $S_n$ are indeed satisfied.

The assumptions $R \geq 2^{15}\log_2(4k)$ and $R \geq 2^{30}(t^{1/2}\log_2 t+\log_2 R_1)+200\sigma$, which in our case can be unified to 
\[
n^{\frac{1+(\epsilon/2)}{2}} \geq \max \left\{2^{15}\log_2(4n), 2^{30}((n-t)^{1/2}\log_2(n-t)+\log_2 n)+2000 \log_2 n \right\},
\]
clearly hold assuming $n$ is sufficiently large.

Therefore, we can apply Theorem~\ref{thmapprox1} to approximate $\m G$ by an $(n-t)$-intersecting family $\m Q$ consisting of `small' subsets of $\mb X''$ that represent partial permutations. By Assertion~(iii) of the theorem, $\m R:= \m G \setminus \m G[\m Q]$ satisfies 
\[
|\m R| \leq n^{-10} \cdot 4 \log_2(t \cdot n^{\frac{1+(\epsilon/2)}{2}}+10\log_2 n)\alpha_n^{(n-t)} \leq n^{-9} t! \leq n^{-9}|\m G|,
\]
where the last inequality holds  as the $(n-t)$-intersecting family  $\m F_{n,n-t,0} \subset S_n$ has size $t!$, while $\m G$ is assumed to be a maximum-sized $(n-t)$-intersecting family in $S_n$. This proves Assertion~(2) of the lemma.

\smallskip As for the sizes of the sets in $\m Q$, we can obtain an improved bound using the `in addition' part of Theorem~\ref{thmapprox1}. 
The additional assumption $R \geq 2^{30}t^{\alpha}$, which in our case reads as $n^{\frac{1+(\epsilon/2)}{2}} \geq 2^{30}(n-t)^{\alpha}$,
holds for $\alpha = \log_{n-t}(n^{\frac{1+(\epsilon/2)}{2}}/2^{30})$. 
As in our range, $t\ge n^{(1+\epsilon)/2}$, for a sufficiently large $n$ we have 
$\log_{n-t}(n^{\frac{1+(\epsilon/2)}{2}}/2^{30}) \ge \frac{1}{2}+\frac{\epsilon}{5}$. Hence, 
the size of each set in $\m Q$ is at most 
\[
(n-t)+(n-t)^{1-(\frac{1}{2}+\frac{\epsilon}{5})}+40\log_2 n+4\log_2(n-t+1) \leq (n-t)+(n-t)^{\frac{1}{2}-\frac{\epsilon}{8}},
\]
where the inequality holds assuming $n$ is sufficiently large.
This proves Assertion~(1) of the lemma.
\end{proof}

\noindent \emph{Transforming partial permutations to sets.} Consider the approximating family $\m Q$ obtained in the lemma, and view it as a family of partial permutations. As was shown above, $\m Q$ is $(n-t)$-intersecting and the domain of any $\sigma \in \m Q$ is of size at most $(n-t)+k$, where $k:=(n-t)^{\frac{1}{2}-\frac{\epsilon}{8}}$.  We would like to represent it by a family $\m F$ of subsets of $\mb X'=[n] \sqcup ([n] \times [n])$, in the way presented in Section~\ref{subsec:PartialPerms_to_Sets}. We briefly recall this representation in the following paragraphs. As was explained in Section~\ref{subsec:PartialPerms_to_Sets}, we may assume w.l.o.g.~that   
$\m Q$ contains a \emph{partial identity permutation} on some set $I$ of size $n-t \leq |I| \leq n-t+k$. We replace each $\sigma \in \m Q$ with all partial permutations $\sigma'$ containing it, whose domain contains $I$ (i.e., the set $\{\sigma' \in \Sigma_n: \sigma \subset \sigma', I \subset I_{\sigma'}\}$), thus obtaining an $(n-t)$-intersecting family $\m Q'$ in which the domains of all permutations contain $I$. We have 
$\m G[\m Q] = \m G[\m Q']$ (that is, each $\sigma \in \m G$ that contains some element of $\m Q$ must contain some element of $\m Q'$). 

Throughout the rest of this section, for each (partial or full) permutation $\sigma$, we set 
$D_\sigma := \mathrm{Moving}(\sigma)\cap I$, $M_\sigma:=(I_\sigma\setminus I)\cup D_\sigma$,
$E_{\sigma} := \{(i,\sigma(i)): i\in M_\sigma\}$, and $\m F_{\sigma}:= D_\sigma \sqcup E_{\sigma}$. We use the notations $E_\sigma := (M_\sigma, \sigma(M_\sigma))$ and $F_{\sigma}:=(D_{\sigma},E_{\sigma})$. By Observation~\ref{obssize}, for each $\sigma \in \m Q'$, we have 
$|D_\sigma|\le |I|-(n-t) \leq k, |M_\sigma|\le |D_{\sigma}|+k \le 2k,$ and $|F_\sigma|\le 2|D_{\sigma}|+k \leq 3k$.

Denote $t':=|I|-(n-t)$, so $|D_\sigma|\le t'$ for all $\sigma \in \m Q'$. To transform $\m Q'$ to a family of sets,
we first decompose  $\m Q' = \cup_{i=0}^{t'} \m Q'_i$, where $\m Q'_i:= \{\sigma\in \m Q': |D_{\sigma}|=i\}.$ Then, we transform $\m Q'$ into the family $\m F:= \sqcup_{i=0}^{t'} \m F_i \subset \m P(\mb X')$, where
$\m F_i = \{F_\sigma: \sigma\in \m Q'_i\}$.
For each $0 \leq i \leq t'$, we denote 
\[
\m B_i := \{F \subset \mb X': F=F_{\sigma},\sigma \in \Sigma_n^{(j)}, j \in \{n-t,\ldots,n-t+t'+k\}, |D_{\sigma}|=i, I_{\sigma} \supset I\}.
\]
Clearly, for each $i$ we have $\m F_i \subset \m B_i$. 

By Claim~\ref{clatranslation2}, for any $i,j$, the families $\m F_i$  and $\m F_j$ are cross $(i+j-t')$-intersecting. In particular, each $\m F_i$ is a $(2i-t')$-intersecting family of subsets of $\mb X'$. 

\medskip \noindent \textbf{Step~2: Iteratively applying a weighted variant of the `peeling simplification' procedure.}
%\andrey{This is what I actually call peeling-simplification procedure.} \nathan{I used this term throughout the paper, since the term `peeling-simplification' which you used your draft sounded less clear to me. If you prefer to use `peeling off' simplification (and especially if you use this term for this simplification in other papers), please tell me and I'll change throughout the paper.}
Fix $\lceil t'/2 \rceil \leq i \leq t'$. (As we shall show in Claim~\ref{Cl:Aux-Large-t1} below, the contribution of the $\m F_i$'s for $i< \lceil t'/2\rceil$ is negligible in the range of values of $t$ we consider now). 

$\m F_i$ is a $(2i-t')$-intersecting family of sets of size at most $2i+k$. While the difference between the size of the sets in $\m F_i$ and the intersection size is much smaller than at the beginning of Step~1 (where the size of the sets was $n$ and the intersection size was $n-t$), it is still too large for being exploited. In this step, we reduce this difference to a constant, using another approximation process based on a weighted variant of the `peeling  simplification' procedure.  

As was described in Section~\ref{subsec:PartialPerms_to_Sets}, for each partial permutation $\sigma \in \Sigma_n$, we set the weight $\mu(\sigma)$ to be  the number of extensions of $\sigma$ to a full permutation on $[n]$. For a family $\m Z \subset \Sigma_n$ we set $\mu(\m Z)$ to be the number of full permutations that extend at least one permutation from $\m Z.$ For a set $F_\sigma$, where $\sigma \in \Sigma_n$, we define $\mu(F_{\sigma}):=\mu(\sigma)$, and for a family $\m X \subset \m P(\mb X')$, we define $\mu(\m X)$ to be equal to the weight of the corresponding family of partial permutations. Note that by Lemma~\ref{lemregime3step1}, we have
\begin{equation}\label{Eq:Aux-Large-t00}
    \mu(\m F) \geq |\m G[\m Q']|=|\m G[\m Q]| \geq (1-n^{-9})|\m G|.
\end{equation}
For any $\ell \leq i$, we set 
$b^{(i)}_\ell:=\max_{F:|F|=\ell}\mu(\bb_i[F])$.

We use a weighted variant of the notion of spreadness that slightly differs from the notion of $(r,\mu)$-spreadness introduced in Section~\ref{subsec:PartialPerms_to_Sets}. In this notion, the size $|\m S|$ of a family $\m S$ is replaced by $\mu(\m B_i[\m S])$, rather than by $\mu(\m S)$. Namely, we say that a family $\m S \subset \mb X'$ is $(r,\mu,\m B_i)$-spread if for any non-empty $X \subset \mb X'$, we have 
\[
\mu(\m B_i[\m S[X]]) < r^{-|X|} \mu(\m B_i[\m S]).
\]
Note that for $\m S \subset \m B_i$, $\mu(\m B_i(\m S[X]))$, coincides with $\mu(\m S[X])$ for any $X$, and thus, $(r,\mu,\m B_i)$-spreadness coincides with $(r,\mu)$-spreadness. 

The $r$-simplification process described in Section~\ref{sec:peeling_off} can be performed verbatim, with `$r$-spread' replaced by `$(r,\mu,\m B_i)$-spread'. We obtain the following weighted variant of the `peeling simplification' lemma (i.e., Lemma~\ref{lemsimplsimple} above). 
\begin{lem}\label{lemsimplsimple2} Let $\mb X',i,t', \m B_i, \mu$ be as defined above, and let $p,r \in \mathbb{N}$ be such that $r>p-(2i-t')+1$. Let $\m X \subset \m P(\mb X')$ be a $(2i-t')$-intersecting family in which each set is of size at most $p$. Let $\m S$ be an $(r,\mu,\m B_i)$-simplification of $\m X$, and denote the family of all $j$-element sets in $\m S$ by $\m S^{(j)}$. Then:
\begin{enumerate}
    \item $\m X = \m X[\m S]$;

    \item $\m S$ is $(2i-t')$-intersecting; and

    \item For any $j \geq 2i-t'$, we have
    \[\mu(\m B_i[\m S^{(j)})|)\le {j\choose 2i-t'} \cdot r^{j-(2i-t')} \cdot b_j^{(i)}.
    \]
\end{enumerate}
%Furthermore, if $\m X,\m W \subset \m P(\mb X')$ are cross $(2i-t')$-intersecting families, both consisting of sets of size at most $p$, and $\m S_{\m X},\m S_{\m W}$ are their respective $(r,\mu,\m B_i)$-simplifications, then each of the families $\m X,\m S_{\m X}$ cross $(2i-t')$-intersects each of the families $\m W, \m S_{\m W}$. \nathan{Maybe the `furthermore' is not needed.}
\end{lem}

\begin{proof}[Proof of Lemma~\ref{lemsimplsimple2}]
    The proof is almost identical to the proof of Lemma~\ref{lemsimplsimple}, with $\mb X'$ in place of $[n]$, $2i-t'$ in place of $t$ and $j$ in place of $i$; only two changes are needed. One change is replacing $r$-spreadness with $(r,\mu,\m B_i)$-spreadness throughout the proof, and correspondingly, replacing the size of each family by its `weighted measure' (e.g., replacing $|\m Z_j|$ by $\mu(\m B_i[\m Z_j])$) in the places where sizes of families appear -- namely,~\eqref{Eq:Aux_Simplification0.5},~\eqref{Eq:Aux_Simplification1}, and the proof of Observation~\ref{obs13} that is used at the end of the proof of Lemma~\ref{lemsimplsimple}. The second change is adapting one statement from the proof of Observation~\ref{obs13}. Namely, that proof uses the obvious fact that if $\m F \subset \binom{[n]}{\leq k'}$ and $|X'|=k'$, then  $|\ff(X')| \leq 1$. In our case, this inequality should be replaced by the upper bound $\mu(\m B_i[\m F[X']]) \leq \max_{|X|=k'} \mu(\m B_i[\m F[X]]) \leq b_{k'}^{(i)}$, in which the second inequality holds since for any $X$ with $|X|=k'$, $\m F[X]$ contains at most one element. As this argument is used for $k'=j$, the assertion of Lemma~\ref{lemsimplsimple2} is weaker by a factor of $b_{j}^{(i)}$ than the assertion of Lemma~\ref{lemsimplsimple}. The rest of the proof can be applied verbatim. \end{proof}

The following lemma allows us to approximate $\m F_i$ by a $(2i-t')$-intersecting family in which all sets are of size at most $(2i-t')+\frac{200}{\epsilon}$. The constant difference between the size of the sets in the family and the intersection size which we obtain here is much smaller than the difference we obtained in the `medium $t$' range (which was $n^{2\epsilon}$, see Lemma~\ref{lemregime2}). The reason for the `better' situation here is that the weak $(r,\mu,s)$-spreadness property of the families $\m B_i$ is much stronger than the weak $(r,s)$-spreadness property of the families $\m A_i$ we dealt with in the `medium $t$' range. (The difference is between $r \approx n^{1/2}$ for $\m B_i$, as shown in Claim~\ref{claimdecay}, and $r \approx t$ that can be as small as $n^{\epsilon}$ for $\m A_i$, as shown in Lemma~\ref{lemadecay}(i)).
\begin{lem}\label{lem:iterative_peeling_off}
        Let $n,\epsilon,t,t', \m F$ be as defined above, and for each $\lceil t'/2 \rceil  \leq i \leq t'$, let $\m F_i$ be as defined above. There exist families $\m S_i \subset \m P(\mb X')$ such that:
    \begin{enumerate}
        \item Each set in $\m S_i$ is of size at most $(2i-t')+\frac{200}{\epsilon}$;

        \item $\mu(\m F_i \setminus \m F_i[\m S_i]) \leq n^{-9} |\m G|$;  
        %$\mu(\m B_i[\m F_i \setminus \m F_i[\m S_i]]) \leq n^{-9} |\m G|$; 
        %\nathan{How does this bound translate to a direct bound on $\mu(\m F_i \setminus \m F_i[\m S_i])$, which is what we really need?}\andrey{Aren't we just counting the number of permutations corresponding to the LHS and to the RHS anyway? This is again  negligible w.r.t. the whole family, so we can ignore it.}

        \item For any $i,j$, the families $\m S_i$ and $\m S_j$ are cross $(i+j-t')$-intersecting. 
    \end{enumerate} 
\end{lem}

\begin{proof}
For each $\lceil t'/2 \rceil \leq i \leq t'$, we apply Lemma~\ref{lemsimplsimple2} iteratively to construct a sequence of families approximating $\m F_i$. At each iteration, we show that the contribution of the `large' sets in the approximating family to the measure $\mu(\m F_i)$
%$\mu(\m B_i[\m F_i])$ 
is small, and then we remove them from the approximating family, thus reducing the maximum size of sets in it. We show that eventually, we can throw out all sets of size $> (2i-t')+\frac{200}{\epsilon}$, and still remain with a good approximation of $\m F_i$. 

Formally, for each $i$ we set $\m S_i^0:=\m F_i$ and for each $\ell=1,2,\ldots,$ we apply the following two-step procedure:
\begin{itemize}
    \item Apply Lemma~\ref{lemsimplsimple2} to the family $\m S_i^{\ell-1}$, with $p_\ell^i = (2i-t')+6k\cdot 2^{-\ell}$ in place of $p$ and $r_\ell = 60k\cdot 2^{-\ell+1}$ in place of $r$, to obtain the approximating family $\m V_i^{\ell}$;

    \item Set $\m S_i^\ell:=\m V_i^{\ell}\cap {\mb X'\choose \le p_{\ell+1}^i}$. 
\end{itemize}
We stop the process at the step $\ell_0:=\min\{\ell: p_{\ell+1}^i-(2i-t') \leq \frac{200}{\epsilon}\}$, and set $\m S_i:=\m S_i^{\ell_0}$. 

It is clear that throughout the process, the assumptions of Lemma~\ref{lemsimplsimple2} are satisfied and so the lemma can be applied. By Lemma~\ref{lemsimplsimple2}, all families $\m V_i^\ell, \m S_i^\ell$ constructed during the process, and in particular, $\m S_i$, are $(2i-t')$-intersecting. By construction, the size of each set in $\m S_i$ is at most $(2i-t')+\frac{200}{\epsilon}$. This proves Assertion~(1) of the lemma. 

\smallskip \noindent \emph{Proving the cross intersection assertion.}
Now we show that for any $\lceil t'/2 \rceil \leq i_1,i_2 \leq t$, the families $\m S_{i_1}, \m S_{i_2}$ are cross $(i_1+i_2-t')$-intersecting. To this end, we prove by induction on $\ell$ that for each $\ell \geq 0$, the families $\m S_{i_1}^{\ell}, \m S_{i_2}^{\ell}$ are cross $(i_1+i_2-t')$-intersecting. The assertion holds for $\ell=0$, since $\m S_{i_1}^{0} = \m F_{i_1}$ and $\m S_{i_2}^{0} = \m F_{i_2}$ are 
cross $(i_1+i_2-t')$-intersecting by Claim~\ref{clatranslation2}. The induction step we now present combines components from the proofs of Lemmas~\ref{lemsimplsimple}(3) and~\ref{lemregime2}(3) above. 

Let $\ell \geq 1$. Assume that the assertion holds for $\ell-1$, and suppose on the contrary that there exist $i_1,i_2$ and $S_1\in \m S_{i_1}^{\ell},$ $S_2\in \m S_{i_2}^{\ell}$ such that $|S_1\cap S_2| = i_1+i_2-t'-j$ for some $j>0$. Assume w.l.o.g.~that $i_2\ge i_1$. By the construction, we have $|S_{1}|\le 2i_1-t'+6k \cdot 2^{-\ell-1} = 2i_1-t'+\frac{r_{\ell}}{40}$. %\ohad{$|S_{i_1}|$ replaced by $|S_{1}|$?}\nathan{Yes!}. 
%\andrey{should it be $40$ instead of $20$ in this calculation?}
%\nathan{Indeed. Thanks!}
%\le 2i_1-t+n^{0.02}$, where the second inequality holds assuming $n$ is sufficiently large, since $\epsilon \leq 0.01$. 
Hence,
\[
0\le |S_1\setminus S_2| \le 2i_1-t'+\tfrac{r_{\ell}}{40}-(i_1+i_2-t'-j)=i_1+\tfrac{r_{\ell}}{40}+j-i_2,
\]
and consequently, $i_2\le i_1+\frac{r_{\ell}}{40}+j$. Applying the same argument with the roles of $i_1,i_2$ interchanged and using the bound $i_2-i_1 \leq \frac{r_{\ell}}{40}+j$, we obtain 
\[
|S_2\setminus S_1| \le 2i_2-t'+\tfrac{r_{\ell}}{40}-(i_1+i_2-t'-j)=i_2+\tfrac{r_{\ell}}{40}+j-i_1\le \tfrac{r_{\ell}}{20}+2j.
\]
By the process in which $\m S_{i_1}^\ell$ was constructed,  either $S_1 \in \m S_{i_1}^{\ell-1}$ or there exists $\m U_{i_1}^{\ell-1} \subset \m S_{i_1}^{\ell-1}$ such that $\m U_{i_1}^{\ell-1}(S_1)$ is  $(r_{\ell},\mu,\m B_{i_1})$-spread and contains at least two elements. Similarly, either $S_2 \in \m S_{i_2}^{\ell-1}$ or there exists $\m U_{i_2}^{\ell-1} \subset \m S_{i_2}^{\ell-1}$ such that $\m U_{i_2}^{\ell-1}(S_2)$ is  $(r_{\ell},\mu,\m B_{i_2})$-spread and contains at least two elements. If $S_1 \in \m S_{i_1}^{\ell-1}$ and $S_2 \in \m S_{i_2}^{\ell-1}$ then $|S_{1} \cap S_{2}| \geq i_1+i_2-t'$, 
%\ohad{ $|S_{1} \cap S_{2}| \geq i_1+i_2-t'$?}\nathan{Yes!} 
contradicting the assumption. Hence, we assume w.l.o.g.~that regarding $S_1$, the latter holds.
Denote for simplicity $\m U_1:= U_{i_1}^{\ell-1}$. 
Set
\[
\m X_1:=\{S' \in \m U_1(S_1): |S' \cap (S_2 \setminus S_1)| < \lceil \tfrac{j}{2} \rceil \}. 
\]
We claim that $\mu(\m B_{i_1}(\m X_1)) \geq \frac{1}{2} \mu(\m B_{i_1}(\m U_1(S_1)))$ (and in particular, $\m X_1 \neq \emptyset$). Indeed, by the $(r_{\ell},\mu,\m B_{i_1})$-spreadness of $\m U_1(S_1)$ and a union bound, we have
\begin{align*}
\frac{\mu(\m B_{i_1}(\m U_1(S_1) \setminus \m X_1))}{\mu(\m B_{i_1}(\m U_1(S_1)))} &\leq 
\binom{|S_2\setminus S_1|}{\lceil \frac{j}{2} \rceil}r_{\ell}^{\lceil -j/2 \rceil}\le \Big(\frac{2e |S_2\setminus S_1|}{jr_{\ell}}\Big)^{\lceil j/2 \rceil} \\ &\le \Big(\frac{2e (\tfrac{r}{20}+2j)}{jr}\Big)^{\lceil j/2 \rceil} \leq \frac{1}{2}. 
\end{align*}
Hence, $\m X_1$ is an $(r_{\ell}/2,\mu,\m B_{i_1})$-spread family. 

If $S_2 \in \m S_{i_2}^{\ell-1}$, then for any $X_1 \in \m X_1$, the set $S'_1:= (S_1 \sqcup X_1) \in \m S_{i_1}^{\ell-1}$ satisfies 
\[
|S'_1 \cap S_2|=|S_1 \cap S_2|+|X_1 \cap (S_2 \setminus S_1)| < i_1+i_2-t'-j+\lceil \tfrac{j}{2} \rceil \leq i_1+i_2-t',
\]
contradicting the assumption that $\m S_{i_1}^{\ell-1}$ and $\m S_{i_2}^{\ell-1}$ are cross $(i_1+i_2-t')$-intersecting. 

If
$S_2 \not \in \m S_{i_2}^{\ell-1}$, we perform for $S_2$ the same process as for $S_1$, to obtain an $(r_{\ell}/2,\mu,\m B_{i_2})$-spread family $\m X_2 \subset \m S_{i_2}^{\ell-1}$, defined similarly to $\m X_1$.
Let $X_1 \in \m X_1$. As $X_1 \in \m U_1(S_1) \subset \m S_{i_1}^{\ell-1}$ and $S_1$ belongs to the $(2i_1-t')$-intersecting family $\m S_{i_1}^\ell$, we have 
\[
|X_1| \leq (2i_1-t')+6k \cdot 2^{-\ell}-(2i_1-t')=6k \cdot 2^{-\ell} < r_{\ell}/2. 
\]
Hence, there exists $X_2 \in \m X_2$ such that $X_1 \cap X_2 = \emptyset$, as otherwise, by the pigeonhole principle we would have 
\[
\mu(\m B_{i_2}(\m X_2(\{v\}))) > (\tfrac{r_{\ell}}{2})^{-1} \mu(\m B_{i_2}(\m X_2))
\]
for some $v \in X_1 \setminus (S_1 \cup S_2)$, contradicting the $(r_{\ell}/2,\mu,\m B_{i_2})$-spreadness of $\m X_2$. The sets 
$S'_1:= (S_1 \sqcup X_1) \in \m S_{i_1}^{\ell-1}$ and $S'_2:= (S_2 \sqcup X_2) \in \m S_{i_2}^{\ell-1}$ satisfy 
\begin{align*}
|S'_1 \cap S'_2|&=|S_1 \cap S_2|+|X_1 \cap (S_2 \setminus S_1)| + |(S_1 \setminus S_2) \cap X_2|+|X_1 \cap X_2| \\ &\leq i_1+i_2-t'-j+(\lceil \tfrac{j}{2} \rceil -1) + (\lceil \tfrac{j}{2} \rceil -1) < i_1+i_2-t',
\end{align*}
contradicting the assumption that $\m S_{i_1}^{\ell-1}$ and $\m S_{i_2}^{\ell-1}$ are cross $(i_1+i_2-t')$-intersecting. 
This completes the proof of Assertion~(3) of the lemma.

\smallskip \emph{Bounding from above the contribution of the removed sets.} It remains to show that for each $i$, $\mu(\m F_i \setminus \m F_i[\m S_i]) \leq n^{-9} |\m G|$. 
%$\mu(\m B_i[\m F_i \setminus \m F_i[\m S_i]]) \leq n^{-9} |\m G|$. 
By Lemma~\ref{lemsimplsimple2}, for each $\ell \geq 1$ we have $\m S_i^{\ell-1} = \m S_i^{\ell-1}[\m V_i^\ell]$, and hence, $\m F_i \setminus \m F_i[\m S_i] \subset \cup_{\ell \geq 1} (\m F_i[\m V_i^\ell \setminus \m S_{i}^\ell])$. Consequently, by a union bound, we have
\begin{equation}\label{Eq:Aux-Large-t0}
\mu(\m F_i \setminus \m F_i[\m S_i]) = \mu(\m B_i[\m F_i \setminus \m F_i[\m S_i]]) \leq \sum_{\ell=1}^{\ell_0} \mu(\m B_i[\m V_i^\ell \setminus \m S_i^\ell]), 
\end{equation}
where the equality holds since $\m F_i \subset \m B_i$ for all $i$. Fix $\ell \geq 1$. We have $V_i^\ell \setminus \m S_i^\ell = \m V_i^\ell  \cap (\cup_{j=p_{\ell+1}+1}^{p_\ell} {\mb X'\choose j})$. By the application of Lemma~\ref{lemsimplsimple2} at the $\ell$'th step of the process, for each $p_{\ell+1}+1 \leq j \leq p_\ell$ we have
\begin{align}\label{Eq:Aux-Large-t1}
    \begin{split}
        \frac{\mu(\m B_i[\m V_i^{\ell}\cap{\mb X'\choose j}])}{b^{(i)}_{j}} &\le {j\choose 2i-t'}r_\ell^{j-(2i-t')} = {j\choose j-(2i-t')}r_\ell^{j-(2i-t')} \\
        &\le \Big(\frac{ej}{j-(2i-t')}\Big)^{j-(2i-t')}\big(60k\cdot 2^{-\ell+1}\big)^{j-(2i-t')} \\ &\le \Big(\frac{3ek}{6k\cdot 2^{-(\ell+1)}}\cdot\big(60k\cdot 2^{-\ell+1}\big)\Big)^{j-(2i-t')} \\
        &\le (400 k)^{j-(2i-t')},
    \end{split}
\end{align}
where the penultimate inequality holds as $j \geq p_{\ell+1}$ and as by Observation~\ref{obssize} all sets in $\m F_i$ are of size at most $3k$.
Recall that by Claim~\ref{claimdecay}, for any $\ell$ such that $b^{(i)}_{\ell+1}>0$, we have 
\begin{equation}\label{Eq:Aux-Large-t4}
 \frac{b^{(i)}_\ell}{ b^{(i)}_{\ell+1}} \geq \frac12 \cdot \min \left\{\frac{n-t-k}{k}, t- 3k \right\} \geq \frac{1}{4} \cdot n^{(1-\frac{\epsilon}8)(\frac{1}{2}+\frac{\epsilon}{8})} \geq n^{1/2}, 
\end{equation}
where the second inequality holds since $n^{(1+\epsilon)/2} \leq t \leq n-n^{1-\frac{\epsilon}{8}}$ and $k=(n-t)^{\frac{1}{2}-\frac{\epsilon}{8}}$, and the third inequality holds for a sufficiently large $n$. Consequently, for any $j \geq 2i-t'$, we have $\frac{b^{(i)}_{2i-t'}}{b^{(i)}_{j}} \geq n^{\frac{j-(2i-t')}{2}}$. Combining this with~\eqref{Eq:Aux-Large-t1}, we obtain that for any $j \geq (2i-t')+\frac{100}\epsilon$, 
\[
\frac{\mu(\m B_i[\m V_i^{\ell}\cap{\mb X'\choose j}])}{b^{(i)}_{2i-t'}}
\le \frac{(400 k)^{j-(2i-t')}b^{(i)}_{j}}{b^{(i)}_{2i-t'}}\le n^{-\frac{\epsilon}{10} \cdot (j-(2i-t'))} \leq n^{-10}, \]
where the penultimate inequality holds for a sufficiently large $n$ since $k \leq n^{\frac{1}{2}-\frac{\epsilon}{8}}$. As all sets removed during the process are of size larger than $(2i-t')+\frac{100}{\epsilon}$, this implies that for each $1 \leq \ell \leq \ell_0$, 
\[
\mu(\m B_i[V_i^\ell \setminus \m S_i^\ell]) \leq \sum_{j=p_{\ell+1}+1}^{p_{\ell}} \mu \left(\m B_i \left[V_i^\ell \cap {\mb X'\choose j}\right] \right) \leq (p_\ell-p_{\ell+1})\cdot n^{-10} b^{(i)}_{2i-t'}.
\]
Combining this with~\eqref{Eq:Aux-Large-t0} yields
\[
\mu(\m F_i \setminus \m F_i[\m S_i]) \leq \sum_{\ell=1}^{\ell_0} \mu(\m B_i[\m V_i^\ell \setminus \m S_i^\ell]) \leq \sum_{\ell=1}^{\ell_0} (p_\ell-p_{\ell+1})\cdot n^{-10} b^{(i)}_{2i-t'} \leq n^{-9} b^{(i)}_{2i-t'}. 
\]
Finally, as by Claim~\ref{b_intersecting}, $b^{(i)}_{2i-t'}$ is the size of an $(n-t)$-intersecting family of permutations and $\m G$ is assumed to be a maximum-size $(n-t)$-intersecting family in $S_n$, we have $n^{-9} b^{(i)}_{2i-t'} \leq n^{-9} |\m G|$. This completes the proof of Lemma~\ref{lem:iterative_peeling_off}.
\end{proof}

\medskip \noindent \textbf{Step~3: Finding a simple sub-structure within $\m F$.} We begin this step with further simplifying the family $\m F$ by showing that certain parts of it have a negligible contribution to $\mu(\m F)$ and thus can be removed without affecting the assertion significantly. The following claim shows that 
the contribution of the $\m F_i$'s for $i< \lceil t'/2\rceil$ to $\mu(\m F)$ is negligible.
\begin{claim}\label{Cl:Aux-Large-t1}
    Let $n,t,t', \m F, \m B_i$ be as defined above. We have
    \[
    \sum_{i=0}^{\lceil t'/2 \rceil-1} \mu(\m F_i) \leq \sum_{i=0}^{\lceil t'/2 \rceil-1} \mu(\m B_i) \leq n^{-\frac{\epsilon}{4}} |\m G|.
    \]
\end{claim}

\begin{proof}
The left inequality is obvious since $\m F_i \subset \m B_i$ for all $i$. To prove the right inequality, note that as was explained in the proof of Claim~\ref{claimdecay}, 
$\mu(\m B_{\lceil t'/2 \rceil})$ is the number of full permutations that have exactly $\lceil t'/2 \rceil$ moving points in $I$. This number is clearly lower bounded by  
\[
{|I| \choose \lceil t'/2 \rceil}\cdot d_{n-|I|+\lceil t'/2 \rceil} \geq {|I| \choose \lceil t'/2 \rceil} \cdot \frac13 \cdot (n-|I|+\lceil t'/2 \rceil)!,
\]
where the inequality uses~\eqref{eq:der-short}. 
Similarly, as $\mu(\cup_{i=0}^{\lceil t'/2 \rceil-1} \m B_i)$ is the number of full permutations that have at most $\lceil t'/2 \rceil-1$ moving points in $I$, we have 
\[
\mu(\cup_{i=0}^{\lceil t'/2 \rceil-1} \m B_i) \leq \binom{|I|}{\lceil t'/2 \rceil-1}(n-|I|+\lceil t'/2 \rceil-1)!.
\]
Hence, 
\begin{align*}
\frac{\mu \left(\cup_{i=0}^{\lceil t'/2 \rceil-1} \m B_i \right)}{\mu(\m B_{\lceil t'/2 \rceil})} &\leq \frac{\binom{|I|}{\lceil t'/2 \rceil-1}(n-|I|+\lceil t'/2 \rceil-1)!}{{|I| \choose \lceil t'/2 \rceil} \cdot \frac13(n-|I|+\lceil t'/2 \rceil)!} \\ &= \frac{3\lceil t'/2 \rceil}{(|I|-\lceil t'/2 \rceil+1)(n-|I|+\lceil t'/2 \rceil)}.
\end{align*}
As $n^{(1+\epsilon)/2} \leq t \leq n-n^{1-\frac{\epsilon}{8}}$
and $t'=|I|-(n-t)$ satisfies $t' \leq k \leq n^{\frac{1}{2}-\frac{\epsilon}{8}}$, it follows that for a sufficiently large $n$,
\begin{equation}\label{Eq:Aux-Large-t3} 
    \mu \left(\cup_{i=0}^{\lceil t'/2 \rceil-1} \m B_i \right) \leq n^{-1-\frac{\epsilon}{4}}\mu(\m B_{\lceil t'/2 \rceil}).
\end{equation}
%\andrey{I think that this estimate could be improved (if needed) - and it may be needed if we seriously want to obtain a stability result.} \nathan{Probably yes. But it seems this is not needed for the stability claim we prove. Maybe this will be needed for Hilton-Milner.}  
If $t'$ is even, then by Claim~\ref{b_intersecting}, $\mu(\m B_{\lceil t'/2 \rceil})=b^{(t'/2)}_0$ is the size of an $(n-t)$-intersecting family of permutations, and hence,~\eqref{Eq:Aux-Large-t3} implies that $\mu \left(\cup_{i=0}^{\lceil t'/2 \rceil-1} \m B_i \right) \leq n^{-1-\frac{\epsilon}{4}} |\m G|$. To handle the case of an odd $t$, note that for any $i$, we clearly have $b^{(i)}_0/ b^{(i)}_{1} \leq n$. Thus,~\eqref{Eq:Aux-Large-t3} implies that 
\[
\mu \left(\cup_{i=0}^{\lceil t'/2 \rceil-1} \m B_i \right) \leq n^{-1-\frac{\epsilon}{4}}b^{(\lceil \frac{t'}{2} \rceil)}_0 \leq n^{-\frac{\epsilon}{4}}b^{(\lceil t'/2 \rceil)}_1 \leq n^{-\frac{\epsilon}{4}}|
\m G|,
\]
where the last inequality holds since by Claim~\ref{b_intersecting}, $b^{(\lceil t'/2 \rceil)}_1$ is the size of an $(n-t)$-intersecting family of permutations. 
%\nathan{Maybe change this argument so it will be similar to the `small t' setting.} 
Hence, for both even and odd values of $t'$, the assertion of the claim follows. 
\end{proof} 

The combination of~\eqref{Eq:Aux-Large-t00},  Lemma~\ref{lem:iterative_peeling_off}(2), and Claim~\ref{Cl:Aux-Large-t1} yields that for a sufficiently large $n$,  
\begin{equation*}\mu \left(\ff\setminus\bigcup_{i=\lceil t'/2 \rceil}^{t'} \ff_i[\m S_i]\right) \leq n^{-\epsilon/5} |\m G| \leq n^{-\epsilon/6} \mu(\m F). 
\end{equation*}
Let $q$ be the smallest integer such that for any $\lceil t'/2 \rceil \leq i \leq t'$, the size of any set in $\m S_i$ is at most $(2i-t')+q$. By Lemma~\ref{lem:iterative_peeling_off}(1), $q \leq \frac{200}{\epsilon}$. Let $p\ge \lceil t'/2\rceil$ be the smallest such that $\m S_{p} \neq \emptyset$. Then all the families $\m S_{p+q+j}$, $j\ge 1$, must be empty, as by Lemma~\ref{lem:iterative_peeling_off}(3), for each $S_1 \in \m S_p$ and $S_2 \in \m S_{p+q+j}$ we have
\[
|S_1 \cap S_2| \geq p+(p+q+j)-t'>(2p-t')+q,
\]
while the maximum size of a set in $\m S_p$ is at most $(2p-t')+q$.
Combining with the above, we obtain
\begin{equation}\label{eqmostf} \mu(\ff\setminus \bigcup_{i=0}^{q}\ff_{p+i}[\m S_{p+i}]) \leq n^{-\epsilon/5}|\m G| \leq n^{-\epsilon/6} \mu(\ff).
\end{equation}
For each $p \leq m \leq p+q$, denote $\m S_m^{(j)}:= \{S \in \m S_{m}:|S|=j\}$. As for each $p \leq m \leq p+q$, $\m S_{m}= \cup_{j=2m-t'}^{2m-t'+q}\m S_m^{(j)}$, by the pigeonhole principle there exist $p \leq m \leq p+q$ and $0 \leq l \leq q$ such that
\begin{equation}\label{eqchoicem2} \mu(\ff_{m}[\m S_m^{(2m-t'+l)}])\ge (q+1)^{-2}(1-n^{-\epsilon/6}) \mu(\ff) \geq \frac{\epsilon^2}{50000}\mu(\m F).
\end{equation}

The following claim asserts that~\eqref{eqchoicem2} must hold for $l=0$. 
\begin{claim}\label{Cl:Aux-Large-t2}
    Let $\m F$ be as above, and let $p \leq m \leq p+q$ and $0 \leq l \leq q$ be such that $\mu(\ff_{m}[\m S_m^{(2m-t'+l)}])\geq \frac{\epsilon^2}{50000}\mu(\m F)$. Then $l=0$.
\end{claim}

\begin{proof}[Proof of Claim~\ref{Cl:Aux-Large-t2}]
Recall that by the proof of Lemma~\ref{lem:iterative_peeling_off}, $\m S_m \subset \m V_m^{\ell_0}$, where $\m V_m^{\ell_0}$ was constructed by applying Lemma~\ref{lemsimplsimple2} to the family $\m S_m^{\ell_0-1}$, with $p_{\ell_0}^m = (2m-t')+6k\cdot 2^{-\ell_0}$ in place of $p$ and $r_{\ell_0} = 60k\cdot 2^{-\ell_0+1} \leq \frac{4000}{\epsilon}$ in place of $r$. By Lemma~\ref{lemsimplsimple2}, this implies that for any $l \geq 0$, 
\begin{align}\label{eqdecaylast}
\begin{split}
\mu(\m B_m(\m S_m^{(2m-t'+l)}))&\le {2m-t'+l\choose{2m-t'}}\cdot \left(\tfrac{4000}{\epsilon} \right)^{l} \cdot b_{2m-t'+l}^{(m)} \\ &\le \left(\tfrac{4000 \cdot 2t'}{\epsilon} \right)^{l} \cdot b_{2m-t'+l}^{(m)}
\le n^{(\frac{1}{2}-\frac{\epsilon}{9})l} b_{2m-t'+l}^{(m)},
\end{split}
\end{align}
where the last inequality holds for a sufficiently large $n$, since $t' \leq k \leq n^{\frac{1}{2}-\frac{\epsilon}{8}}$. 

Since by~\eqref{Eq:Aux-Large-t4}, for any $j \leq i \leq k$ we have $b^{(i)}_j/ b^{(i)}_{j+1} \geq n^{1/2}$,~\eqref{eqdecaylast} implies that 
\begin{align*}
    \begin{split}
        \frac{\sum_{l=1}^{q} \mu(\ff_m(\m S_m^{(2m-t'+l)}))}{b^{(m)}_{2m-t'}}&\le \frac{\sum_{l=1}^{q} \mu(\m B_m(\m S_m^{(2m-t'+l)}))}{b^{(m)}_{2m-t'}} \leq   \frac{\sum_{l=1}^{q}n^{(\frac{1}{2}-\frac{\epsilon}{9})l} \cdot b_{2m-t'+l}^{(m)}}{b^{(m)}_{2m-t'}} \\
        &\le \sum_{l=1}^q \frac{n^{(\frac{1}{2}-\frac{\epsilon}{9})l}}{n^{\frac l2 }} \leq n^{-\epsilon/10},
    \end{split}
\end{align*}
where the last inequality holds for a sufficiently large $n$ since $q \leq \frac{200}{\epsilon}$. Therefore, we have
\[
\sum_{l=1}^{q} \mu(\ff_m(\m S_m^{(2m-t'+l)})) \leq n^{-\epsilon/10} b^{(m)}_{2m-t'} \leq n^{-\epsilon/10}\mu(\m G) \leq n^{-\epsilon/11}\mu(\m F),
\]
where the penultimate inequality holds since $b^{(m)}_{2m-t'}$ is the size of an $(n-t)$-intersecting family of permutations, and the last inequality holds by~\eqref{Eq:Aux-Large-t00}. Since $\frac{\epsilon^2}{50000} > n^{-\epsilon/11}$ for a sufficiently large $n$, this implies that~\eqref{eqchoicem2} must hold for $l=0$, as asserted.
\end{proof}

Claim~\ref{Cl:Aux-Large-t2} implies, in particular, that there exists $\lceil t'/2 \rceil \leq m \leq t'$ such that $\m S_{m}^{(2m-t')} \neq \emptyset$. Let $F \in \m S_{m}^{(2m-t')}$.
As $\m S_m$ is a $(2m-t')$-intersecting family by Lemma~\ref{lem:iterative_peeling_off}(3), all sets in $\m S_m$ contain $F.$

Set $\m H'=\cup_{i=0}^{t'}\m H'_i$, where 
\[
\m H'_i:=\{B\in \m B_i: |B\cap F|\ge i-(t'-m)\}.
\]
Since for any $\lceil t'/2 \rceil \leq i \leq t'$, the family $\m S_i$ cross $(i+m-t')$-intersects the family $\m S_m$ by Lemma~\ref{lem:iterative_peeling_off}(3), we have $\m F_i \subset \m H'_i$ for all $\lceil t'/2 \rceil \leq i \leq t'$. %\nathan{The fact that $\m S_i$ cross $(i+m-t')$ intersects $\m S_m$ does not readily imply that each $S \in \m S_i$ intersects $F$ in $i+m-t'$ elements! Another argument is needed here!}
By~\eqref{eqmostf}, this implies  
\begin{equation}\label{Eq:Aux-Large-t5}
    \mu(\m F \setminus (\m F \cap \m H')) \leq 
n^{-\epsilon/5}|\m G| \leq n^{-\epsilon/6} \mu(\ff),
\end{equation}
which means that $\m F$ is `almost' contained in $\m H'$.

\medskip \noindent \textbf{Step~4: Transforming from partial permutations to full permutations.}
The last step of the proof is leveraging the structural result we obtained for $\m F$ to the original setting of full permutations and their representation by sets.

Let $\m F' \subset \mb [n] \cup \{(i,j):i,j \in [n], i \neq j\}$ be the family of sets that corresponds to $\m G$ in the representation of full permutations by sets presented in Section~\ref{sec:PermutationstoSets}. To complete the proof of Theorem~\ref{thmfm3}, we have to show that there exist $t/2 \leq m' \leq t$, an extendable set $F'$ of size $2m'-t$, and a family $\m H=\cup_{i=0}^{t}\m H_i$,  
where $\m H_i:=\{A\in \m A_i: |A\cap F'|\ge i-(t-m')\}$,
such that $|\m F' \cap \m H| \geq (1-n^{-\epsilon/6})|\m F'|$.

\smallskip

Consider $F \subset \mb X'$ constructed in the previous step. Denote $F = \{D,(M,\pi(M))\}$, where $\pi \in S_n$. We may assume w.l.o.g.~that $\pi$ has no fixed points in $[n] \setminus I$. Otherwise, we can take a permutation $\pi'$ such that $\pi'|_I=Id_I$ and $\pi' \pi$ has no fixed points in $[n] \setminus I$ (which is possible since $n-|I|>1$), replace the original family $\m G$ by $\pi'\m G$, and repeat the above process with all permutations multiplied by $\pi'$ and all representations of (partial) permutations by sets modified accordingly. Set
\[
F'=\{D',(M, \pi(M))\}, \qquad \mbox{where} \qquad D' = D\cup ([n]\setminus I).
\]
\emph{Extendability of $F'$, its size and structure.} The set $F'$ is clearly extendable (i.e., represents part of the information on the set of moving points of some $\sigma \in S_n$ and the places they move to). Indeed, the only possible obstacle for the extendability of $F'$ is that $(M,\pi(M))$ contains two different pairs of the form $(x,y),(x',y)$ or $(x,y),(x,y')$. In such a case, $F$ would contain these pairs as well, and then we would have $\mu(\m S_m[F])=0$ (as there would be no way to extend $F$ into a full permutation), contradicting~\eqref{eqchoicem2}. 

Put $m':=m+n-|I|$ (so $m'-m=t-t'$). As $\frac{t'}{2} \leq m \leq t'$, we have $\frac{t}{2} \leq t-\frac{t'}{2} \leq m' \leq t$. Furthermore,  
\[
|F'| = |F|+n-|I| = (2m-t')+(n-|I|)=2m'-t,
\]
and thus, $F'$ is extendable and of size $2m'-t$ for some $\frac{t}{2} \leq m' \leq t$. Furthermore, we can assume w.l.o.g.~that   
$M\subset D'$ and $\pi(M) \subset D'$. The easy argument showing this is identical to the argumentation of the corresponding statement in Step~4 of the proof of Theorem~\ref{thmfm2}, and thus we omit it here. Hence, $F'$ has the asserted structure. Define 
\[
\m H=\cup_{i=0}^{t}\m H_i, \qquad \mbox{where} \qquad \m H_i = \{A\in \m A_i: |A\cap F'|\ge m'+i-t\}.
\]
\emph{$\m H$ corresponds to an $(n-t)$-intersecting family of permutations.} For any $i_1,i_2$, any $S_1 \in \m H_{i_1}$ and $S_2 \in \m H_{i_2}$ intersect on at least $(m'+i_1-t)+(m'+i_2-t)-(2m'-t)=i_1+i_2-t$ elements of $F'$, and hence, $\m H_{i_1}$ and $\m H_{i_1}$ are cross $(i_1+i_2-t)$-intersecting. By Claim~\ref{clatranslation}, this implies that $\m H$ corresponds to an $(n-t)$-intersecting family of permutations.  

\medskip \noindent \emph{$\m F'$ is `almost' contained in $\m H$.} In order to prove this, we show that for any $0 \leq i \leq t'$, any full permutation that extends a partial permutation from $\m H'_i$, belongs to $\m H$. Let $\sigma \in S_n$ be a full permutation and denote $F_\sigma = \{D_\sigma, (M_\sigma, \sigma(M_{\sigma}))\}$. Assume that $F_\sigma$ extends a set $Z\in \m H'_i$ and that $|D_\sigma| = i+j$. We have $F_\sigma\in \aaa_{i+j}$ and $|D_\sigma\cap ([n]\setminus I)|=j$. Hence, 
\[
|F_\sigma \cap F'| = |F_\sigma \cap F|+|D_\sigma\cap ([n]\setminus I)| = (m+i-t')+j = m' +(i+j) -t.
\]
Thus, $F_\sigma\in \m H'_{i+j}$, and in particular, $F_{\sigma} \in \m H$. By Lemma~\ref{lemregime3step1}, all but $n^{-9}|\m G|$ of the permutations in $\m G$ extend some partial permutation in $\m Q'$ which corresponds to a set in $\m F$, and on the other hand, by~\eqref{Eq:Aux-Large-t5}, the number of full permutations extending some set in $\m F \setminus (\m F \cap \m H')$ is at most $n^{-\epsilon/5}|\m G|$. Therefore, at least $(1-n^{-\epsilon/6})|\m G|$ of the permutations in $\m G$ extend some partial permutation in $\m H'$, and thus, belong to $\m H$. This completes the proof of Theorem~\ref{thmfm3}.
\end{proof}

\section{Completing the Proof of Theorem~\ref{Thm:Main}}
\label{sec:Completing_the_Proof}

Let $\m G \subset S_n$ be a maximum-size $(n-t)$-intersecting family of permutations and let $\m F \subset [n] \cup \{(i,j):i,j \in [n], i \neq j\}$ be the family of sets that corresponds to $\m G$ in the representation of permutations by sets presented in Section~\ref{sec:PermutationstoSets}. In the previous sections, we showed the existence of a simple sub-structure within $\m F$. 
Specifically, we showed that there exists $t/2 \leq m\le t$ (denoted by $m'$ in Theorems~\ref{thmfm2} and~\ref{thmfm3}) and an extendable set $F = \{D,(M,\sigma_0(M))\}$ of size $2m-t$ with $M\subset D$ and $\sigma_0(M) \subset D$, such that for each $t/2 \leq i \leq t$, a large portion of the sets in $\ff_i$ is contained in the family $\m H_i:=\{A\in \m A_{i}: |A\cap F|\ge i-(t-m)\}$ of the sets in $\m A_{i}$ that intersect $F$ in at least $m+i-t$ elements. In this section, we use this structural information to deduce that $\m G$ is included in a double translate of a family of the form
\[
\m F_{n,n-t,(t-r)/2}=\{\sigma \in S_n: |\mathrm{Moving}(\sigma) \cap \{1,2,\ldots,n-r\}| \leq \tfrac{t-r}{2}\},
\]
for some $r \geq 0$. In Section~\ref{sec:sub:Finalizing-approximate} we show that the structural information allows deducing easily that for an appropriate $\sigma_1 \in S_n$, a large portion of the family $\sigma_1 \m G$ is  included in a double translate of some $\m F_{n,n-t,(t-r)/2}$. In Section~\ref{sec:sub:Finalizing-exact} we show how to leverage the `approximate-containment' statement into an exact containment statement, with different proofs for `small' values of $t$ and `medium and large' values of $t$. In Section~\ref{sec:sub:Finalizing-stability} we show that a straightforward modification of the entire proof process allows obtaining a stability version, thus completing the proof of Theorem~\ref{Thm:Main}.

\medskip \noindent \textbf{Notation.} Throughout this section, we use the following notations. For $n \in \mathbb{N}$, $t \leq n$ and $S \subset [n]$, let
\[
\m E_S:=\{\sigma \in S_n: |\mathrm{Moving}(\sigma) \setminus S| \leq \tfrac{t-|S|}{2}\}, \; \m E'_S:=\{\sigma \in S_n: |\mathrm{Moving}(\sigma) \setminus S| = \tfrac{t-|S|}{2}\}.
\]
Note that for each $S$, the $(n-t)$-intersecting family $\m E_S$ is a double translate of the family $\m F_{n,n-t,(t-|S|)/2}$, and that $\m E'_S \subset \m E_S$.

\subsection{Correction of the set $F$ and the family $\m G$}
\label{sec:sub:Finalizing-approximate}

In the discussion following the statement of Theorem~\ref{thmfm2}, we noted that if $M=\emptyset$, then the family $\m H=\cup_{i=0}^t \m H_i$ containing a large portion of $\m F$, corresponds to the family of permutations $\m E_F$, which is a double translate of the family $\m F_{n,n-t,t-m}$.
We now show that a similar conclusion can be derived without the assumption $M = \emptyset$, for a `corrected' variant of $\m G$, for which $F'=(D \setminus M,\emptyset)$ plays the role of $F$. Essentially, the reason is that once the elements to which the elements in $M$ move are specified, there is no loss of generality in assuming that the elements of $M$ are actually fixed points, and thus, removing them from the set $D$ of moving points. Formally, the `correction' is performed by multiplying $\m G$ with a permutation $\sigma_1$, which we call an \emph{$F$-correction}.
\begin{definition}\label{defmapset}
    Let $F \subset \mb X$ be an extendable set of the form $F = \{D,(M,\sigma_0(M))\}$, where $M \subset D$ and  $\sigma_0(M)\subset D$. Let $\hat{M}:=M\cup \sigma_0(M)$. An $F$-correction is $\sigma_1 \in S_n$ that satisfies the following properties: 
    \begin{itemize}
    \item[(i)] For every $x \in  M$, $\sigma_1(\sigma_0(x))=x$;
    \item[(ii)] $\sigma_1(\hat{M})=\hat{M}$;
    \item[(iii)] $\sigma_1|_{[n] \setminus \hat{M}} = Id_{[n] \setminus \hat{M}}$.
    \end{itemize}
\end{definition}    
Clearly, there may be many possible $F$-corrections. We choose one of them arbitrarily. Once an $F$-correction $\sigma_1$ is picked, we call $\sigma_1 \m G=\{\sigma_1 \sigma: \sigma \in \m G\}$ `the corresponding $F$-correction of $\m G$'. 
\begin{lem}\label{lemcorrcont}
  Let $F = \{D,E\}= \{D,(M,\sigma_0(M))\}\subset \mb X$ 
  be an extendable set of size $2m-t$ with $M \subset D$ and $\sigma_0(M)\subset D$. Let $\m H = \cup_{i=0}^t \m H_i$, where $\m H_i = \{A\in \m A_i: |A\cap F|\ge m+i-t\}$. Let $\m G'_i$ be the family of permutations to which $\m H_i$ corresponds, and set $\m G':=\cup_{i=1}^t \m G'_i$.
  Let $\sigma_1$ be an $F$-correction. Then 
  \begin{equation}\label{Eq:Aux-Wrappingup1}
  \sigma_1 \m G'\subset \m E_{D \setminus M}, \qquad \mbox{and} \qquad \sigma_1 \m G'_m \subset \m E'_{D \setminus M}.
  \end{equation}
\end{lem}
Recall that $\m E_{D \setminus M}$ is a double translate of the family $\m F_{n,n-t,(t-r)/2}$, for $r=|D \setminus M|$. Hence, Lemma~\ref{lemcorrcont} implies that a large part of $\m G$ is included in a double translate of 
$\m F_{n,n-t,(t-r)/2}$.

\begin{proof}[Proof of Lemma~\ref{lemcorrcont}]
We begin with proving the left inclusion in~\eqref{Eq:Aux-Wrappingup1}. Denote $r:=|D \setminus M|$. As
$2m-t=|F| = |D\setminus M|+2|M|$, we have $|M| = \frac{2m-t-r}2$. Let $0 \leq i \leq t$, and let $Q=\{D_Q,E_Q\} \subset \mb X$ be an element of $\m H_i$ that corresponds to the permutation $\tau \in S_n$. We show that $\sigma_1 \tau\in \m E_{D \setminus M}$. Showing this for each $i$ and each $Q \in \m H_i$ will imply that $\sigma_1 \m G' \subset \m E_{D \setminus M}$.

Denote $j:=|E\cap E_Q|$. As $|Q \cap F| \geq m+i-t$, we have $|D\cap D_Q|\ge m+i-t-j$, and consequently, $|D_Q \setminus D|\le t+j-m$.

Let us analyze the set  $\mathrm{Moving}(\sigma_1 \tau)$. As $\sigma_1$ fixes all elements outside $D$ and sends elements of $D$ to elements of $D$, we have 
\begin{equation}\label{Eq:Aux-Wrappingup2}
    |\mathrm{Moving}(\sigma_1\tau)\setminus D| = |\mathrm{Moving}(\tau)\setminus D|
    \le t+j-m.    
\end{equation}
Next, for any $v \in M$ such that $(v,\tau(v)) \in E \cap E_Q$, we have $\sigma_1\tau(v) = v$, and thus, $v \not \in \mathrm{Moving}(\sigma_1 \tau)$. Thus,
\begin{equation}\label{Eq:Aux-Wrappingup3}
|\mathrm{Moving}(\sigma_1\tau)\cap M|\le |M|-j = \tfrac{2m-t-r}2-j.
\end{equation}
Summing up Inequalities~\eqref{Eq:Aux-Wrappingup2} and~\eqref{Eq:Aux-Wrappingup3}, we get
\[
|\mathrm{Moving}(\sigma_1\tau)\setminus (D\setminus M)| \le t+j-m+\tfrac{2m-t-r}2-j= \tfrac{t-r}2,
\]
and thus, $\sigma_1 \tau \in \m E_{D \setminus M}$, which proves the left inclusion of~\eqref{Eq:Aux-Wrappingup1}. To prove the right inclusion of~\eqref{Eq:Aux-Wrappingup1}, note that all inequalities in the above proof hold as equalities for $Q \in \m H_m$. %\andrey{I'm confused - Why does \eqref{Eq:Aux-Wrappingup3} necessarily hold with equality?}\nathan{As $Q \in \m H_m$, we have $|Q \cap F| \geq 2m-t$. As $|F|=2m-t$, this means that $F \subset Q$. In particular, $E \subset E_Q$, and hence, $j=|E \cap E_Q|=|E|=|M|$. Therefore, $|M|-j=0$, and thus, the inequality $|\mathrm{Moving}(\sigma_1 \tau) \cap M| \leq |M|-j=0$, must be an equality.}\andrey{Yes, right! I guess I'm too sleepy =) Thank you!}\nathan{:)}
This completes the proof of Lemma~\ref{lemcorrcont}.
\end{proof}

\subsection{Proof of the maximality statement of Theorem~\ref{Thm:Main}}
\label{sec:sub:Finalizing-exact}

In this subsection we show how to leverage the `approximate-containment' statement of Lemma~\ref{lemcorrcont} into an exact containment statement, which proves that for any $n \geq n_0$, and any $t$, the maximum size of an $(n-t)$-intersecting family is always obtained by one of the families $\m F_{n,n-t,r}$. 
%We formulate some of our results for `almost' maximum-sized $(n-t)$-intersecting families, as this will be useful in the proof of the stability statement of Theorem~\ref{Thm:Main} below.

\subsubsection{The `small $t$' range.}

Let $\epsilon \leq 0.01$, $n\geq n_0$ and $3 \leq t \leq n^{\epsilon}$, and let $\m G$ be a maximum-size $(n-t)$-intersecting family of permutations. Let $\m F =\cup_{i=0}^t \m F_i \subset \m P(\mb X)$ be the family of sets that corresponds to $\m G$. Theorem~\ref{thmfm} asserts that there exists $t/2 \leq m \leq t$ and a set $F = \{D,(M,\sigma_0(M))\}$, where $|F|=2m-t$, $\sigma_0 \in S_n$, $M\subset D$, $\sigma_0(M) \subset D$, $|D|=m-\lfloor t/2\rfloor$, and $|M|=m-\lceil t/2\rceil$, such that $|\m F_m| \geq n^{-2\epsilon} |\m G|$ and $|\m F_m[F]| \geq (1-n^{-\epsilon})|\m F_m|$. 

Let $\m H_m = \{A\in \m A_m: |A\cap F|\ge 2m-t\}$, and let $\m G'_m$ be the family of permutations to which $\m H_m$ corresponds. Let $\sigma_1$ be an $F$-correction. Lemma~\ref{lemcorrcont} asserts that 
$\sigma_1 \m G'_m \subset \m E'_{D \setminus M}$. As $\m F_m[F] \subset \m H_m$, this implies
\[
|\sigma_1 \m G \cap \m E'_{D \setminus M}| \geq |\m F_m[F]| \geq (1-n^{-\epsilon})n^{-2\epsilon}|\m G| \geq (1-n^{-\epsilon})n^{-2\epsilon} |\m E'_{D \setminus M}|,
\]
where the last inequality holds since $\m E'_{D \setminus M} \subset S_n$ is $(n-t)$-intersecting and $\m G$ is maximum-size $(n-t)$-intersecting. Note that for an odd $t$, $|D \setminus M|=1$, and for an even $t$, $D \setminus M =\emptyset$. Therefore, the following lemma (applied to the family $\sigma_1 \m G$) implies the maximality assertion of Theorem~\ref{Thm:Main} in this range of values of $t$. 
\begin{lem}\label{lemdomstar}
    For any $\epsilon \leq 0.01$, there exists $n_0$ such that the following holds for all $n \geq n_0$. Let $3 \leq t \leq n^{\epsilon}$, and let $\m G \subset S_n$ be a maximum-size $(n-t)$-intersecting family.
    %such that $|\m G| \geq \frac{1}{20}\max_{r'} |\m F_{n,n-t,r'}|$. 
    Suppose that 
    \begin{equation}\label{Eq:Aux-Wrappingup4}
    |\m G \cap \m E'_{S}| \geq n^{-3\epsilon} |\m E'_{S}|,
    \end{equation}
    where $|S|=1$ for an odd $t$ and $S=\emptyset$ for an even $t$. Then $\m G \subset \m E_{S'}$, for some $|S'|=|S|$.
\end{lem}

\begin{proof}
    We present the proof for an odd $t$. The `even $t$' case can be proved by a much simpler version of the same argument. 
    
    Let $\m G$ be a family that satisfies the assumptions of the lemma for an odd $t$, and write $S=\{x\}$. Denote
    \[
    \m E''_S:=\{\sigma \in \m E'_S: x \in \mathrm{Moving}(\sigma)\}.
    \]
    %Note that if $\sigma \in \m E'_S \setminus \m E''_S$, then $x \not \in \mathrm{Moving}(\sigma)$, and thus, $|\mathrm{Moving}(\sigma)| \leq \frac{t-1}{2}$, which implies that $\sigma \in \m F_{n,n-t,0}:=\{\sigma \in S_n:|\mathrm{Moving}(\sigma)| \leq \lfloor t/2 \rfloor \}$. Hence, $\m E'_S \subset \m E''_S \cup \m \{\sigma \in F_{n,n-t,0}: x \notin \mathrm{Moving}(\sigma)\}$. Therefore, 
    The assumption~\eqref{Eq:Aux-Wrappingup4} clearly implies that
    \begin{equation}\label{Eq:Aux-Wrappingup4.5}
    |\m G \cap \m E''_{S}| \geq \tfrac{1}{2} n^{-3\epsilon} |\m E'_{S}|, \qquad \mbox{or} \qquad
    |\m G \cap (\m E'_S \setminus \m E''_{S})| \geq \tfrac{1}{2} n^{-3\epsilon} |\m E'_{S}|.
    %|\m G \cap \m \{\sigma \in F_{n,n-t,0}:x \notin \mathrm{Moving}(\sigma)\}| \geq \tfrac{1}{2} n^{-3\epsilon} |\m E'_{S}|.
    \end{equation}
    We consider these cases separately.

\medskip \noindent \textbf{Case~1: $|\m G \cap \m E''_{S}| \geq \tfrac{1}{2} n^{-3\epsilon} |\m E'_{S}|$.} We show that in this case, $\m G \subset \m E_S$. 

Assume on the contrary that there exists  $\rho \in \m G \setminus \m E_S$.
As $\rho \notin \m E_S$, we have $|\mathrm{Moving}(\rho) \setminus S| \ge \lfloor t/2 \rfloor + 1$. We will obtain a contradiction by showing that only a tiny fraction of the elements of $\m E''_S$ can $(n-t)$-intersect $\rho$.

For any $\pi \in \m E''_S$ to $(n-t)$-intersect $\rho$, the number of common fixed points plus common moving points must be at least $n-t$. This requires
\[
(n-|\mathrm{Moving}(\rho) \cup \mathrm{Moving}(\pi)|)+|E_{\rho} \cap E_{\pi}| \geq n-t,
\]
where $E_{\rho}=\{(i,\rho(i)):i \in \mathrm{Moving}(\rho)\}$. By inclusion-exclusion, this is equivalent to 
\begin{equation}\label{Eq:Aux-Wrappingup5}
    |\mathrm{Moving}(\rho) \cap \mathrm{Moving}(\pi)| + |E_\rho \cap E_\pi| \ge |\mathrm{Moving}(\rho)| + |\mathrm{Moving}(\pi)| - t.
    \end{equation}
    Note that $|\mathrm{Moving}(\pi)| = \lfloor t/2 \rfloor + 1$, and that 
    \[
    |\mathrm{Moving}(\rho)| = |\mathrm{Moving}(\rho) \setminus \{x\}| + |\mathrm{Moving}(\rho) \cap \{x\}| \ge \lfloor t/2 \rfloor + 1 + |\mathrm{Moving}(\rho) \cap \{x\}|.
    \]
    Hence, a lower bound on the right hand side of~\eqref{Eq:Aux-Wrappingup5} is
    \begin{align*}
     |\mathrm{Moving}(\rho)| + |\mathrm{Moving}(\pi)| - t &\ge \lfloor t/2 \rfloor + 1 + |\mathrm{Moving}(\rho) \cap \{x\}| + \lfloor t/2 \rfloor + 1 - t \\ 
     &= |\mathrm{Moving}(\rho) \cap \{x\}| +1.
    \end{align*}
    Substituting 
    \[
    |\mathrm{Moving}(\rho) \cap \mathrm{Moving}(\pi)| = |(\mathrm{Moving}(\rho) \setminus \{x\}) \cap (\mathrm{Moving}(\pi))| + |\mathrm{Moving}(\rho) \cap \{x\}|,
    \]
    in the left hand side of~\eqref{Eq:Aux-Wrappingup5}, the term $|\mathrm{Moving}(\rho) \cap \{x\}|$ cancels out, and we obtain that the following requirement is strictly weaker than~\eqref{Eq:Aux-Wrappingup5}:
    \begin{equation}\label{Eq:Aux-Wrappingup6}
        |(\mathrm{Moving}(\rho) \setminus \{x\}) \cap (\mathrm{Moving}(\pi))| + |E_\rho \cap E_\pi| \ge 1.
    \end{equation}
    
    Suppose $(\mathrm{Moving}(\rho) \setminus \{x\}) \cap (\mathrm{Moving}(\pi)) = \emptyset$. Then to satisfy~\eqref{Eq:Aux-Wrappingup6}, $\rho$ and $\pi$ must share a moving point. Since their sets of moving points only overlap at $\{x\}$, this requires $\pi(x) = \rho(x) = y$ for some $y \neq x$. However, any permutation maps its set of moving points to itself, so $y =\pi(x) \in \mathrm{Moving}(\pi) \setminus \{x\}$, and similarly, $y =\rho(x) \in \mathrm{Moving}(\rho) \setminus \{x\}$. This forces $y \in (\mathrm{Moving}(\rho) \setminus \{x\}) \cap (\mathrm{Moving}(\pi))$, a contradiction. 

    Therefore, any $\pi \in \m E''_S$ that $(n-t)$-intersects $\rho$ must satisfy 
    \begin{equation}\label{Eq:Aux-Wrappingup7}
    (\mathrm{Moving}(\rho) \setminus \{x\}) \cap (\mathrm{Moving}(\pi)) \neq \emptyset.
    \end{equation}
    We now show that only a tiny fraction of the elements in $\m E_S''$ can $(n-t)$-intersect $\rho$. As $\rho$ must $(n-t)$-intersect any $\pi \in \m G \cap \m E_{S}''$ and as $\m G \cap \m E_{S}'' \neq \emptyset$, we have $|\mathrm{Moving}(\rho)|\leq t+\lfloor t/2 \rfloor+1$.
    Consider a uniformly chosen $\pi \in \m E''_S$. As the $\lfloor t/2 \rfloor$ moving points of $\pi$ outside $\{x\}$ are distributed uniformly in $[n] \setminus \{x\}$, and as $|\mathrm{Moving}(\rho)| \le t+\lfloor t/2 \rfloor +1$, the fraction of $\pi \in \m E_S''$ for which $\rho,\pi$ satisfy~\eqref{Eq:Aux-Wrappingup7} is at most
    \[
    \frac{\lfloor t/2 \rfloor \cdot |\mathrm{Moving}(\rho)|}{n-1} \le \frac{t^2}{n} \leq n^{2\epsilon-1},
    \]
    where the last two inequalities use the assumption $3 \leq t \leq n^{\epsilon}$. This implies that $\rho$ can $(n-t)$-intersect at most an $n^{2\epsilon-1}$ fraction of the elements of $\m E''_S$, which (for a sufficiently large $n$) contradicts the assumption $|\m G \cap \m E''_{S}| \geq \tfrac{1}{2} n^{-3\epsilon} |\m E'_{S}| \geq \tfrac{1}{2} n^{-3\epsilon} |\m E''_{S}|$. This completes the proof in this case.
    
\medskip \noindent \textbf{Case~2: $|\m G \cap (\m E'_S \setminus \m E''_{S})| \geq \tfrac{1}{2} n^{-3\epsilon} |\m E'_{S}|$.}  
In this case, the assumption means that $\m G$ contains a somewhat-large part of the family 
\[
\m E'_S \setminus \m E''_{S}= \{\sigma \in \m E'_S: x \notin \mathrm{Moving}(\sigma)\},
\]
each of whose elements have exactly $\lfloor t/2 \rfloor$ moving points. We show that in this case, $\m G \subset \m E_{S'}$, for some $S' \subset [n]$ such that $|S'|=|S|=1$.

The same argument as in Case~1 allows showing that if $\rho \in \m G$, then $|\mathrm{Moving}(\rho)| \leq \lfloor t/2 \rfloor +1$, as otherwise, $\rho$ can $(n-t)$-intersect only a tiny fraction of the elements of $\m E'_S \setminus \m E''_{S}$, contradicting the assumption $|\m G \cap (\m E'_S \setminus \m E''_{S})| \geq \tfrac{1}{2} n^{-3\epsilon} |\m E'_{S}| \geq \tfrac{1}{2} n^{-3\epsilon} |\m E'_{S} \setminus \m E_S''|$. 

This implies that $\m G \subset \{\sigma \in S_n: |\mathrm{Moving}(\sigma)| \leq \lfloor t/2 \rfloor+1\}$. At this stage, $(n-t)$-intersection with elements of $\m E'_S \setminus \m E''_S$ is no longer helpful to us, since any permutation with at most $\lfloor t/2 \rfloor+1$ moving points $(n-t)$-intersects each element of $\m E'_S \setminus \m E''_S$. Instead, we show that there exists $|S'|=1$ such that $\m G \subset \m E_{S'}$, using a Hilton-Milner type argument like the argument we used in Step~2 of the proof of Theorem~\ref{thmfm}. 

As was shown in~\eqref{Eq:Aux-Small-t1.11}, by the maximality of $\m G$, we have
\[
|\m G| \geq 
a_1^{\lceil t/2 \rceil}+\sum_{i=0}^{\lceil t/2 \rceil-1} a_0^{(i)} = a_1^{\lfloor t/2 \rfloor+1}+\sum_{i=0}^{\lfloor t/2 \rfloor} a_0^{(i)}.
\]
Denote
\[
\m G'=\{\sigma \in \m G: |\mathrm{Moving}(\sigma)|=\lfloor t/2 \rfloor+1\}.
\]
As the total number of permutations with at most $\lfloor t/2 \rfloor$ moving points is $\sum_{i=0}^{\lfloor t/2 \rfloor} a_0^{(i)}$, we have
\begin{equation}\label{Eq:Aux-Wrappingup8}
|\m G'| \geq a_1^{\lfloor t/2 \rfloor+1} = \binom{n-1}{\lfloor t/2 \rfloor}d_{\lfloor t/2 \rfloor+1} \geq  \binom{n}{\lfloor t/2 \rfloor} \cdot \frac{1}{4} \cdot (\lfloor t/2 \rfloor+1)!, 
\end{equation}
where the second inequality holds for a sufficiently large $n$
%\ohad{I changed a bit, Im not sure how we want to treat the (1-o(1))} \nathan{Fixed.}
(see~\eqref{Eq:Aux-Small-t9} and the calculation before it). 

Consider the family $\m V'=\{\mathrm{Moving}(\sigma): \sigma \in \m G' \} \subset \binom{[n]}{\lfloor t/2 \rfloor+1}$. Since each set of $\lfloor t/2 \rfloor +1$ moving points corresponds to at most $(\lfloor t/2 \rfloor +1)!$ permutations in $\m G'$, by~\eqref{Eq:Aux-Wrappingup8} we have
\begin{equation}\label{Eq:Aux-Wrappingup9} 
    |\m V'| \geq \frac{|\m G'|}{(\lfloor t/2 \rfloor +1)!} \geq \frac{1}{3} \cdot \binom{n}{\lfloor t/2 \rfloor}.
\end{equation}
On the other hand, for any $\sigma,\sigma' \in \m G'$, we have $|\mathrm{Moving}(\sigma) \cap \mathrm{Moving}(\sigma')| \geq 1$, as otherwise, $\sigma$ and $\sigma'$ disagree on at least $(\lfloor t/2 \rfloor +1) + (\lfloor t/2 \rfloor +1) =t+1$ elements, contradicting the $(n-t)$-intersection property of $\m G$. 

Hence, $\m V'$ is intersecting. 
Since $|\m V'| \geq \frac{1}{20} \cdot \binom{n}{(\lfloor t/2 \rfloor+1)-1}$, by Corollary~\ref{Cor:HM} this implies that assuming $n$ is sufficiently large, there exists $i \in [n]$ such that $i \in \mathrm{Moving}(\sigma)$ for all $\sigma \in \m G'$. Hence, denoting $S':=\{i\}$, we have
\[
\m G \subset \m E_{S'}=\{\sigma \in S_n: |\mathrm{Moving}(\sigma) \setminus \{i\}| \leq \lfloor t/2 \rfloor \}.
\]
This completes the proof of Lemma~\ref{lemdomstar}.
\end{proof}

%\nathan{For stability: The above holds for all $t \geq 3$ assuming $\eta > 1/2$ and for all $t>c(\eta)$ assuming $\eta>0$. The only place where the assumption $t>c$ is needed is~\eqref{Eq:Aux-Wrappingup8} in which in the stability version one has to use~\eqref{Eq:Aux-Small-t1.2}.}

\subsubsection{The `medium-large $t$' range.}

Let $\epsilon \leq 0.01$, $n\geq n_0$ and $n^{\epsilon} \leq t \leq n-n^{1-(\epsilon/8)}$, and let $\m G$ be a maximum-size $(n-t)$-intersecting family of permutations. Let $\m F =\cup_{i=0}^t \m F_i \subset \m P(\mb X)$ be the family of sets that corresponds to $\m G$. Theorems~\ref{thmfm2} and~\ref{thmfm3} assert (for $n^{\epsilon} \leq t \leq n^{(1+\epsilon)/2}$ and for $n^{(1+\epsilon)/2} \leq t \leq n-n^{1-(\epsilon/8)}$, respectively) that there exist $t/2 \leq m \leq t$, an extendable set $F$ of size $2m-t$ of the form $F = \{D,(M,\sigma_0(M))\}$, where $\sigma_0 \in S_n$, $M\subset D$ and $\sigma_0(M) \subset D$, and a family $\m H=\cup_{i=0}^{t}\m H_i$,  where 
\[
\m H_i:=\{A\in \m A_i: |A\cap F| \ge i-(t-m)\},
\]
such that $|\m F \cap \m H| \geq (1-n^{-\epsilon/6})|\m F|$.

Let $\m G'_i$ be the family of permutations to which $\m H_i$ corresponds, and let $\m G':=\cup_{i=0}^t \m G'_i$. Let $\sigma_1$ be an $F$-correction. Lemma~\ref{lemcorrcont} asserts that 
$\sigma_1 \m G' \subset \m E_{D \setminus M}$. As $|\m F \cap \m H| \geq (1-n^{-\epsilon/6})|\m F|$, this implies
\[
|\sigma_1 \m G \cap \m E_{D \setminus M}| \geq (1-n^{-\epsilon/6})|\m F| = (1-n^{-\epsilon/6})|\m G|.
%\geq (1-n^{-\epsilon/6})|\m E_{D \setminus M}|,
\]
%where the last inequality holds since $\m E_{D \setminus M} \subset S_n$ is $(n-t)$-intersecting and $\m G$ is maximum-size $(n-t)$-intersecting. 
We would like to show that $\sigma_1 \m G \subset \m E_{D \setminus M}$, which will imply the maximality assertion of Theorem~\ref{Thm:Main} in this range of values of $t$. 

As a preparation step, we use the maximality of $\m G$ to obtain an estimate on $|D \setminus M|$. Denote $r:=|D \setminus M|$ and $\ell:=\frac{t-r}{2}$. Note that in these notations, $\m E_{D \setminus M}$ is a double translate of the family $\m F_{n,n-t,\ell}$. 

\begin{lem}\label{lemsizeofmaxFrankl}
For any $\epsilon \leq 0.01$, there exists $n_0 \in \mathbb{N}$ such that the following holds for all $n \geq n_0$ and all $n^{\epsilon} \leq t \leq n-n^{1-\frac{\epsilon}{8}}$. Let $\m G \subset S_n$ be a maximum-size $(n-t)$-intersecting family of permutations.
Assume that 
$|\m G \cap \m E_{S}| \geq (1-n^{-\epsilon/6})|\m G|$ for some $S \subset [n]$. Denote $r:=|S|$ and $\ell:=\tfrac{t-r}{2}$. Then 
\begin{equation}\label{Eq:Aux-Wrappingup9.3}
\ell^2 \leq 3(n-t).
\end{equation}
\end{lem}

\begin{proof}[Proof of Lemma~\ref{lemsizeofmaxFrankl}]
As $\m G$ is a maximum-sized $(n-t)$-intersecting family in $S_n$, we have $|\m G| \geq |\m F_{n,n-t,0}|=t!$.
%\ohad{is it  $|\m G| \ge |\m F_{n,n-t,0}|=t!$}\nathan{Yes, of course.} 
Hence, 
\[
|\m E_S| \geq |\m G \cap \m E_{S}| \geq (1-n^{-\epsilon/6})|\m G| \geq (1-n^{-\epsilon/6})t!. 
\]
On the other hand, as $\m E_S$ is a double translate of $\m F_{n,n-t,\ell}$, we have
\begin{equation}\label{Eq:Aux-Wrappingup9.5}
|\m E_S| \leq \binom{n-t+2\ell}{\ell} (t-\ell)! \leq \frac{(n-t+2\ell)^\ell}{\ell!} (t-\ell)!
\end{equation}
Combining these bounds, we get
\[
(1-n^{-\epsilon/6})t! \leq |\m E_S| \leq \frac{(n-t+2\ell)^\ell}{\ell!} (t-\ell)!.
\]
Rearranging, using the bounds $\ell! \ge (\ell/e)^\ell$ and $\frac{t!}{(t-\ell)!} \ge (t-\ell)^\ell$, and taking $\ell$'th root of both sides, we obtain
\[
(1-n^{-\epsilon/6})^{1/\ell}(t-\ell) \le  \frac{e(n-t+2\ell)}{\ell},
\]
and thus, for a sufficiently large $n$, we have
\begin{equation}\label{Eq:Aux-Wrappingup10}
\ell(t-\ell) \le 2.8(n-t+2\ell).
\end{equation}
Since $\ell \leq t/2$, for $t>n^{(1+\epsilon)/2}$ the inequality~\eqref{Eq:Aux-Wrappingup10} implies $\ell \leq 5.6n^{(1-\epsilon)/2}$, and thus, $\ell^2 \leq 36n^{1-\epsilon} \leq n-t$, for a sufficiently large $n$. For $t \leq n^{(1+\epsilon)/2}$,~\eqref{Eq:Aux-Wrappingup10} implies $\ell^2 \leq \ell(t-\ell) \le 2.8(n-t+2\ell) \leq 3(n-t)$ for a sufficiently large $n$. %This completes the proof.
\end{proof}

The following lemma (applied to the family $\sigma_1 \m G$) implies that in the above notations, $\sigma_1 \m G \subset \m E_{D \setminus M}$, which shows the maximality assertion of Theorem~\ref{Thm:Main} in this range of values of $t$. 
\begin{lem}\label{lem:almost-maximal-contained-in-Ek}
    For any $\epsilon \leq 0.01$, there exists $n_0$ such that the following holds for all $n \geq n_0$. Let $n^{\epsilon} \leq t \leq n-n^{1-(\epsilon/8)}$, and let $\m G \subset S_n$ be a maximum-size $(n-t)$-intersecting family.
    %such that $|\m G| \geq \frac{1}{20}\max_{r'} |\m F_{n,n-t,r'}|$. 
    Suppose that 
    \begin{equation}\label{Eq:Aux-Wrappingup11}
    |\m G \cap \m E_{S}| \geq (1-n^{-\epsilon/6}) |\m G|,
    \end{equation}
    for some $S \subset [n]$. Then $\m G \subset \m E_S$. 
\end{lem}

\begin{proof}[Proof of Lemma~\ref{lem:almost-maximal-contained-in-Ek}]
As $\m G$ is a maximum-size $(n-t)$-intersecting family and $\m E_S$ is $(n-t)$-intersecting,~\eqref{Eq:Aux-Wrappingup11} implies
\begin{equation}\label{Eq:Aux-Wrappingup12}
|\m G \cap \m E_{S}|\geq (1-n^{-\epsilon/6})|\m E_{S}|.
\end{equation}
We will use this to show that $\m G \subset \m E_S$, since any $\pi \not \in \m E_S$ can $(n-t)$-intersect at most a constant fraction of the elements of $\m E_S$.

Denote $r:=|S|$ and $\ell:=\frac{t-r}{2}$. Assume w.l.o.g.~that $S=\{n-r+1,\ldots,n\}$, and hence,
\[
\m E_S=\{\sigma \in S_n:|\mathrm{Moving}(\sigma) \cap [n-t+2\ell]| \leq \ell\}.
\]
Assume on the contrary that there exists $\pi \in \m G \setminus \m E_S$. We arrive at a contradiction in a two-step argument.

\medskip \noindent \textbf{Step~1: Bounding from below the number of elements of $\m E_S$ that do not $(n-t)$-intersect $\pi$.} Denote
$\m U_\pi=\{\sigma \in \m E_S: |\sigma \cap \pi| < n-t\}$.
We will prove that for a sufficiently large $n$, for any $\pi \notin \m E_S$,
\begin{equation}\label{Eq:Aux-Wrappingup9.4}
    |\m U_\pi| \geq 0.99 \cdot \binom{n-t+\ell-1}{\ell} \cdot d_{t-\ell},
\end{equation}
where $d_{t-\ell}$ is the number of derangements on a set of $t-\ell$ elements. 

For $\ell=0$, we have $\m E_S=\{\sigma \in S_n: \mathrm{Fixed}(\sigma) \supset [n-t]\}$. Hence, for each $\pi \not \in \m E_S$, each $\sigma \in \m E_S$ such that $\sigma(i) \neq \pi(i)$ for all $i \in S$, intersects $\pi$ in less than $n-t$ elements. The number of such $\sigma$'s is clearly at least $d_t$, and hence,~\eqref{Eq:Aux-Wrappingup9.4} holds in this case. 

Thus, we may assume $\ell \geq 1$. Denote $U=\mathrm{Moving}(\pi) \cap [n-t+2\ell]$. As $\pi \notin \m E_S$, we may write $|U|=\ell+1+s$, for some $s \geq 0$. For each $\sigma \in \m E_S'$, define
\[
T:=\mathrm{Moving}(\sigma) \cap [n-t+2\ell], \quad J:=U \cap T, \quad R:=\{i \in S \cup T:\sigma(i)=\pi(i)\}.
\]
Note that the points $i \in [n] \setminus R$ on which $\sigma$ agrees with $\pi$ are exactly the elements of $([n-t+2\ell] \setminus (T \cup U)$. (These elements are fixed points of both $\sigma$ and $\pi$). Hence, by the inclusion-exclusion principle,
\[
|\pi \cap \sigma|=|R|+(n-t+2\ell)-\ell-(\ell+1+s)+|J|=(n-t)+|R|+|J|-s-1.
\]
Thus, each $\sigma \in \m E_S'$ for which we have $|R|+|J| \leq s$, is included in $\m U_{\pi}$. We consider two cases, according to the size of $s$.

\medskip \noindent \emph{Case~1: $s \leq \sqrt{\frac{n-t}{\ell}}$.}
Consider permutations $\sigma \in \m E_S'$ such that $U \cap T = \emptyset$. Note that if such a permutation $\sigma$ satisfies $\sigma(i) \neq \pi(i)$ for all $i \in S$, then we have $|J|=|R|=0$, and thus, $\sigma \in \m U_\pi$. The number of such permutations for each fixed $T$ is at least $d_{t-\ell}$, and the number of choices of $T$ such that $U \cap T = \emptyset$ is   $\binom{n-t+\ell-1-s}{\ell}$. Hence, in this case we have
\[
|\m U_{\pi}| \geq \binom{n-t+\ell-1-s}{\ell}d_{t-\ell}. 
\]
Note that
\[
\frac{\binom{n-t+\ell-1-s}{\ell}}{\binom{n-t+\ell-1}{\ell}}=\prod_{j=0}^{\ell-1} \frac{n-t+\ell-1-s-j}{n-t+\ell-1-j} \geq \left(1-\frac{s}{n-t}\right)^\ell \geq e^{-2\sqrt \frac{\ell}{n-t}} \geq 0.99,
\]
where the last two inequalities hold for a sufficiently large $n$, since 
$s \leq \sqrt{(n-t)/\ell}$ by assumption and $\ell^2 \leq 3(n-t)$ by~\eqref{Eq:Aux-Wrappingup9.3}. Therefore, in this case we have
\[
|\m U_{\pi}| \geq 0.99 \cdot \binom{n-t+\ell-1}{\ell}d_{t-\ell}, 
\]
as asserted in~\eqref{Eq:Aux-Wrappingup9.4}.

\medskip \noindent \emph{Case~2: $s \geq \sqrt{\frac{n-t}{\ell}}$.} Consider a permutation $\sigma$ chosen uniformly from $\m E'_S$. 
In expectation over such a choice, we have
\[
\mathbb{E}[|J|]=\frac{\ell(\ell+s+1)}{n-t+2\ell}, \qquad \mbox{and} \qquad \mathbb{E}[|R|]\leq 3.
\]
The equality holds since in this case, $J=U \cap T$, where $T$ is an $\ell$-element set uniformly chosen from $[n-t+2\ell]$ and $U \subset [n-t+2\ell]$ is a fixed set of size $\ell+s+1$. The inequality holds, as for any given $T$, the expected size of intersection between $\pi|_{S \cup T}$ and a randomly chosen permutation $\sigma$ on $S \cup T$ is $\leq 1$, and the probability of the event that such a $\sigma$ satisfies $\mathrm{Moving}(\sigma) \supset T$ is at least $1/3$. Therefore, by Markov's inequality we have
\[
\Pr[|J|+|R|\geq s+1] \leq \frac{\frac{\ell(\ell+1+s)}{n-t+2\ell}+3 }{s+1} \leq 0.01,
\]
where the last inequality holds for a sufficiently large $n$, by~\eqref{Eq:Aux-Wrappingup9.3} and the assumption $s \geq \sqrt{(n-t)/\ell}$. Consequently, in this case we have 
\[
|\m U_\pi| \geq 0.99|\m E_{S'}| \geq 0.99 \cdot \binom{n-t+2\ell}{\ell} \cdot d_{t-\ell} \geq 0.99 \cdot \binom{n-t+\ell-1}{\ell} \cdot d_{t-\ell},
\]
as asserted in~\eqref{Eq:Aux-Wrappingup9.4}.

\medskip \noindent \textbf{Step~2: Obtaining a contradiction to~\eqref{Eq:Aux-Wrappingup12}.} We present here a simple way to reach a contradiction, which is sufficient for proving Lemma~\ref{lem:almost-maximal-contained-in-Ek}. A more elaborate way that yields optimal constants is presented in Section~\ref{sec:sub:Finalizing-stability} below.

%For each $X \in \binom{[n-t+2k]}{n-t+k}$ that contains $S$, we lower bound the size of the set 
%\[
%\m I_X:=\{\sigma \in \m E_S: (\mathrm{Fixed}(\sigma)=X) \wedge (|\sigma \cap \pi|<n-t)\},
%\]
%i.e., the set of permutations $\sigma \in \m E_S$ such that $\mathrm{Fixed}(\sigma)=X$ and $\sigma$ cannot be an element of $\m G$ since $|\sigma \cap \pi|<n-t$.  

%Note that if $\mathrm{Fixed}(\sigma)=X$ and $\sigma(i) \neq \pi(i)$ for all $i \in [n] \setminus X$, then 
%\[
%|\sigma \cap \pi| = |(\sigma \cap \pi) \cap X| \leq |X|-|U|=(n-t+\ell)-(\ell+1)=n-t-1,
%\]
%since $\sigma$ and $\pi$ disagree on all elements of $U$. Thus, $\sigma \in \m I_X$. For each $X$, the number of such $\sigma$'s is the number of permutations on $n-|X|=t-\ell$ elements that completely disagree with the identity permutation and with a given permutation $\pi$. By \nathan{Add here a reference; currently I know proofs using LLL and using permanents, the better one gives a lower bound of $c=e^{-5}$}, their number is at least $c(t-\ell)!$.

%\nathan{Write down what happens when $\ell=0$.}

%Note that the families $\m I_X$ for different sets $X$ are pairwise disjoint. Hence, the number of $\sigma \in \m E_S$ that intersect $\pi$ in less than $n-t$ elements is at least
%\[
%c(t-\ell)!\binom{n-t+2\ell-(\ell+1)}{n-t+\ell-(\ell+1)}=c(t-\ell)!\binom{n-t+\ell-1}{\ell}.
%\]
As by~\eqref{Eq:Aux-Wrappingup9.5}, $|\m E_S| \leq \binom{n-t+2\ell}{\ell} (t-\ell)!$, the inequality~\eqref{Eq:Aux-Wrappingup9.4} implies that for any $\pi \not \in \m E_S$,
\begin{align}\label{Eq:Aux-Wrappingup12.3}
\begin{split}
\frac{|\m U_\pi|}{|\m E_S|} &\geq \frac{0.99 \cdot \binom{n-t+\ell-1}{\ell} \cdot d_{t-\ell}}{\binom{n-t+2\ell}{\ell} (t-\ell)!} \geq \frac{1}{3}\prod_{j=0}^{\ell-1} \frac{n-t+\ell-1-j}{n-t+2\ell-j} \\
&\ge \frac{1}{3}\left( 1 - \frac{\ell+1}{n-t+\ell} \right)^\ell \ge \frac{1}{3}\exp\left(-\frac{2\ell(\ell+1)}{n-t}\right) \geq \tfrac{1}{3} \cdot e^{-12},
\end{split}
\end{align}
where the second inequality holds for all $\ell \geq 1$ since $d_{t-\ell} \geq \frac{1}{3}(t-\ell)!$, and the last two inequalities hold for a sufficiently large $n$ since by~\eqref{Eq:Aux-Wrappingup9.3}, we have $\ell^2 \leq 3(n-t)$. For $\ell=0$, we clearly have $\frac{|\m U_\pi|}{|\m E_S|} \geq \frac{1}{3} \cdot 0.99$ for any $\pi \notin \m E_S$.

As $\m U_{\pi} \subset \m E_S \setminus \m G$, this implies that $|\m E_S \cap \m G| \leq (1-\frac{1}{3} \cdot e^{-12})|\m E_S|$, which contradicts the assumption~\eqref{Eq:Aux-Wrappingup12} for a sufficiently large $n$. This completes the proof. 
%of Lemma~\ref{lem:almost-maximal-contained-in-Ek}.
\end{proof}

\subsection{Proof of the stability statement of Theorem~\ref{Thm:Main}}
\label{sec:sub:Finalizing-stability}

In this subsection, for the sake of simplicity we use the notation $1+o(1)$ to denote a quantity that tends to $1$ as $n \to \infty$, for all relevant values of the other parameters (e.g., $t$).
Let us reformulate the stability statement of the theorem. 
\begin{theorem*}
    For any $\eta>0$, there exists $n_0 \in \mathbb{N}$ such that for all $n >n_0$ and all $0 \leq t \leq n-1$, the following holds. Let $M_{n,n-t}:=\max_r |\m F_{n,n-t,r}|$, and let $\m G \subset S_n$ be an $(n-t)$-intersecting family such that $|\m G|\geq (1-\frac{1}{e}+\eta)M_{n,n-t}$. Then $\m G \subset \tau_1 \m F_{n,n-t,r} \tau_2$, for some $0 \leq r \leq \lfloor t/2 \rfloor$ and $\tau_1,\tau_2 \in S_n$.  
\end{theorem*}

\begin{proof}
    For $t=0,1,2,$ the assertion is trivial. Indeed, the only non-empty $(n-1)$-intersecting families in $S_n$ are single-element families, and the only $(n-2)$-intersecting families that are not $(n-1)$-intersecting consist of two permutations that differ by a single transposition.
    For $t>n-n^{1-(\epsilon/8)}$, the assertion follows from~\cite[Theorem~1, Remark~10]{keller2024t}. For $3 \leq t \leq n-n^{1-(\epsilon/8)}$, inspection of the proof of the maximality statement of Theorem~\ref{Thm:Main} shows that almost all parts of the proof can be translated almost verbatim to the `stability' setting. The only required change is replacing the assumption that $\m G$ is a maximum-size $(n-t)$-intersecting family with the assumption $|\m G|\geq (1-\frac{1}{e}+\eta)M_{n,n-t}$, and making sure that the multiplicative constant $(1-\frac{1}{e}+\eta)$ does not harm the argument. Specifically:
    \begin{itemize}
        \item The proof of Theorem~\ref{thmfm} holds for any $t \geq 3$ under the weaker assumption $|\m G|\geq (\frac{1}{2}+\eta)M_{n,n-t}$ for any $\eta>0$, provided $n>n_0(\eta)$. In fact, it holds even under the assumption $|\m G|\geq \eta M_{n,n-t}$ for any $\eta>0$, provided $n>n_0(\eta)$ and $t>c_0(\eta)$. The only place that requires care is~\eqref{Eq:Aux-Small-t1.2}, where the constant $0.51$ in the right hand side can be replaced by $\frac{1}{2}+\frac{\eta}{2}$, provided $n \geq n_0(\eta)$, and by $\frac{\eta}{2}$, provided $t \geq c_0(\eta)$. 
        %Note that the case $t \leq c_0$ was covered already by Deza and Frankl~\cite{DF77}, though without a stability version. 

        \item Theorems~\ref{thmfm2} and~\ref{thmfm3} hold under the weaker assumption $|\m G| \geq \eta M_{n,n-t}$ for any $\eta>0$, provided $n \geq n_0(\eta)$. 

        \item Lemma~\ref{lemdomstar}, like Theorem~\ref{thmfm}, holds under the assumption $|\m G|\geq (\frac{1}{2}+\eta)M_{n,n-t}$ for any $t \geq 3$, provided $n \geq n_0(\eta)$, and under the assumption $|\m G|\geq \eta M_{n,n-t}$, provided $n>n_0(\eta)$ and $t>c_0(\eta)$. The only place that requires care
        is~\eqref{Eq:Aux-Wrappingup8}, in which one has to use~\eqref{Eq:Aux-Small-t1.2} for the stability claim to work.
    \end{itemize}
The only place where a significant change is required is Lemma~\ref{lem:almost-maximal-contained-in-Ek}. In Step~1 of this lemma, we prove that for each $\pi \notin \m E_S$, the size of the family $\m U_{\pi}=\{\sigma \in \m E_S:|\sigma \cap \pi| <n-t\}$ satisfies 
\begin{equation}\label{Eq:Aux-Wrappingup12.5}
|\m U_\pi| \geq 0.99 \cdot \binom{n-t+\ell-1}{\ell} \cdot d_{t-\ell},    
\end{equation}
where $\ell=\frac{t-|S|}{2}$. This part holds under the weaker assumption $|\m G| \geq \eta M_{n,n-t}$ for any $\eta>0$, and $0.99$ in the right hand side can be replaced by $1-o(1)$, provided $n>n_0(\eta)$. Step~2, which uses the lower bound obtained in Step~1 to reach a contradiction, requires the assumption $|\m G|>(1-\frac{1}{3} \cdot e^{-12})M_{n,n-t}$, and thus, we have to replace it by a refined argument. 

Specifically, in order to reach a contradiction to~\eqref{Eq:Aux-Wrappingup11} under the weaker assumption $|\m G| \geq (1-\frac{1}{e}+\eta)M_{n,n-t}$, and thus to complete the proof of the stability statement of Theorem~\ref{Thm:Main} it is sufficient to show that for any $\pi \notin \m E_S$,
\begin{equation}
\label{Eq:Aux-Wrappingup13}
|\m E_S \setminus \m U_{\pi}| \leq (1-\tfrac{1}{e}+o(1))M_{n,n-t}.
\end{equation}
Indeed, this will imply that if $\m G \not \subset \m E_S$, then $|\m G \cap \m E_S| \leq (1-\tfrac{1}{e}+o(1))M_{n,n-t} \leq (1-\frac{\eta}{2})|\m G|$, which contradicts~\eqref{Eq:Aux-Wrappingup11} for a sufficiently large $n$.

To prove~\eqref{Eq:Aux-Wrappingup13}, rather than the weaker bound proved in Step~2 of the proof of Lemma~\ref{lem:almost-maximal-contained-in-Ek}, we replace the comparison of $|\m U_{\pi}|$ with $|\m E_S|$ by a finer comparison with $\max\{|\m E_S|,|\m E_{S \cup \{i,j\}}|\}$ (for any $i,j \in [n] \setminus S$), and use the fact that $M_{n,n-t} \geq \max\{|\m E_S|,|\m E_{S \cup \{i,j\}}|\}$. 

We will need the following estimate, that will allow us to compute $|\m E_S|$ almost precisely. For $1 \leq k \leq n$, let 
\[
Q(n,k):=
|\{\sigma\in S_n:\sigma(i)\neq i
\text{ for all }i\in[k]\}|.
\]
\begin{lem}\label{lem:partial}
    For all $0 \leq k \leq n$, we have
    \[
    Q(n,k)=(1+o(1))e^{-k/n}n!.
    \]
\end{lem}
\begin{proof}[Proof of Lemma~\ref{lem:partial}]
    The assertion is essentially a special case of a classical result of Chatterjee, Diaconis and Meckes~\cite[Theorem~11]{CDM05}. The result of Chatterjee et al.~asserts that if $\sigma \in S_n$ is chosen uniformly,  $X_1,\ldots,X_n$ are the indicator random variables $X_i=\textbf{1}\{\sigma(i)=i\}$, and $Y_1,\ldots,Y_n$ are independent $\mathrm{Poisson}(\frac{1}{n})$ random variables, then the total variation distance between the distributions of $(X_1,\ldots,X_n)$ and $(Y_1,\ldots,Y_n)$ is at most $4/n$. Applying this bound to the event $A_k:=\{x_1=x_2=\ldots=x_k=0\}$, we get 
    \[
    |\Pr[X_1=X_2=\ldots=X_k=0]-\Pr[Y_1=Y_2=\ldots=Y_k=0]|\leq \tfrac{4}{n}.
    \]
    We have $\Pr[X_1=\ldots=X_k=0]=\frac{Q(n,k)}{n!}$ by the definition of $Q(n,k)$, and $\Pr[Y_1=\ldots=Y_k=0]=(e^{-1/n})^k=e^{-k/n}$ since $\Pr[Y_i=0]=e^{-1/n}$ for each $i$, and the $Y_i$'s are independent. Therefore, $|\frac{Q(n,k)}{n!}-e^{-k/n}|\leq \frac{4}{n}$. As $e^{-k/n} \geq e^{-1}$, the assertion follows.
\end{proof}

We will also need the following simple inequality.
\begin{claim}
    For any $0 < x \leq 1$ and any $y>0$, we have
\begin{equation}\label{eq:elementary-exp-inequality}
  1-e^{-1+x-y}
  \leq
  \left(1-\tfrac1e\right)
  \max\left\{1,\frac{y}{x}\right\}.
\end{equation}    
\end{claim}
\begin{proof}
If \(y\leq x\), then
\[
    1-e^{-1+x-y}\leq1-e^{-1}.
\]
If \(y>x\), put \(z:=y-x>0\), so we have $1-e^{-1+x-y}=1-e^{-1-z}$. We claim that for any $z \geq 0$, we have 
\begin{equation}\label{Eq:Aux-Wrappingup14}
    1-e^{-1-z}
       \leq
       \left(1-\frac1e\right)(1+z).
\end{equation}
Indeed, the function $h(z):=
       \left(1-\frac1e\right)(1+z)
       -(1-e^{-1-z})$ satisfies \(h(0)=0\) and $h'(z)>0$ for all $z \geq 0$, and thus, $h(z)\geq 0$ for all $z \geq 0$.

Since \(x\leq1\), we have $1+z\leq1+\frac{z}{x}=\frac{y}{x}$, and hence,~\eqref{Eq:Aux-Wrappingup14} implies
\[
1-e^{-1+x-y}=1-e^{-1-z} \leq \left(1-\frac1e\right)(1+z) \leq \left(1-\frac1e\right)\cdot \tfrac{y}{x}.
\]
Hence, in both cases we have $1-e^{-1+x-y} \leq \left(1-\frac1e\right)\cdot \max\{1,\tfrac{y}{x}\}$, as asserted.    
\end{proof}
Now we are ready to prove~\eqref{Eq:Aux-Wrappingup13}, which will complete the proof of the stability statement of Theorem~\ref{Thm:Main}.

\medskip \noindent \textbf{Proof of~\eqref{Eq:Aux-Wrappingup13}.} Denote 
\[
x:= \frac{\ell}{t-\ell}, \qquad \mbox{and} \qquad y:= \frac{\ell^2}{n-t}.
\]
We clearly have $x \in (0,1]$, and by~\eqref{Eq:Aux-Wrappingup9.3}, we have $0<y \leq 3$. The proof proceeds in three steps.

\medskip \noindent \emph{Step~1: Lower bounding $|\m U_{\pi}|/|\m E_S|$.} By a slight variation of the proof of Lemma~\ref{lemadecay}(vi), we have $|\m E'_S|=(1-o(1))|\m E_S|$. As $\m E'_S=\binom{n-t+2\ell}{\ell}Q(t-\ell,\ell)$, Lemma~\ref{lem:partial} (applied with $(t-\ell,\ell)$ in place of $(n,k)$) yields
\[
|\m E_S|=(1+o(1))|\m E_{S}'|=(1+o(1))\binom{n-t+2\ell}{\ell}\cdot e^{-x}(t-\ell)!.
\]
On the other hand, as $d_n=(1+o(1))e^{-1}n!$,~\eqref{Eq:Aux-Wrappingup12.5} (with $1-o(1)$ in place of $0.99$) implies
\[
|\m U_\pi| \geq (1-o(1))\binom{n-t+\ell-1}{\ell} \cdot e^{-1}(t-\ell)!.
\]
Hence, we have
\[
\frac{|\m U_{\pi}|}{|\m E_S|} \geq (1+o(1)) \frac{\binom{n-t+\ell-1}{\ell} \cdot e^{-1}(t-\ell)!}{\binom{n-t+2\ell}{\ell}\cdot e^{-x}(t-\ell)!}=(1+o(1))e^{-1+x}\cdot \frac{\binom{n-t+\ell-1}{\ell}}{\binom{n-t+2\ell}{\ell}}.
\]
As by~\eqref{Eq:Aux-Wrappingup9.3} and~\eqref{Eq:Aux-Wrappingup12.3}, 
\[
\frac{\binom{n-t+\ell-1}{\ell}}{\binom{n-t+2\ell}{\ell}} \geq \left( 1 - \frac{\ell+1}{n-t+\ell} \right)^\ell \geq (1+o(1))e^{-\frac{\ell^2}{n-t}}=(1+o(1))e^{-y},
\]
we obtain
\begin{equation}\label{Eq:Aux-Wrappingup15}
    \frac{|\m U_{\pi}|}{|\m E_S|} \geq (1+o(1))e^{-1+x-y}.
\end{equation}
\emph{Step~2: Estimating $|\m E_{S \cup \{i,j\}}|/|\m E_S|$.} Assume $\ell \geq 1$. By the same argument as in Step~1, for any $i,j \in [n] \setminus S$, we have
\[
|\m E_{S \cup \{i,j\}}|=(1+o(1))|\m E_{S \cup \{i,j\}}'|=(1+o(1))\binom{n-t+2(\ell-1)}{\ell-1}\cdot e^{-\frac{\ell-1}{t-(\ell-1)}}(t-\ell+1)!.
\]
Hence, 
\begin{align}\label{Eq:Aux-Wrappingup16}
    \begin{split}
\frac{|\m E_{S \cup \{i,j\}}|}{|\m E_S|} &=(1+o(1))\frac{\binom{n-t+2(\ell-1)}{\ell-1}\cdot e^{-\frac{\ell-1}{t-(\ell-1)}}(t-\ell+1)!}{\binom{n-t+2\ell}{\ell}\cdot e^{-\frac{\ell}{t-\ell}}(t-\ell)!} \\
&=(1+o(1))\frac{(t-\ell+1)\ell(n-t+\ell)}{(n-t+2\ell)(n-t+2\ell-1)}\\
&=(1+o(1))\frac{\ell(t-\ell)}{n-t}=(1+o(1))\frac{y}{x}.        
    \end{split}
\end{align}
\emph{Step~3: Comparing $|\m U_{\pi}|$ with $\max\{|\m E_{S \cup \{i,j\}}|,|\m E_S|\}$.} For $\ell \geq 1$,
by~\eqref{eq:elementary-exp-inequality} and~\eqref{Eq:Aux-Wrappingup15}, we have
\[
\frac{|\m E_S \setminus \m U_\pi|}{|\m E_S|} \leq 1- (1+o(1))e^{-1+x-y} \leq \left(1-\tfrac1e+o(1)\right)
  \max\left\{1,\frac{y}{x}\right\}.
\]
Therefore, by~\eqref{Eq:Aux-Wrappingup16}, we have
\begin{align*}
|\m E_S \setminus \m U_\pi| &\leq \left(1-\tfrac1e+o(1)\right)
  \max\left\{1,\frac{y}{x}\right\}|\m E_S| \leq \left(1-\tfrac1e+o(1)\right)\max\{|\m E_S|,|\m E_{S \cup \{i,j\}}|\} \\
  &\leq \left(1-\tfrac1e+o(1)\right)M_{n,n-t},
\end{align*}
as asserted in~\eqref{Eq:Aux-Wrappingup13}. 

For $\ell=0$, we have $|\m E_S|=t!$ and $|\m U_\pi|=d_t=(\frac{1}{e}+o(1))t!$. Hence, 
\[
|\m E_S \setminus \m U_\pi| \leq (1-\tfrac{1}{e}+o(1))t! \leq (1-\tfrac{1}{e}+o(1))M_{n,n-t},
\]
as asserted in~\eqref{Eq:Aux-Wrappingup13}. This completes the proof.
\end{proof}

%We prove the following stability consequence of the proof.

%\begin{thm}[Containment stability]\label{thm:containment-stability}
%There exist constants \(\eta>0\) and \(n_0\in\mathbb N\) such that the
%following holds for all \(n\ge n_0\) and all \(0\le t\le n\). Let
%\(\mathcal G\subset S_n\) be an \((n-t)\)-intersecting family, and put
%\[
%M_{n,t}
%:=
%\max_{0\le r\le \lfloor t/2\rfloor}
%|\mathcal F_{n,n-t,r}|.
%\]
%If
%\[
%|\mathcal G|>(1-\eta)M_{n,t},
%\]
%\ohad{We probably want $\eta = 1/e+o(1)$}
%then
%\[
%\mathcal G\subset \alpha\mathcal F_{n,n-t,r}\beta
%\]
%for some \(0\le r\le \lfloor t/2\rfloor\) and some
%\(\alpha,\beta\in S_n\).
%\end{thm}

%The proof has two ingredients.  The first is the weak stability conclusion already
%implicit in Sections 6--8.
%Theroem 18 in the manuscript
%Theorem 22
%Theorem 28
%\begin{thm}[Weak stability]\label{prop:weak-stability}
%For every \(\eta>0\) there exist $c_0$ and \(n_1=n_1(\eta)\) such that the
%following holds for all \(n\ge n_1\) and all \(3\le t\le n\). The conclusions of Theorems 18,22,28 hold when assuming  
%\(\mathcal G\subset S_n\) is \((n-t)\)-intersecting and
%\[
%|\mathcal G|\ge (1-1/e) \cdot  M_{n,t},
%\]
%Instead of $\mathcal G$ maximal $(n-t)$-intersecting. 
%\end{thm}
%\nathan{We probably want to say that for any $\eta>1/2$ this holds for all $t \geq 3$. and for any $\eta>0$, this holds for all $t>c(\eta)$, and to say in the proof that the only delicate place is in Section 6.}

%The second part is using the lemmas in section 10, which work with the requirement $|\mathcal G|>(1-\eta)M_{n,t}$.

\bibliographystyle{plain}
\bibliography{refs}

\end{document}